\documentclass[12pt]{amsart}
\usepackage[english]{babel}
\usepackage[utf8]{inputenc}
\usepackage[T1]{fontenc}
\usepackage[nomath]{lmodern}

\usepackage{amssymb,amsmath,amsthm,amsfonts}
\usepackage[thmtools-compat]{keytheorems}
\usepackage{mathtools}
\usepackage{enumitem} % item labels
\usepackage{comment}

\usepackage{xcolor}
\usepackage{graphicx} % for including images
\usepackage{tikz,tikz-cd,pgfplots}
\pgfplotsset{compat=newest}

\usepackage{caption} % for spacing between a table and its caption
\usepackage{hyperref}

\usepackage{listings}
\makeatletter
\lst@CCPutMacro
\lst@ProcessOther{"2D}{-}
\@empty\z@\@empty
\makeatother

\lstdefinelanguage{Macaulay2}{
comment=[l]{--},
alsoletter={'},
alsoother={_},
}
\addto\extrasenglish{}
\addto\extrasenglish{}
\addto\extrasenglish{\def\equationautorefname~#1\null{\textnormal{(#1)}\null}}
\declaretheorem[name=Theorem,refname={Theorem},style=plain,numberwithin=section]{theorem}
\declaretheorem[name=Theorem,refname={Theorem},style=plain,numbered=no]{theorem*}
\declaretheorem[name=Proposition,refname={Proposition},style=plain,sibling=theorem]{proposition}
\declaretheorem[name=Proposition,refname={Proposition},style=plain,numbered=no]{proposition*}
\declaretheorem[name=Lemma,refname={Lemma},style=plain,sibling=theorem]{lemma}
\declaretheorem[name=Lemma,refname={Lemma},style=plain,numbered=no]{lemma*}

\declaretheorem[name=Definition,refname={Definition},style=definition,numbered=no]{definition*}
\declaretheorem[name=Remark,refname={Remark},style=definition,sibling=theorem]{remark}
\declaretheorem[name=Remark,refname={Remark},style=remark,numbered=no]{remark*}

\declaretheorem[name=Example,refname={Example},style=definition,numbered=no]{example*}
\declaretheorem[name=Corollary,refname={Corollary},style=plain,sibling=theorem]{corollary}
\declaretheorem[name=Corollary,refname={Corollary},style=plain,numbered=no]{corollary*}

\def\bA{{\mathbf{A}}}

\def\bC{{\mathbf{C}}}

\def\bG{{\mathbf{G}}}

\def\bP{{\mathbf{P}}}
\def\bQ{{\mathbf{Q}}}

\def\bZ{{\mathbf{Z}}}

\def\cC{{\mathcal{C}}}
\def\cD{{\mathcal{D}}}
\def\cE{{\mathcal{E}}}

\def\cL{{\mathcal{L}}}
\def\cM{{\mathcal{M}}}
\def\cN{{\mathcal{N}}}
\def\cO{{\mathcal{O}}}
\def\cP{{\mathcal{P}}}
\def\cQ{{\mathcal{Q}}}

\def\cT{{\mathcal{T}}}
\def\cU{{\mathcal{U}}}

\def\cW{{\mathcal{W}}}

\def\m{\mathfrak m}
\def\p{\mathfrak p}
\def\q{\mathfrak q}

\def\Bl{\operatorname{Bl}}
\def\Br{\operatorname{Br}}
\def\ch{\operatorname{ch}}
\def\codim{\operatorname{codim}}
\def\coker{\operatorname{coker}}

\def\disc{\operatorname{disc}}

\def\div{\operatorname{div}}

\def\Flag{\operatorname{Flag}}
\def\GL{\operatorname{GL}}
\def\Gr{\operatorname{Gr}}

\def\Hom{\operatorname{Hom}}
\def\Id{\operatorname{Id}}
\def\Im{\operatorname{Im}}

\def\Mon{\operatorname{Mon}}
\def\NS{\operatorname{NS}}
\def\O{\operatorname{O}}

\def\Pf{\operatorname{Pf}}
\def\Pic{\operatorname{Pic}}
\def\pr{\operatorname{pr}}

\def\rank{\operatorname{rank}}
\def\Sing{\operatorname{Sing}}
\def\SL{\operatorname{SL}}

\def\Sym{\operatorname{Sym}}
\def\td{\operatorname{td}}
\def\Tot{\operatorname{Tot}}

\def\alg{\mathrm{alg}}
\def\et{\text{ét}}
\def\KKK{{\mathrm{K3}}}
\def\prim{{\mathrm{prim}}}
\def\smooth{{\mathrm{smooth}}}

\def\tors{{\mathrm{tors}}}
\def\trans{{\mathrm{trans}}}
\def\van{{\mathrm{van}}}

\let\ordexists\exists
\def\exists{\operatorname{\ordexists}}
\let\ordforall\forall
\def\forall{\operatorname{\ordforall}}

\def\inner#1{{\left<{#1}\right>}}
\def\set#1{{\left\{{#1}\right\}}}
\def\setmid#1#2{{\left\{{#1}\;\middle|\;{#2}\right\}}}

\def\tilde{\widetilde}
\def\setminus{\smallsetminus}
\def\emptyset{\varnothing}
\def\git{/\!\!/}

\def\loccit{{\it loc.\ cit.}}
\def\bw#1{\mathchoice%
 {\textstyle{\bigwedge\mkern-4.5mu^{#1}\mkern1mu}}%
 {\textstyle{\bigwedge\mkern-4.5mu^{#1}\mkern1mu}}%
 {\scriptstyle{\bigwedge\mkern-5mu^{#1}}}%
 {\scriptscriptstyle{\bigwedge\mkern-5mu^{#1}}}%
}

\def\longarrow#1#2{\mathchoice{#2}{#1}{#1}{#1}}
\def\to{\longarrow{\rightarrow}{\longrightarrow}}
\def\simto{\longarrow{\xrightarrow\sim}{\stackrel\sim\longrightarrow}}
\def\into{\longarrow{\hookrightarrow}{\lhook\joinrel\longrightarrow}}
\def\onto{\longarrow{\twoheadrightarrow}{\relbar\joinrel\twoheadrightarrow}}
\let\shortmapsto\mapsto
\def\mapsto{\longarrow{\shortmapsto}{\longmapsto}}

\def\Macaulay{{\tt Macaulay2} }

\usepackage[margin=1in]{geometry}
\begin{document}
% \frontmatter
\title{Some divisors in the moduli space of Debarre-Voisin varieties}

\author[V.~Benedetti]{Vladimiro Benedetti}
%\author[\initial{V.} \lastname{Benedetti}]{\firstname{Vladimiro} \lastname{Benedetti}}
\address{Universit\'e Côte d’Azur, CNRS – Laboratoire J.-A. Dieudonné, Parc Valrose, F-06108 Nice Cedex 2, France}
\email{\href{mailto:vladimiro.benedetti@univ-cotedazur.fr}{\tt vladimiro.benedetti@univ-cotedazur.fr}}

\author[J.~Song]{Jieao Song}
%\author[\initial{J.} \lastname{Song}]{\firstname{Jieao} \lastname{Song}}
\address{Dipartimento di Matematica ‘‘Federigo Enriques'', Universita` degli Studi di Milano
Via Cesare Saldini 50, 20133 Milano, Italy}
\email{\href{mailto:jieao.song@unimi.it}{\tt jieao.song@unimi.it}}

\date{\today}
%%%%%%%%%%%%%%%%%%%%%%%%%%%%%%%%%%%%%%%%%%%%%%%%%%%%%%%%%%%%%%%%%%%%%%%%%%%%%%%

\begin{abstract}
Let $V_{10}$ be a $10$-dimensional complex vector space and let
$\sigma\in\bigwedge^3V_{10}^\vee$ be a non-zero alternating~$3$-form. One can
define several associated degeneracy loci: the Debarre--Voisin hyperk\"ahler
variety~$X_6^\sigma\subset\mathrm{Gr}(6,V_{10})$, the Peskine
variety~$X_1^\sigma\subset\mathbf P(V_{10})$, and the hyperplane
section~$X_3^\sigma\subset\mathrm{Gr}(3,V_{10})$. %Their interest stems from the fact that the Debarre--Voisin varieties form a locally complete family of projective hyperkähler fourfolds of $\mathrm{K3}^{[2]}$-type.  
We prove that when smooth, the
varieties~$X_6^\sigma$,~$X_1^\sigma$, and~$X_3^\sigma$ share one common integral
Hodge structure, and that~$X_1^\sigma$ and~$X_3^\sigma$ both satisfy the
integral Hodge conjecture in all degrees.  This is obtained as a consequence of
a detailed analysis of the geometry of these varieties along three divisors in
the moduli space. On one of the divisors, an associated K3 surface $S$ of
degree $6$ can be constructed geometrically and the Debarre--Voisin fourfold is
shown to be isomorphic to a moduli space of twisted sheaves on $S$, in analogy
with the case of cubic fourfolds containing a plane.
\end{abstract}

\subjclass{14C30, 14J28, 14J42, 14J45}
\keywords{Hyperkähler fourfolds, trivectors, K3 Hodge structure}
% \altkeywords{Variétés hyperkählériennes, trivecteurs, structures de Hodge de type K3}
% \alttitle{Diviseurs dans l'espace de modules des variétés de Debarre--Voisin}
\begin{comment}
\begin{altabstract}
Soit $V_{10}$ un espace vectoriel complexe de dimension 10, et soit
$\sigma\in\bigwedge^3V_{10}^\vee$ un trivecteur non nul. On peut associer à
$\sigma$ les lieux de dégénérescence suivants : la variété de Debarre--Voisin
$X_6^\sigma\subset\mathrm{Gr}(6,V_{10})$, la variété de Peskine
$X_1^\sigma\subset\mathbf P(V_{10})$ et la section hyperplane
$X_3^\sigma\subset \mathrm{Gr}(3,V_{10})$. Leur intérêt provient du fait que les
variétés de Debarre--Voisin forment une famille localement complète de variétés
hyperkählériennes de type $\mathrm{K3}^{[2]}$. Nous montrons que les variétés
$X_6^\sigma$, $X_1^\sigma$ et $X_3^\sigma$ partagent une même structure de
Hodge entière lorsqu'elles sont lisses, et les deux dernières satisfont la
conjecture de Hodge en tous degrés. Les résultats sont obtenus par une analyse
détaillée de la géométrie de ces variétés le long de trois diviseurs dans
l'espace de modules. Sur l'un des diviseurs, une surface K3 associée $S$ de
degré 6 peut être construite géométriquement, et nous montrons que la variété
de Debarre--Voisin est isomorphe à un espace de modules de faisceaux tordus sur
$S$, en analogie avec le cas d'une cubique de dimension 4 contenant un plan.
\end{altabstract}
\end{comment}

\maketitle
\tableofcontents
%\mainmatter

\section{Introduction}

Hyperk\"ahler geometry has gained in the last years a lot of attention in complex algebraic geometry. From the one side, the hyperk\"ahler condition is quite restrictive, as it is shown by the general lack of (explicit and abstract) examples of hyperk\"ahler manifolds. This has lead to an ongoing attempt of classifying these varieties modulo complex deformations (see for instance \cite{comp_rr}) and to study their moduli space. 

From the other hand, this rigidity has been used to develop powerful tools to study their geometry (e.g. variation of Hodge structures, theory of fibrations, derived categories and Bridgeland stability, \emph{etc}). One interesting feature is that these tools, instead of remaining bounded to the hyperk\"ahler domain, allow to study other classes of algebraic varieties. The most striking example of this is the connection between some special classes of Fano varieties, called \emph{Fano varieties of K3 type}, and hyperk\"ahler varieties (see \cite{fmo}). In order to uncover what constitutes this relationship it is useful to study very explicit examples. The defining example in this context is the one of cubic fourfolds and their varieties of lines; the former are Fano varieties of K3 type, while the latter are hypark\"ahler fourfolds of $K3^{[2]}$ type. These varieties are connected from various points of view (\cite{bd}, \cite{hassett} and \cite{kuz}). A lot of work has focused on the construction of special divisors in the moduli space in order to attack open problems such as rationality of cubics (see \cite{hassett}) and the integral Hodge conjecture.

In this paper we study another family of hyperk\"ahler varieties for which the moduli space has been explicitly described, namely Debarre-Voisin varieties (\cite{dv}). The aim is twofold. Firstly, we want to describe their moduli space and construct some divisors in it which reproduce the \emph{easier} case of cubic fourfolds. Secondly, we want to show that Debarre-Voisin varieties are related to \emph{many} different Fano varieties of K3 type; we will show that this relationship holds at a geometrical and cohomological level, and we will use it to show the integral Hodge conjecture for some of these varieties.
\medskip

In \cite{dv} Debarre and Voisin start from $V_{10}$ a $10$-dimensional complex vector space and
$\sigma\in\bw3V_{10}^\vee$ a trivector, that is, an alternating $3$-form
on $V_{10}$. They consider the Grassmannian $\Gr(6,V_{10})$ and the tautological subbundle $\cU_6$; By viewing~$\sigma$ as a section of the vector bundle $\bw3\cU^\vee_6$, they showed that the zero locus $X_6^\sigma\subset\Gr(6,V_{10})$ of~$\sigma$
is smooth of dimension $4$ for a general~$\sigma$. Moreover, by varying $\sigma$, these fourfolds
form a locally complete family of projective hyperkähler varieties of
deformation type $\KKK^{[2]}$.

Along with $X^\sigma_6$, there are several other loci determined by
the trivector~$\sigma$ that have interesting Hodge theoretical and categorical
properties. We will denote by $X^\sigma_k$
a subvariety defined in the Grassmannian $\Gr(k,V_{10})$ as follows.
%(see \autoref{sec:degeneracy-loci} for more precise definitions).
\begin{itemize}
\item The variety $X^\sigma_3$ is the hyperplane section of $\Gr(3,V_{10})$
defined by $\sigma$, i.e.
\[
X^\sigma_3\coloneqq\setmid{[V_3]\in\Gr(3,V_{10})}
{\sigma|_{V_3}=0}.
\] 
It has an interesting subcategory in its derived category,
which is a K3 category (see~\cite{msK3}) and has been studied already in \cite{dv}.
\item The variety $X^\sigma_1$, a $6$-dimensional
degeneracy locus in $\bP(V_{10})$, also known as the Peskine variety, which was studied in the context of projective normality (\cite{han1} and \cite{han2}). It is
defined as
\[
X^\sigma_1\coloneqq\setmid{[V_1]\in\bP(V_{10})} {\rank\sigma(V_1, -, -)
\le 6}.
\] 
%The Fano variety of lines of $X^\sigma_1$ is related to another degeneracy locus $X^\sigma_7$ in $\Gr(7,V_{10})$ that we will define later on as \[X^\sigma_7\coloneqq\setmid{[V_7]\in\Gr(7,V_{10})}{\rank(\sigma|_{V_7})\le5}.\]
\item The variety $X^\sigma_2$ inside $\Gr(2,V_{10})$ which was studied in \cite{bfm} and is defined as
\[
X^\sigma_2\coloneqq\setmid{[V_2]\in\Gr(2,V_{10})} {\sigma(V_2, V_2,
-)=0}.
\]
\end{itemize}
%First, we have obtained the following criteria for the smoothness of $X_1^\sigma$, $X_3^\sigma$, and $X_6^\sigma$.
%\begin{proposition} Consider the following two conditions on a trivector $\sigma$ \begin{enumerate} \item for all $[V_3]\in \Gr(3,V_{10})$, $\sigma(V_3,V_3,-)\neq 0$; \item for all $[V_1]\in \bP(V_{10})$, $ \rank \sigma(V_1,-,-)\geq 6$. \end{enumerate}  Then $X_3^\sigma$ and $X_6^\sigma$ are smooth of expected dimension if and only if condition (1) holds, and $X_1^\sigma$ is smooth of expected dimension 6 if and only if conditions (1) and (2) both hold. \end{proposition}

For what concerns the moduli space of trivectors, there is a $20$-dimensional irreducible
quasi-projective GIT moduli space
\[
\cM\coloneqq \bP\big(\bw3V_{10}^\vee\big)\git \SL(V_{10})
\]
for the trivectors $\sigma$. Furthermore, for what concerns Debarre-Voisin varieties, there exists a $20$-dimensional
irreducible quasi-projective moduli space $\cM^{(2)}_{22}$; the superscript $(2)$ and the subscript $22$ refer to the fact that Debarre--Voisin
varieties are endowed with a natural polarisation $H\coloneqq \cO_{\Gr(6,V_{10})}(1)|_{X_6^\sigma}$ of square $22$ and
divisibility~$2$, that is $\q(H,H)=22$ and
$\q\big(H,H^2(X^\sigma_6,\bZ)\big)=2\bZ$ (see for instance \cite{dv} or \cite{ogrady}). Here and later on by $\q$ we denote the Beauville-Bogomolov-Fujiki quadratic form on $H^2(X_6^\sigma,\bZ)$. The Torelli theorem for polarized hyperkähler manifolds tells us
that the period map
\[
\p\colon\cM^{(2)}_{22}\into\cP
\]
is an open immersion into the quasi-projective period domain $\cP$ that
parametrizes the corresponding Hodge structures. The Debarre--Voisin
construction gives a rational map
\[\begin{tikzcd}
\m\colon\cM\rar[dashed]&\cM_{22}^{(2)},
\end{tikzcd}\]
which was recently proved by O'Grady to be birational~\cite{ogrady}. One
can therefore study divisors in $\cM$ coming from $\SL(V_{10})$-invariant divisors
in $\bP\big(\bw3V_{10}^\vee\big)$, Noether--Lefschetz divisors in
$\cM^{(2)}_{22}$, and Heegner divisors in $\cP$. In this paper, we will denote Heegner divisors by~$\cD_{2e}$, where the label $2e$ is the discriminant of the corresponding algebraic lattice. 

Debarre and Voisin in the original article \cite{dv} already studied one of these divisors, namely $\cD_{22}$. This divisor corresponds to the locus where $X_6^\sigma$ becomes singular. Although the variety $X^\sigma_6$ is not smooth in this case, its singular
locus contains (and we will show that it coincides with) a K3 surface $S_{22}$ of
degree~$22$. It was proved in~\cite{dv} that $X^\sigma_6$ is birational to
the Hilbert square $S_{22}^{[2]}$, and this was essentially how $X_6^\sigma$ was proved to be of $K3^{[2]}$ deformation type.
\medskip

\noindent {\bf Hodge structures and Hodge conjecture} Inspired by the case of cubic fourfolds (\cite{bd}), all the varieties $X_k^\sigma$, $k=1,2,3,6$, associated with $\sigma$ are expected to have some common
Hodge structures. One particularly interesting one is the second integral cohomology group of the
hyperkähler fourfold $X^\sigma_6$, which carries the Beauville--Bogomolov--Fujiki
quadratic form $\q$. This provides a polarized Hodge structure on the
primitive part $H^2(X^\sigma_6,\bZ)_\prim$.
On the two Fano varieties~$X^\sigma_3$ and $X^\sigma_1$, we can consider the middle degree \emph{vanishing cohomologies}  which are generated by cohomology classes not coming from the ambient space. Both Hodge structures are polarized with cup product as the polarization. The first goal of this paper is to prove the following theorem, which provides a precise relationship between the integral Hodge structures of $X_1^\sigma$, $X_3^\sigma$ and $X_6^\sigma$.
% We will denote by the subscript \emph{van} the vanishing cohomology groups and by the subscript \emph{prim} the primitive cohomology groups; $\q$ will denote the {\em Beauville--Bogomolov--Fujiki form}. We will prove the following result.
\begin{theorem}[see \autoref{thm:36} and \autoref{thm:16}]
\label{thm:main_intro}
We have Hodge isometries
\[
\big(H^{20}(X_3^\sigma,\bZ)_\van,\,\cdot\,\big)\simeq
\big(H^2(X^\sigma_6,\bZ)_\prim,-\q\big)\simeq
\big(H^6(X^\sigma_1,\bZ)_\van,\,\cdot\,\big)
\]
given by algebraic correspondences between $X^\sigma_3$ and $X^\sigma_6$,
and between $X^\sigma_6$ and~$X^\sigma_1$, whenever they are smooth of expected
dimension. Moreover, $X_1^\sigma$ and $X_6^\sigma$ are smooth of expected dimension if and only if for all $[V_3]\in \Gr(3,V_{10})$, $\sigma(V_3,V_3,-)\neq 0$; $X_1^\sigma$ is smooth of expected dimension if and only if $X_6^\sigma$ is smooth and for all $[V_1]\in \bP(V_{10})$, $\rank(\sigma(V_1,-,-))\geq 6$.
\end{theorem}
The isometry between the integral Hodge structures of $X^\sigma_3$ and
$X^\sigma_1$ was already established using a different method by
Bernardara--Fatighenti--Manivel~\cite{bfm}.  Our method uses the fact that the above isomorphisms behave well in families; we can thus show that they hold for special points of the moduli space, and then deduce that they hold on the whole moduli space. More precisely, we will show the above isomorphisms to hold for varieties belonging to a special divisor in the moduli space
$\cM$. This divisor in $\cM$ is the one that maps onto the Heegner divisor $\cD_{28}\subset \cP$; it corresponds to the divisor of Debarre-Voisin varieties containing a plane $P\cong \bP^2\subset X_6^\sigma$ (which is a divisorial condition in $\cM_{22}^{(2)}$). By assuming that $X_6^\sigma$ is contained in such a divisor, we gain the existence of an extra algebraic class $[P]$ in its cohomology. Moreover, by using the incidence variety $I_{1,6}^\sigma$ between $X_1^\sigma$ and $X_6^\sigma$, and the incidence variety $I_{3,6}^\sigma$ between $X_3^\sigma$ and $X_6^\sigma$ we can construct correspondences between these three varieties. Through these correspondences, we are then able to ``transport'' the extra class $[P]$ from the cohomology of $X_6^\sigma$ to the cohomology of $X_1^\sigma$ and $X_3^\sigma$ and to compare $\q([P],[P])$ with the corresponding intersection product in $X_1^\sigma$ and $X_3^\sigma$.\medskip

The theorem above  closely resembles the correspondence between a cubic fourfold and its variety of lines. In~\cite{MO}, Mongardi--Ottem proved the integral Hodge conjecture for 1-cycles on hyperkähler manifolds of $\KKK^{[n]}$-type, and used this to deduce the integral Hodge conjecture for 2-cycles on cubic fourfolds. Following the same idea, we can prove the integral Hodge conjecture for
$X_3^\sigma$ and $X^\sigma_1$. Let $Z$ be a smooth complex projective variety of dimension $n$. We say that
the integral Hodge conjecture holds in degree $d$ on $Z$ if $H^{d,d}(Z, \bZ)$ is generated by
classes of $(n-d)$-dimensional algebraic subvarieties of $Z$. Here $H^{d,d}(Z, \bZ)$ denotes the intersection $H^{2d}(Z, \bZ) \cap H^{d,d}(Z, \bC)$.
% using the relevant result for $1$-cycles on hyperkähler manifolds of $\KKK^{[n]}$-type developed by Mongardi--Ottem in~\cite{MO}.
\begin{corollary}[see \autoref{thm:HodgeConjX3} and \autoref{thm:HodgeConjX1}]
The integral Hodge conjecture holds for both $X^\sigma_3$ and $X^\sigma_1$ in
all degrees, whenever they are smooth of expected dimension.
\end{corollary}

Notice that the integral Hodge conjecture is known not to be true in general while the rational Hodge conjecture is still open. The validity of the integral Hodge conjecture seems to depend on birational properties of the varieties in question (see for instance \cite{v4} as an example of a class of varieties for which it holds). By using Lefschetz Hyperplane Theorem, one can reduce the proof of the corollary above to an analogous statement only concerning the vanishing, i.e. non-ambient, cohomology of $X_1^\sigma$ and $X_3^\sigma$. Then one notices that the vanishing cohomology is in fact only constituted by copies of the second cohomology $H^2(X_6^\sigma,\bZ)$ of the associated Debarre-Voisin variety thanks to Theorem \ref{thm:main_intro} above. Since the correspondence between $X_1^\sigma$ and $X_6^\sigma$, and between $X_3^\sigma$ and $X_6^\sigma$, is geometric (i.e. given by incidence varieties $I_{1,6}^\sigma$ and $I_{3,6}^\sigma$), one can apply the results in \cite{MO} to deduce the integral Hodge conjecture for $X_1^\sigma$ and $X_3^\sigma$.

\medskip

\noindent {\bf The Heegner divisor $\cD_{24}$ and twisted K3 surfaces} The final section of this paper concerns the study of the Heegner divisor $\cD_{24}\subset \cP$. The reason why we focus on this divisor is that it is the analogous of the divisor of cubic fourfolds containing a plane (see Appendix \ref{app_cubic_fourfolds_containing} for a review on the subject). This divisor is especially important because it partially inspired Hassett's conjecture on rationality of cubics (see \cite{hassett}). A very general cubic fourfold containing a plane is isomorphic to a moduli space of \emph{twisted} sheaves on a K3 surface; Hassett showed that when the twisting is trivial (which occurs on a countable union of divisors), the cubic fourfold is rational. He then conjectured that a cubic fourfold is rational exactly when the associated hypark\"ahler variety is birational to a non-twisted moduli space of sheaves on a K3 surface.

Similarly, we will show that for a very general Debarre-Voisin variety in the divisor $\cD_{24}$, one can construct a Brauer-twisted K3 surface $(S_6,\beta)$. Then, by using a strategy analogous to the one used in \cite{voisinthesis}, we will show that the Hodge structure of $(S_6,\beta)$ embeds in the one of
$X_6^\sigma$. Moreover, we will be able to recover $X^\sigma_6$ as a moduli space
of~$\beta$-twisted sheaves:
\begin{theorem}[see Theorem \ref{thmmodulitwistedsheaves}]
\label{thm:D24_intro}
A very general Debarre--Voisin fourfold $X^\sigma_6$ whose periods is in $\cD_{24}$
is isomorphic to a moduli space of twisted sheaves on the twisted K3 surface $(S,\beta)$.
\end{theorem}
Notice that both $\cD_{24}$ and the divisor of cubic fourfolds containing a plane provide
examples of Brill--Noether contractions with non-trivial Brauer class
on hyperkähler fourfolds, and their general theory has been thoroughly studied
in the recent~\cite{KapuvanGeemen}. In their notations, the Heegner
divisor $\cD_8$ for cubic fourfolds and the divisor $\cD_{24}$ for
Debarre--Voisin varieties are denoted by $\cD^{(1)}_{8k,8k,\beta}$ for $k=1$
and $k=3$ respectively.

Finally, we mention the recent results of G.~Oberdieck~\cite{oberdieck}
regarding invariants of a generic pencil of Debarre--Voisin
varieties, obtained using Gromov--Witten techniques and modular forms.
The three Heegner divisors $\cD_{22},\cD_{24}$, and $\cD_{28}$ that we
study are precisely the first three non-HLS divisors with lowest
discriminants (see \autoref{sec:modulispace} for the definition of an HLS divisor), and the corresponding Noether--Lefschetz numbers indeed coincide
with the degrees of the $\SL(V_{10})$-invariant hypersurfaces that we have
computed.
\medskip

Let us come back to the results contained in the present paper. In Table \ref{table:intro}, we sum up the properties of the Heegner
divisors $\cD_{22}$, $\cD_{24}$, and $\cD_{28}$. 

\begin{table}[ht]
\centering
\scalebox{.8}{
\def\arraystretch{1.5}
\begin{tabular}{|r|c|c|c|}
\hline
Heegner divisor & $\cD_{22}$ & $\cD_{24}$ & $\cD_{28}$ \\
\hline
$(\p\circ \m)^{-1}(\cD_{\bullet})$ & $\cD^{3,3,10}$ & $\cD^{1,6,10}$ & $\cD^{4,7,7}$ \\
degree in $\bP\big(\bw3V_{10}^\vee\big)$ & $640$ & $990$ & $5500$\\
distinguished flag in $V_{10}$  & $V_3$ & $V_1\subset V_6$ & $V_4\subset V_7$ \\
degeneracy condition on $\sigma$  & $\sigma(V_3,V_3,V_{10})=0$ & $\sigma(V_1,V_6,V_{10})=0$ & $\sigma(V_4,V_7,V_7)=0$ \\
\hline
singular locus of $X_1^\sigma$  & $\bP(V_3)$ & $\{[V_1]\}$ & $\emptyset$ \\
singular locus of $X_3^\sigma$  & $\{[V_3]\}$ & $\emptyset$ & $\emptyset$ \\
singular locus of $X_6^\sigma$  & $S_{22}$ & $\emptyset$ & $\emptyset$ \\
%
% Singular locus of $X_7^\sigma$  & (unknown) & (unknown) & $\supseteq \{[V_7]\}$ \\
\hline
degree of associated K3  & $22$ & $6$ with a Brauer class of order $2$ & none \\
birational model of $X_6^\sigma$  & $S_{22}^{[2]}$ & $M(S_6,v,\beta)$ & - \\
\hline
\end{tabular}
}
\caption{\label{table:intro}Divisors in the moduli spaces}
\end{table}

Let us explain the contents of the table and where to find the corresponding
results. The preimage via $\p \circ \m$ of each of the three Heegner divisors is given as the set of
trivectors~$\sigma$ satisfying a vanishing condition with respect to a certain
flag of subspaces of $V_{10}$. The Heegner divisors are defined by subscripts ($22$, $24$ and $28$) which correspond to their discriminants in the period domain $\cP$ (see
\autoref{sec:modulispace} for details).
In \autoref{propsingD33} and \autoref{propsingD16}, we study the
singular loci of $X_1^\sigma$, $X_3^\sigma$, and $X_6^\sigma$, and identify
the divisors $\cD_{22}$ and $\cD_{24}$ as the loci where these varieties turn
singular.
For a general member of the divisor $\cD_{24}$, we define a twisted
associated K3 surface $(S_6,\beta)$ of degree $6$ in \autoref{propK3S},
and we prove in \autoref{thmmodulitwistedsheaves} that the hyperkähler variety
$X_6^\sigma$ is isomorphic to a moduli space of twisted sheaves on
$(S_6,\beta)$ with a certain Mukai vector $v$. Notice that since $\cD^{3,3,10},\cD^{1,6,10}$, and
$\cD^{4,7,7}$ are unirational, the Heegner divisors $\cD_{22},\cD_{24}$, and
$\cD_{28}$ are unirational as well (see Proposition~\ref{unirat_div}).
\medskip

Finally, let us give an overview of the paper. In Section \ref{sec_2} we recall the definition of the varieties $X_k^\sigma$ for various $k$ and basic facts about moduli spaces $\cM$, $\cM_{22}^{(2)}$ and $\cP$. In Section \ref{sec_3} we study the divisor $\cD_{22}$ and we construct and caracterize the incidence varieties $I_{1,6}^\sigma$ and $I_{3,6^\sigma}$; we also study the smoothness of $X_3^\sigma$ and $X_1^\sigma$. In Section \ref{sec:D28} we study the Heegner divisor $\cD_{28}$, the integral Hodge structures of $X_3^\sigma$ and $X_1^\sigma$ and we show Theorem \ref{thm:main_intro} as well as its corollary concerning the integral Hodge conjecture. Finally, in Section \ref{sec:D24}, we study the Heegner divisor $\cD_{24}$ and we show Theorem \ref{thm:D24_intro}. In the Appendix we prove some technical results; moreover we explicitly describe the Fano variety of lines of $X_1^\sigma$.
\medskip

\noindent {\bf Notation.} {\em Grassmannians.}
Throughout the paper, we will denote by $U_n$, $V_n$,
or $W_n$ an $n$-dimensional complex vector space. We denote by
$\Flag(k_1,\dotsc,k_r,V_n)$ the flag variety parametrizing nested subspaces of
$V_n$ of dimensions $k_1,\dotsc,k_r$. We will denote by $\cU_{k_i}$ the
tautological vector subbundle of $V_n\otimes \cO_{\Flag(k_1,\dotsc,k_r,V_n)}$
of rank $k_i$. When $r=1$, we recover the ordinary Grassmannian, which we
denote by $\Gr(k,V_n)$ (or $\bP(V_n)$ if $k=1$); when no confusion arises,
$\cU,\cQ$ will denote respectively the tautological and the quotient vector
bundles on $\Gr(k,V_n)$.

For a trivector $\sigma\in\bw3V_n$, its {\em rank} is defined as the dimension
of the smallest subspace $V\subset V_n$ such that $\sigma\in\bw3V$. The rank is a
$\GL(V_n)$-invariant.
If $\sigma\in\bw3 V_n^\vee$ and $V_i\subset V_n$, we will denote by
$\sigma(V_i,V_i,V_i)$ the restriction $\sigma|_{V_i}\in \bw3 V_i^\vee$.
Similarly, if $V_i\subset V_j\subset V_k\subset V_n$, we use
$\sigma(V_i,V_j,V_k)$ to denote the image of $\sigma$ in $(V_i\wedge V_j\wedge
V_k)^\vee$ (seen as a quotient of $\bw3 V_k^\vee)$.

The notation for Schubert varieties inside a Grassmannian
$\Gr(k,V_n)$ is as follows. Let us fix a complete flag $0=V_0\subset V_1\subset
\cdots \subset V_n$. For any sequence of integers $\lambda= (\lambda_1\geq
\cdots \geq \lambda_{k})$ with $\lambda_1\leq n-k$  and $\lambda_k\geq0$, we
define the Schubert variety
\[
\Sigma_\lambda=\setmid{W\in \Gr(k,V_n) }{ \dim(W\cap V_{n-k-\lambda_j+j})\geq j
\text{ for } 1\leq j\leq k},
\]
which is of codimension $\sum_i\lambda_i$ inside $\Gr(k,V_n)$. We let
$\sigma_{\lambda}$ be the Schubert class representing $\Sigma_\lambda$ in
cohomology.

\bigskip

\noindent {\em Lattices.} By a lattice we shall mean a finitely
generated free $\bZ$-module $L$ endowed with an integral quadratic
form $q$. The following basic properties can be found in~\cite[Chapter~I.2]{BHPV}. 

The {\em discriminant group} $D(L)$ of $L$ is defined as $L^\vee/L$, where
$L^\vee\coloneqq\Hom_\bZ(L,\bZ)$ is the dual.
If $M$ denotes the Gram matrix of $q$ in an integral basis of $L$, we call $\det(M)$ the {\em discriminant} of $L$ and denote it by $\disc(L)$. We have $|D(L)|=|\disc(L)|$. A
lattice $L$ is called {\em even} if $q(x)\in 2\bZ$ for all $x\in L$, and {\em
unimodular} if $D(L)$ is trivial. Beware that the discriminant of a unimodular lattice can either be 1 or $-1$. For a sublattice~$A\subset L$ with $\rank(L)=\rank(A)$, its
index $[L:A]$ is finite and satisfies $[L:A]^2=|\disc(A)/\disc(L)|$. 

A sublattice $A \subset L$ is called {\em saturated} if $L/A$ is torsion-free.
For any sublattice $A\subset L$, one can define the orthogonal
sublattice $A^\perp \subset L$ with respect to $q$.
If~$L$ is unimodular and~$A$ is a saturated sublattice, $A$ and
$A^\perp$ have isomorphic discriminant groups, hence the same discriminant up to sign
$|\disc(A)|=|\disc(A^\perp)|$. In this case, the direct sum $A\oplus A^\perp$ has
index $|\disc(A)|$ in~$L$.
\medskip

\noindent {\bf Acknowledgements}
We would like to thank Frédéric Han and Claire Voisin for interesting
discussions, and Lie Fu, Alexander Kuznetsov, and Georg Oberdieck for useful
comments. We would like moreover to express our special thanks to Olivier
Debarre, for reading thoroughly our paper and suggesting many and substantial
improvements. We also thank Grzegorz Kapustka and Bert van Geemen for pointing
out to us the relation between $\cD_{24}$ and their work in \cite{KapuvanGeemen}. We finally thank the two anonymous referees who very carefully proofread this paper and suggested many necessary and useful improvements.

The first author was partially supported by the EIPHI Graduate School (contract ANR-17-EURE-0002).

\section{General results for Debarre--Voisin varieties}
\label{sec_2}

In this section, we prove some general results on Debarre--Voisin varieties
and on the various degeneracy loci that we will study. We focus on properties
that hold for general members of the moduli space instead of those belonging
to special divisors, which will be the focus of later sections.

\subsection{Degeneracy loci for a trivector}
\label{sec:degeneracy-loci}

Let us begin by recalling general facts about the orbital degeneracy locus construction. Let~$G$ be an algebraic group and let~$V$ be a~$G$-module. Given a
principal~$G$-bundle over an algebraic variety~$X$, there is an associated
vector bundle~$E$ with fiber~$V$. Let~$Y$ be a~$G$-stable subvariety of~$V$.
For a global section~$s$ of~$E$, we will denote by
\[
D_Y(s)=\setmid{ x\in X}{s(x)\in Y\subset V\simeq E_x}
\]
the $Y$-degeneracy locus associated with $s$ as defined in~\cite[Definition 2.1]{benedetti}.
We will use basic properties of these loci that can be found
in~\cite{benedetti} and~\cite{ben2}. In particular, a Bertini-type theorem
holds for orbital degeneracy loci (\cite[Proposition 2.3]{benedetti}): if $E$ is globally generated, $s$ is
general, and $D_Y(s)\neq \emptyset$, then $\codim_{X}D_Y(s)=\codim_{V}Y$
and $\Sing D_Y(s)=D_{\Sing Y}(s)$. More generally, for arbitrary~$s$, we have
$\codim_{X}D_Y(s)\le \codim_{V}Y$, and when equality holds,
$D_{\Sing Y}(s)\subset \Sing D_Y(s)$. We will refer to $\codim_{V}Y$
(resp.\ $\dim X-\codim_{V}Y$) as the expected codimension (resp.\ expected
dimension) of the orbital degeneracy locus~$D_Y(s)$. 

Let us give two examples of orbital degeneracy loci. The first one is somehow trivial: when $Y=\{0\}\subset V$, $D_Y(s)$ is just the zero locus of the section $s$; since $Y=\{0\}$ is smooth, one recovers from \cite[Proposition 2.3]{benedetti} the classical Bertini theorem for zero loci of general sections of vector bundles. The second example we give is that of skew--symmetric matrices. Indeed, one can consider the set $Y_{2r}\coloneqq \{M\in \bw2 \bC^f \mid \rank(M)\leq 2r\}\subset V\coloneqq \bw2 \bC^f$ for $2r<f-1$ (here we fixed $G=\SL_e$). Let $F$ be a rank-$f$ vector bundle on $X$ and denote by $E=\bw2 F$. Let $s$ be a global section of the vector bundle $E$; then, for any point $x\in X$, the fiber $E_x$ is isomorphic to $\bw2 \bC^f$ and it makes sense to consider the locus
\[
D_{Y_{2r}}(s)=\setmid{ x\in X}{\rank(s(x))\leq 2r, \mbox{ where }s(x)\in E_x\cong\bw2 \bC^f}.
\]
If $E$ is globally generated, $s$ is a general section and $D_{Y_{2r}}(s)\neq \emptyset$, then $\codim_{X}D_{Y_{2r}}(s)=(f^2-f-4rf+4r^2+2r)/2$ and $\Sing D_{Y_{2r}}(s)=D_{Y_{2r-2}}(s)$.

\medskip

Let us come back to Debarre--Voisin varieties, and thus let $\sigma\in\bw3V_{10}^\vee$ be a trivector. We introduce the various
degeneracy loci that we will study. By the Borel--Weil theorem, the space $\bw3 V_{10}^\vee$ can be identified with
the space of global sections of certain vector bundles on the Grassmannians
$\Gr(k,V_{10})$ for $k\in\{1,\dotsc,9\}$.  More precisely, we have the
following:
\begin{itemize}
\item When $k=1$, we have $\bw3 V_{10}^\vee\simeq \Gamma\big(\bP(V_{10}),
(\bw2\cQ^\vee)(1)\big)$ where $\cQ^\vee\cong \Omega^1_{\bP(V_{10})}(1)$, and the evaluation of $\sigma$ at a point $[V_1]\in
\bP(V_{10})$ is given by $\sigma(V_1,-,-)\in V_1^\vee\otimes
\bw2(V_{10}/V_1)^\vee$;
\item When $k=2$, we have $\bw3 V_{10}^\vee\simeq \Gamma\big(\Gr(2,V_{10}),
\cQ^\vee(1)\big)$, and the evaluation of $\sigma$ at a point $[V_2]\in
\Gr(2,V_{10})$ is given by $\sigma(V_2,V_2,-)\in (\bw2 V_2^\vee)\otimes
(V_{10}/V_2)^\vee$;

\item For $3\le k\le 9$, we have $\bw3 V_{10}^\vee\simeq
\Gamma\big(\Gr(k,V_{10}), \bw3 \cU^\vee\big)$, and the evaluation of $\sigma$
at a point $[V_k]\in \Gr(k,V_{10})$ is given by $\sigma(V_k,V_k,V_k)\in \bw3
V_k^\vee$.
\end{itemize}
We introduce the following degeneracy loci, all of which have already been
studied in the literature:
\begin{itemize}
\item $X^\sigma_1\coloneqq\setmid{[V_1]\in\bP(V_{10})} {\rank\sigma(V_1, -, -)
\le6}$. The skew-symmetric form $\sigma(V_1, -, -)$ always has rank less than
or equal to $8$, and $X^\sigma_1$ is the locus where $\sigma$ degenerates
further. Its expected dimension is $6$ and it is known as the {\em Peskine
variety} (this variety has been studied, for instance, in \cite{han1}).

\item $X^\sigma_2\coloneqq\setmid{[V_2]\in\Gr(2,V_{10})} {\sigma(V_2, V_2,
-)=0}$. This locus has expected dimension $8$ and was thoroughly studied
in~\cite{bfm}.

\item $X^\sigma_3\coloneqq\setmid{[V_3]\in\Gr(3,V_{10})}
{\sigma|_{V_3}=0}$. This is a Plücker hyperplane section of $\Gr(3,V_{10})$ and
is irreducible of dimension $20$ whenever $\sigma\neq 0$.

\item $X^\sigma_6\coloneqq\setmid{[V_6]\in\Gr(6,V_{10})} {\sigma|_{V_6}=0}$.
This is known as the {\em Debarre--Voisin variety} and its expected dimension
is $4$.
\end{itemize}
For general $\sigma$, the variety $X_1^\sigma$ can also be realized as a subvariety of
$\Flag(1,4,V_{10})$ by mapping each point $[V_1]$ to the flag $[V_1\subset
K_4]$, where $K_4$ is the $4$-dimensional kernel of the skew-symmetric form
$\sigma(V_1,-,-)$. For general $\sigma$, we get an isomorphism onto the image.
This alternative construction will be useful in several places.

We remark that for $k>6$, the vanishing locus in $\Gr(k,V_{10})$
has negative expected dimension and is therefore empty for general
$\sigma$. Nevertheless, for $k=7$, we introduce the orbital degeneracy locus
\[X^\sigma_7\coloneqq\setmid{[V_7]\in\Gr(7,V_{10})}{\rank(\sigma|_{V_7})\le5}.\]
In other words, we consider those $[V_7]$ such that the restriction of~$\sigma$
on~$V_7$ has rank~$\le 5$, or equivalently, there exists $V_2\subset V_7$ such
that~$\sigma(V_2,V_7,V_7)=0$. We get a subvariety of expected dimension $6$ in
this way, which will play an important role in the picture.  Finally, we
suspect that the vanishing loci $X^\sigma_4$ and $X^\sigma_5$ (which can be
defined similarly to $X^\sigma_6$) may also have interesting properties related
to the geometry of Debarre--Voisin varieties.

\subsection{Moduli spaces}
\label{sec:modulispace}

For a general trivector $\sigma$, the variety $X^\sigma_6$ was shown to be a smooth hyperkähler fourfold of $\KKK^{[2]}$-deformation type in~\cite{dv}. We recall some basic properties for such manifolds.

The second cohomology group $H^2(X^\sigma_6,\bZ)$ carries a unique primitive integral quadratic form $\q$, known as
the {\em Beauville--Bogomolov--Fujiki form}. It is
non-degenerate of signature $(3,20)$, with the property that for all $
x_1,x_2,x_3,x_4\in H^2(X^\sigma_6,\bZ)$, we have
\begin{equation}\label{eq:bbf}
\int_{X^\sigma_6}x_1\cdot x_2\cdot x_3\cdot x_4=\q(x_1,x_2)
\q(x_3,x_4)+\q(x_1,x_3)\q(x_2,x_4)+\q(x_1,x_4)\q(x_2,x_3).
\end{equation}
The lattice $(H^2(X^\sigma_6,\bZ),\q)$ is isomorphic to the following even lattice
\[\Lambda\coloneqq U^{\oplus3}\oplus E_8(-1)^{\oplus2}\oplus\inner{-2},\]
where $U$ is the hyperbolic plane, $E_8(-1)$ is the $E_8$ lattice with negative
definite form, and $\inner{-2}$ is the lattice generated by one element with
square $-2$. The discriminant of~$\Lambda$ is equal to~$2$.
The Plücker line bundle coming from the ambient Grassmannian
provides a polarization $H$ on $X_6^\sigma$ of square $22$ and
divisibility~$2$, that is $\q(H,H)=22$ and
$\q\big(H,H^2(X^\sigma_6,\bZ)\big)=2\bZ$. These two invariants together determine a unique $\O(\Lambda)$-orbit, so we may fix one element $h\in\Lambda$ in this orbit. 

By the property of the quadratic form $\q$, the primitive cohomology
$H^2(X^\sigma_6,\bZ)_\prim$ can be identified with $H^\perp$, the sublattice orthogonal
to $H$ with respect to $\q$. This is a lattice of signature~$(2,20)$ and discriminant~$11$.
It carries a polarized integral Hodge structure of type $(1,20,1)$.
The {\em
period domain} that parametrizes such Hodge structures is the normal
quasi-projective variety
\[\cP\coloneqq\Omega_h/\O(\Lambda,h)\quad\text{where}
\quad\Omega_h\coloneqq\setmid{[x]\in\bP(\Lambda_\bC)}{\q(x,x)=\q(x,h)=0,\q(x,\bar x)
>0}.\]
In other words, we consider the domain of period points $\Omega_h$, and take its quotient by the group $\O(\Lambda,h)\coloneqq\setmid{\varphi\in \O(\Lambda)}{\varphi(h)=h}$.%
\footnote{Note that in general, $\Omega_h$ admits two connected components $\Omega_h^+$ and $\Omega_h^-$, and to define the period domain $\cP$, one need to pick one of the connected components and take its quotient by a smaller group $\Mon(\Lambda,h)\coloneqq \Mon(\Lambda)\cap \O(\Lambda, h)$, which consists of all monodromy operators fixing the class $h$. But for $\KKK^{[2]}$-type, the two descriptions of $\cP$ are equivalent. See the survey~\cite[Section~3.9]{d}.}
There exists also an irreducible coarse moduli space $\cM_{22}^{(2)}$ for polarized hyperkähler fourfolds of $\KKK^{[2]}$-type with a polarization of
degree~$22$ and divisibility~$2$, and the polarized period map
\[\p\colon \cM^{(2)}_{22}\into\cP,\quad [X]\mapsto [H^{2,0}(X,\bC)]\] is an open immersion by \cite[Theorem~1.10]{markman}. We refer to this result as the polarized global Torelli theorem. It relies on the global Torelli theorem proved by Verbitsky.

For each saturated sublattice~$K\subset \Lambda$ of rank~$2$ and
signature~$(1,1)$ containing~$h$, its orthogonal complement $K^\perp$ defines a
codimension~$2$ subspace $\bP(K^\perp\otimes\bC)\subset\bP(\Lambda_\bC)$,
whose image in~$\cP$ is an irreducible algebraic hypersurface~$\cD_K$
called a {\em Heegner divisor}. The discriminant of~$K^\perp$ is a negative
even integer, and we refer to $|\disc(K^\perp)|$ as the {\em discriminant} of the Heegner
divisor. Following Hassett and Debarre--Macrì, we will label each Heegner divisor
using its discriminant.
In our case,~\cite[Proposition~4.1]{dm} shows that
each Heegner divisor~$\cD_{2e}$ with given discriminant~$2e$ is irreducible if
non-empty (in \loccit{}, the divisor $\cD_{2e}$ is denoted by $\cD_{22,2e}^{(2)}$, referring to the fact that we are working with the moduli space $\cM_{22}^{(2)}$% of hyperk\"ahler fourfolds with a divisibility-$2$ polarization of square $22$
). Since the lattice $\Lambda$ is of discriminant 2, {\it a priori} the discriminant $\disc (K)$ can either be $2\disc (K^\perp)$ or $\frac12\disc (K^\perp)$. But by assuming that $K$ contains the class $h$, we are always in the first case. This very useful fact is obtained in the proof of \cite[Proposition~4.1]{dm}. We extracted it as a lemma.
% Moreover, in the proof of the same proposition, it is shown that the discriminant of~$K$ is always two times that of~$K^\perp$, a fact that we will be needing later.
\begin{lemma}[Debarre--Macrì]
\label{lemma:disc_K_two_times}
Let $n$ be an positive integer with $n\equiv-1\pmod 4$, and let $h\in \Lambda=\Lambda_{\KKK^{[2]}}$ be a class of square $2n$ and divisibility $2$. Let $K\subset \Lambda$ be a saturated sublattice of rank $2$ and signature $(1,1)$ containing $h$. Then we have $\disc (K) = 2\disc (K^\perp)$.
\end{lemma}

As mentioned in the introduction, we also have the GIT quotient
\[\cM\coloneqq\bP(\bw3V_{10}^\vee)\git\SL(V_{10})\]
for the trivectors~$\sigma$. The class of a trivector $\sigma$ is denoted by $[\sigma]\in \cM$, and we denote by $\cM^\smooth$ the open locus of classes $[\sigma]$ such that $X^\sigma_6$ is smooth of dimension 4. Therefore we get the following diagram.
\begin{equation}
\label{eq:period_map}
\begin{tikzcd}
\cM\rar[dashed, "\m"]&[2em]\cM_{22}^{(2)} \rar[hook,"\p"]&\cP\\[-1.2em]
{[\sigma]}\rar[mapsto] & {[X^\sigma_6]\qquad [X]} \rar[mapsto] & {[H^{2,0}(X,\bC)]}
\end{tikzcd}
\end{equation}
Here $\m$ is the modular map given by the Debarre--Voisin construction.
%It is dominant by \cite[Lemma~4.6]{dv}, which shows that the differential of $\m$ is surjective everywhere.
%
%\begin{lemma}[Debarre--Voisin]
%\label{lemma:diff_m}
%Write $(X,L)$ for the pair $\big(X_6^\sigma,\cO_{X_6^\sigma}(1)\big)$.
%Whenever $X$ is of dimension $4$, any first-order deformation of the pair $(X, L)$ is given by a deformation of $\sigma$. More precisely, the Kodaira--Spencer map
%\[
%\KS\colon \bw3V_{10}^\vee /\bC\sigma\to \Def_{(X,L)}(\bC[\varepsilon])
%\simeq \Ext^1\big(\cP_{X,L},\cO_X\big)
%\]
%is surjective.
%where the bundle $\cP_{X,L}$ is the extension
%\[
%0\to \Omega_X\to \cP_{X,L}\to \cO_X\to 0\]
%given by $c_1(L)\in \Ext^1(\cO_X,\Omega_X)$.
%\end{lemma}
%\begin{corollary}
%\label{lemma:m_dominant}
%The map $\m$ restricted to the stable locus $\cM^s\subset \cM$ is quasi-finite and dominant.
%\end{corollary}
%\begin{proof}
%This essentially follows from the surjectiveness of the Kodaira--Spencer map and the fact that both $\cM$ and $\cM^{(2)}_{22}$ are of dimension 20.

%More precisely, the orbit of a stable element $[\sigma]\in \bP(\bw3V_{10}^\vee)^s$ with respect to the $\SL(V_{10})$-action is of codimension 20 in $\bP(\bw3V_{10}^\vee)$. Suppose that a curve $C$ inside $\cM^s$ is contracted by the map $\m$, then the preimage of $C$ in $\bP(\bw3V_{10}^\vee)$ will have codimension 19 and is contracted to a point in $\cM^{(2)}_{22}$. This would contradict the surjectiveness of the Kodaira--Spencer map.
%\end{proof}
We recollect the results concerning the moduli spaces and period maps in the following theorem.
\begin{theorem}
\label{thm:moduli-space}
Consider the diagram~\eqref{eq:period_map}.
\begin{enumerate}
\item The modular map $\m$ is dominant and birational.
% More precisely, when restricted to the open locus $\cM^\smooth$, the map $\m$ gives an open immersion $\m\colon \cM^\smooth\into \cM^{(2)}_{22}$.
\item The polarized period map $\p$ is an open immersion. The complement of the image of $\p$ is an irreducible Heegner divisor $\cD_{22}$.
\end{enumerate}
\end{theorem}
\begin{proof}
% We just shown in Lemma~\ref{lemma:m_dominant} that the map $\m$ is dominant and quasi-finite.
The map $\m$ is dominant by \cite[Lemma~4.6]{dm}. The birationality of $\m$ is proved by O'Grady in \cite[Theorem~1.9]{ogrady}. The period map $\p$ is an open immersion by the polarized global Torelli theorem (see \cite[Theorem~1.10]{markman}), and the complement of the image of $\p$ is always a union of Heegner divisors. In the $\KKK^{[2]}$-case, the set of these Heegner divisors has been completely determined by Debarre--Macrì in \cite[Theorem 6.1]{dm}.
%
%Finally, since the period domain $\cP$ is normal, so is the moduli space $\cM^{(2)}_{22}$. By the Zariski Main Theorem, the restriction of $\m$ to the open locus $\cM^\smooth$ is therefore an open immersion.
\end{proof}

To relate the divisors of $\cM$ and $\cP$, we state the following useful lemma.
\begin{lemma}
\label{lemma:normality-divisors-birational}
Let $f\colon X\to Y$ be a birational morphism of varieties.
% , extended to its regular locus (in other words, $f$ is defined in codimension $1$).
Assume that $Y$ is regular in codimension $1$. Let $D$ be an irreducible divisor of $X$ that is not contracted by $f$, that is, $f(D)$ is a divisor of $Y$. Then $D$ is mapped birationally onto its image. Conversely, let $D'$ be an irreducible divisor of $Y$ such that the preimage $f^{-1}(D')$ is non-empty, then $f$ restricts to a birationl morphism from $f^{-1}(D')$ to $D'$.
\end{lemma}
\begin{proof}
Since $f$ is a birational morphism, it induces an isomorphism between the function fields
\[
f^\sharp\colon k(Y)\simto k(X).
\]
If $f(D)$ is a divisor of $Y$, since $Y$ is regular in codimension 1, the local ring at the generic point of $f(D)$ is a discrete valuation ring and therefore integrally closed in $k(Y)$. For dimension reasons, the morphism $f^\sharp\colon\cO_{Y,\eta_{f(D)}}\to \cO_{X,\eta_D}$ is finite. So this gives an isomorphism of local rings and induces also an isomorphism of function fields $k(f(D))\simto k(D)$. The converse follows from the same argument.
\begin{comment}
We first apply the Nagata's compactification theorem to get a factorization
\[
f\colon X\stackrel{i}{\into} \overline X\stackrel{\overline f}{\to} Y,
\]
where $i$ is an open immersion and $\overline f$ is proper. The morphism $\overline f$ is still birational, so by the Zariski main theorem, all the fibers of $\overline f$ are connected.

If $f(D)$ is a divisor of $Y$, then for dimension reason, there exists an open subset $U$ of $f(D)$ such that $\overline f\colon {\overline f}^{-1}(U)\to U$ is finite. Since the fibers are connected, this is in fact an isomorphism. Hence $f$ indeed maps $D$ birationally on to $f(D)$. Similarly, if $D'$ is an irreducible divisor of $Y$, then $\overline f$ restricts to a proper birational morphism $\overline f\colon {\overline f}^{-1}(D')\to D'$. Since the preimage $f^{-1}(D')$ can be identified with $X\cap {\overline f}^{-1}(D')$, if it is non-empty then it is indeed mapped birationally onto $D'$.
\end{comment}
\end{proof}

In the next section, we will see that the complement in~$\cM$ of the
locus~$\cM^\smooth$ is an irreducible divisor~$\cD^{3,3,10}$. This divisor is induced by the
$\SL(V_{10})$-invariant discriminant hypersurface
in $\bP\big(\bw3V_{10}^\vee\big)$. We can extend the period map $\p\circ\m$ so that it is defined in codimension~1. We denote this extension by
\[\begin{tikzcd}
\tilde\p\colon\cM\rar[dashed]&\cP.
\end{tikzcd}\]
Since $\cP$ is normal and $\tilde\p$ is birational, Lemma~\ref{lemma:normality-divisors-birational} applies. So each irreducible divisor of $\cM$ is mapped birationally onto its image, and conversely, for each irreducible divisor of $\cP$ with non-empty preimage, the preimage will be an irreducible divisor of $\cM$, and the two divisors are birational via $\tilde \p$. 
We will show that this divisor $\cD^{3,3,10}$ is birationally mapped onto the Heegner
divisor $\cD_{22}$ via the extended period map~$\tilde\p$. Combining these facts, we get the following picture of the moduli spaces
\[\begin{tikzcd}[row sep=0.1em]
\cM&&\cP\\
\shortparallel&&\shortparallel\\
\cM^\smooth\ar[r,"\m","\text{bir.}"'] & \cM^{(2)}_{22} \ar[r,,"\p","\sim"'] & \Im(\p)\\
\sqcup&&\sqcup\\
\cD^{3,3,10}\ar[rr,dashed,"\tilde\p","\text{bir.}"'] && \cD_{22}.
\end{tikzcd}\]
Divisors of $\cP$ with empty preimage are called {\em Hassett--Looijenga--Shah divisors} (HLS for
short) and they are the main focus of the paper~\cite{dhov}. These divisors
correspond to $\SL(V_{10})$-orbits of higher codimension in~$\cM$ that need to
be blown up in order to resolve the indeterminacy of~$\tilde\p$. Such loci are necessarily contained in the divisor $\cD^{3,3,10}$.

We will also study two other divisors in $\cM$ coming from
$\SL(V_{10})$-invariant hypersurfaces. Via the extended period map $\tilde\p$, they are birationally mapped onto some
Heegner divisors~$\cD_{2e}$. In
principle, for each Heegner divisor in~$\cP$ that is not HLS, we could try to
describe it in terms of divisors in $\cM$.  In the case of the three divisors
that we study, this is done by imposing various degeneracy conditions
on~$\sigma$.  Such descriptions also allow us to characterize these Heegner
divisors as the loci where the varieties $X^\sigma_k$ become singular.

Finally, we call the preimage $\p^{-1}(\cD_{2e})$ in $\cM^{(2)}_{22}$ of each Heegner divisor~$\cD_K=\cD_{2e}$ a {\em Noether--Lefschetz divisor}, and we denote it by~$\cC_{2e}$. Thanks to Theorem~\ref{thm:moduli-space}, Noether--Lefschetz divisors and Heegner divisors give almost the same notion: $\cC_{2e}$ can be identified with $\cD_{2e}\setminus \cD_{22}$ via the period map, and in particular $\cC_{22}$ is empty. A very general member~$X$ of each Noether--Lefschetz
divisor has Picard number~$2$, and the Picard group can be identified with the
sublattice~$K\subset\Lambda$ of rank~$2$ via the identification $H^2(X,\bZ)\simeq \Lambda$.  
Its transcendental sublattice $H^2(X,\bZ)_\trans$ is defined as the orthogonal complement of the Picard group inside $H^2(X,\bZ)$, and can be identified with $K^\perp\subset\Lambda$.
By Lemma~\ref{lemma:disc_K_two_times}, the discriminant $2e$ of a Noether--Lefschetz divisor $\cC_{2e}$/a Heegner divisor $\cD_{2e}$ is equal to $|\frac12\disc(\Pic(X))|$ for a very general member $X$. A Noether--Lefschetz $\cC_{2e}$ is called HLS if the corresponding Heegner divisor $\cD_{2e}$ is HLS. Such a divisor parametrizes polarized hyperkähler fourfolds in $\cM^{(2)}_{22}$ that do not arise from the Debarre--Voisin construction.

\section{The Heegner divisor of degree $22$ }
\label{sec_3}
%\subsection{Smoothness of \texorpdfstring{$X^\sigma_3$}{X3} and \texorpdfstring{$X^\sigma_6$}{X6}}

We first state a lemma regarding the smoothness of~$X_6^\sigma$.
\begin{lemma}\label{lemma:singularV3}
Let $[V_6]$ be a point in $X^\sigma_6$. The Debarre--Voisin variety $X^\sigma_6$ is not smooth of
dimension $4$ at $[V_6]$ if and only if there exists $V_3\subset V_6$ such that
$\sigma(V_3,V_3,V_{10})=0$.
\end{lemma}
\begin{proof}
The Zariski tangent space $T_{X^\sigma_6,[V_6]}$ of the Debarre--Voisin variety $X_6^\sigma$ at $[V_6]$ is given as the kernel of the differential
\[d\sigma\colon  T_{\Gr(6,V_{10}), [V_6]}=\Hom(V_6,V_{10}/V_6)\to\bw3
V_6^\vee,\]
which maps $f\in\Hom(V_6,V_{10}/V_6)$ to the~$3$-form
\[d\sigma(f)\colon (v_1,v_2,v_3)\mapsto \sigma(f(v_1),v_2,v_3) +
\sigma(v_1,f(v_2),v_3) + \sigma(v_1,v_2,f(v_3)).\]
Therefore $X_6^\sigma$ is not smooth of dimension $4$ if and only if the differential is not surjective, or equivalently, if there exists some non-zero $\omega\in(\bw3
V_6^\vee)^\vee=\bw3 V_6$ such that $\omega |_{\Im(d\sigma)}=0$, that is, for any $f\in\Hom(V_6,V_{10}/V_6)$ we have $d\sigma(f)(\omega)=0$.

Suppose that $V_3\subset V_6$ is a subspace satisfying the vanishing condition $\sigma(V_3,V_3,V_{10})=0$. Then a non-zero $\omega\in \bw3V_3$ satisfies the above property, so $X_6^\sigma$ is not smooth of dimension 4 at $[V_6]$.

Conversely, the orbit closures for the $\GL(V_6)$-action on $\bw3V_6$ have long been classified (see \cite[Section~35.2]{gurevich}, also \cite[Section~5.2]{kwe6}): there are five of
them, including $\set 0$. So we study the four non-zero orbits case by case.
\begin{itemize}
\item If $\omega$ is completely decomposable, that is when $\omega=e_1\wedge
e_2\wedge e_3$, consider a map~$f$ with $f(e_1)=f(e_2)=0$: the property of $\omega$ shows that $\sigma(e_1,e_2,f(e_3))=0$, so by varying $f$ we get $\sigma(e_1,e_2,V_{10})=0$. Similarly we have $\sigma(e_1,e_3,V_{10})=\sigma(e_2,e_3,V_{10})=0$. So the subspace $V_3=\inner{e_1,e_2,e_3}$ satisfies the vanishing condition $\sigma(V_3,V_3,V_{10})=0$.
\item If $\omega$ is of rank 5, it can be written as $e_1\wedge
e_2\wedge e_3+e_1\wedge e_4\wedge e_5$. Let $V_1=\inner{e_1}$ and
$V_4=\inner{e_2,e_3,e_4,e_5}$. We get $\sigma(V_1,V_1+V_4,V_{10})=0$. Consider
the map
\[\varphi_\sigma\colon \bw2V_4\to (V_{10}/V_6)^\vee\]
induced by $\sigma$ (note that $\sigma|_{V_6}=0$).  The kernel of
$\varphi_\sigma$ is a subspace of dimension at least $2$. Note also that the
subset in $\bw2V_4$ of decomposable elements is the affine cone over the
Grassmannian $\Gr(2,V_4)$, which is a quadric hypersurface. This shows that
there is some decomposable element $u\wedge v$ in the kernel of
$\varphi_\sigma$. The subspace $\inner{e_1,u,v}$ thus provides the $V_3$ we
want. Moreover, without loss of generality, we may suppose that $u\wedge v=e_2\wedge e_3$; then $\inner{e_1,e_4,e_5}$ gives another $V_3$ satisfying the vanishing condition.
\item If $\omega$ is of type $e_1\wedge e_2\wedge e_4+e_2\wedge e_3\wedge
e_5+e_1\wedge e_3\wedge e_6$, by considering a map $f$ with $f(e_1)=f(e_2)=f(e_3)=0$, we can see that $\inner{e_1,e_2,e_3}$ gives a $V_3$ such that $\sigma(V_3,V_3,V_{10})=0$.
\item If $\omega$ is general, so of type $e_1\wedge e_2\wedge e_3+e_4\wedge
e_5\wedge e_6$, both $\inner{e_1,e_2,e_3}$ and $\inner{e_4,e_5,e_6}$ give
a $V_3$ such that $\sigma(V_3,V_3,V_{10})=0$.
\end{itemize}
Therefore, the Zariski tangent space is not of dimension 4 if and only if there exists a $V_3$ satisfying the vanishing condition.
\end{proof}

Inside $\bP\big(\bw3V_{10}^\vee\big)$, the set of trivectors $\sigma$ admitting
a $V_3$ such that $\sigma(V_3,V_3,V_{10})=0$ is the projective dual variety
$\Gr(3,V_{10})^*=V(f)$, which is a hypersurface defined by an
$\SL(V_{10})$-invariant polynomial~$f$ usually referred to as
the {\em discriminant}. The degree of $f$ is equal to 640, which was first
computed by Lascoux in \cite{lascoux}.
By \cite{dv}, there is a Zariski-dense open subset of thehypersurface $V(f)$ defining the irreducible subvariety
\begin{equation}
\label{eqD3310}
\cD^{3,3,10}\coloneqq\setmid{[\sigma]\in \cM}{\exists V_3\quad
\sigma(V_3,V_3,V_{10})=0}
\end{equation}
in the GIT quotient~$\cM$.
By a parameter count, this is a divisor in $\cM$. Any $[\sigma]$ in
$\cM\setminus\cD^{3,3,10}$ defines a smooth~$4$-dimensional hyperkähler
$X^\sigma_6$, while general points in the divisor $\cD^{3,3,10}$ admit a unique
$V_3$ satisfying the degeneracy condition, which provides a point on $X^\sigma_3$.
Hence the divisor~$\cD^{3,3,10}$ is exactly the complement of the locus
$\cM^\smooth$.

For a general element $[\sigma]\in\cD^{3,3,10}$, due to \autoref{lemma:singularV3} the
singular locus of $X^\sigma_6$ consists exactly of those $V_6$ containing
$V_3$, so set-theoretically we get
\[\Sing(X^\sigma_6)=S_{22}\coloneqq\setmid{[V_6]\in X^\sigma_6}{V_6\supset V_3},\]
which is a degree $22$ K3 surface living in the Grassmannian
$\Gr(3,V_{10}/V_3)$. We remark that in~\cite{dv}, it was only proved that the K3
surface $S_{22}$ is contained in the singular locus, instead of an equality.

In fact, for a general $[\sigma]\in \cD^{3,3,10}$, we can get a precise description of the type of singularity along~$S_{22}$: similar to the nodal cubic case, by blowing up the singular locus, we obtain a smooth hyperkähler fourfold of $\KKK^{[2]}$-type, with the exceptional divisor being a $\bP^1$-bundle over the K3 surface. We will not need this fact for the rest of the paper. Its proof can be found in the appendix of \cite{oberdieck}. We summarize our results on the singular loci of~$X_3^\sigma$ and~$X_6^\sigma$
in the following proposition.

\begin{proposition}
\label{propsingD33}
The divisor $\cD^{3,3,10}$ is the locus in~$\cM$ of $\GL(V_{10})$-classes of trivectors $[\sigma]$ for
which $X^\sigma_2$, $X^\sigma_3$, and~$X^\sigma_6$ become singular. Moreover,
for a general element $[\sigma]\in \cD^{3,3,10}$ such that
$\sigma(V_3,V_3,V_{10})=0$, we have
\begin{align*}
\Sing(X^\sigma_2)&= \setmid{[V_2]\in \Gr(2,V_{10})}{V_2\subset V_3}\simeq
\bP(V_3^\vee), \\
\Sing(X^\sigma_3)&= \set{[V_3]}, \\
\Sing(X^\sigma_6)&= \setmid{[V_6]\in X^\sigma_6}{V_6\supset V_3}=S_{22},
\end{align*}
where $S_{22}$ is a degree $22$ K3 surface. Finally, the divisor $\cD^{3,3,10}$
is mapped birationally onto the Heegner divisor $\cD_{22}$ via the extended period
map $\tilde\p\colon\cM\dashrightarrow\cP$.
\end{proposition}
\begin{proof}
The statement on the singular locus of~$X_3^\sigma$ is obvious.
For~$X^\sigma_6$, the statement follows from \autoref{lemma:singularV3}.
For~$X^\sigma_2$, we use~\cite[Lemma~11]{bfm} to conclude that $X^\sigma_2$ is
singular exactly when $[\sigma]\in \cD^{3,3,10}$, and the singular locus consists
of points $[V_2]\in X^\sigma_2$ such that $V_2\subset V_3$.

The K3 surface $S_{22}$ has Picard rank 1 for a very general $[\sigma]\in\cD^{3,3,10}$. In~\cite{dv}, it was shown that $X^\sigma_6$ is birational to the Hilbert
square~$S_{22}^{[2]}$. Hence the extended period map $\tilde\p$ is defined over an open set of $\cD^{3,3,10}$ by mapping $[\sigma]$ to the period of $S_{22}^{[2]}$. In this case, the fourfold $S^{[2]}_{22}$ has Picard rank 2, and the Picard lattice is generated by the classes $H$ and $\delta$, where $H$ is induced by the polarization on $S_{22}$, and $\delta$ is half the class of the diagonal. Hence the Picard lattice has intersection matrix
\[\begin{pmatrix}22&0\\0&-2\end{pmatrix} \]
which is of discriminant~$-44$. So we may use Lemma~\ref{lemma:disc_K_two_times} to deduce that the extended period map $\tilde \p$ maps the divisor $\cD^{3,3,10}$ onto the Heegner divisor $\cD_{22}$.
We conclude that $\tilde\p$ restricts to a birational map between $\cD^{3,3,10}$ and $\cD_{22}$ using Lemma~\ref{lemma:normality-divisors-birational}.
\end{proof}

\begin{remark}
\label{rmqunirationality22}
This result implies that $\cD_{22}$ is unirational (which was
already known from \cite{dv}). Indeed, $\cD^{3,3,10}$ can be seen as a quotient
of the vector bundle $(\cU_3\wedge \cU_3\wedge V_{10})^\perp \subset
\bigwedge^3 V_{10}^\vee \otimes \cO_G$ over the Grassmannian $G=\Gr(3,V_{10})$
by the natural action of the group $\GL(V_{10})$.
\end{remark}

\subsection{Rational Hodge structures of \texorpdfstring{$X^\sigma_3$}{X3}}
In this section, we suppose that $\sigma$ is such that $X^\sigma_3$
and~$X^\sigma_6$ are both smooth of respective expected dimensions $20$ and $4$,
which holds for~$[\sigma]$ not in the divisor~$\cD^{3,3,10}$.
We study the Hodge structures of~$X^\sigma_3$. Note that the cohomology
ring~$H^*(X_3^\sigma,\bZ)$ is torsion-free, thanks to the Lefschetz hyperplane theorem
and the universal coefficient theorem.

We introduce one interesting Hodge structure on $X^\sigma_3$.  Denote by
$j\colon X^\sigma_3\to\Gr(3,V_{10})$ the canonical embedding.  For a given
coefficient ring $R$, the {\em vanishing cohomology}, studied in the
original work~\cite{dv}, is defined as
\[H^{20}(X^\sigma_3,R)_\van\coloneqq\ker\big(j_*\colon H^{20}(X^\sigma_3,R)\to
H^{22}(\Gr(3,V_{10}),R)\big).\]
When the coefficient ring is $\bQ$, the vanishing cohomology can also be
defined as the orthogonal complement of $j^*H^{20}(\Gr(3,V_{10}),\bQ)$ with
respect to the cup-product on $H^{20}(X^\sigma_3,\bQ)$; indeed, for $\beta \in H^{20}(\Gr(3,V_{10}),\bQ)$ and $\alpha\in H^{20}(X_3,\bQ)$, we have that $\alpha \cdot j^* \beta = j_* \alpha \cdot \beta$, and moreover $j^*$ is injective in degree $20$ by the Lefschetz hyperplane theorem. Hence there is an
orthogonal decomposition
\begin{equation}
\label{eq:H20}
H^{20}(X^\sigma_3,\bQ)=H^{20}(X^\sigma_3,\bQ)_\van\stackrel\perp\oplus
j^*H^{20}(\Gr(3,V_{10}),\bQ).
\end{equation}
This decomposition does not work for $\bZ$-coefficients, as the sum of the
two sublattices is not saturated. In fact, $H^{20}(X^\sigma_3,\bZ)$ is a
unimodular lattice and the lattice $j^*H^{20}(\Gr(3,V_{10}),\bZ)$ is
generated by ten Schubert classes
\[j^*\sigma_{730}, j^*\sigma_{721},j^*\sigma_{640}, j^*\sigma_{631},
j^*\sigma_{622}, j^*\sigma_{550}, j^*\sigma_{541}, j^*\sigma_{532},
j^*\sigma_{442},j^*\sigma_{433},\]
with intersection product given by $j^*\alpha\cdot
j^*\beta=\alpha\cdot\beta\cdot\sigma_{100}$. Thus we can explicitly write out
the intersection matrix of $j^*H^{20}(\Gr(3,V_{10}),\bZ)$ as
\begin{equation}\label{int_matrix}\scalebox{.9}{$
\arraycolsep=5pt
\begin{pmatrix}
1  &  0  &  1  &  0  &  0  &  0  &  0  &  0  &  0  &  0  \\
0  &  0  &  1  &  0  &  0  &  1  &  0  &  0  &  0  &  0  \\
1  &  1  &  0  &  1  &  0  &  0  &  0  &  0  &  0  &  0  \\
0  &  0  &  1  &  1  &  0  &  0  &  1  &  0  &  0  &  0  \\
0  &  0  &  0  &  0  &  0  &  1  &  1  &  0  &  0  &  0  \\
0  &  1  &  0  &  0  &  1  &  0  &  0  &  0  &  0  &  0  \\
0  &  0  &  0  &  1  &  1  &  0  &  0  &  1  &  0  &  0  \\
0  &  0  &  0  &  0  &  0  &  0  &  1  &  1  &  1  &  0  \\
0  &  0  &  0  &  0  &  0  &  0  &  0  &  1  &  0  &  1  \\
0  &  0  &  0  &  0  &  0  &  0  &  0  &  0  &  1  &  1
\end{pmatrix}$}
\end{equation}
which has determinant $11$.  Therefore, $j^*H^{20}(\Gr(3,V_{10}),\bZ)$ is a
saturated sublattice of discriminant~$11$, and so is its orthogonal
$H^{20}(X^\sigma_3,\bZ)_\van$.  The whole lattice $H^{20}(X^\sigma_3,\bZ)$ is
not even (for example $(j^* \sigma_{730})^2=1$), while the vanishing cohomology
is, as it is generated by the vanishing cycles, whose self-intersection is
always 2
(see~\cite[Chapter~2.3.3, Lemma~2.26 and Exercise~2]{voisinbook}).

The Hodge structure on the vanishing cohomology $H^{20}(X^\sigma_3,\bZ)_\van$
is of K3-type (\cite[Theorem~2.2]{dv}; see also~\cite[Theorem~3.27]{voisinbook}): it has Hodge numbers $h^{9,11}=h^{11,9}=1$,
$h^{10,10}=20$, and the other Hodge numbers are all zero; moreover, for very general
$\sigma$ (those outside the union of all Noether--Lefschetz divisors), there are
no Hodge classes of type $(10,10)$ and the Hodge structure is simple.

To relate the varieties $X^\sigma_3$ and $X^\sigma_6$, we define a diagram
\begin{equation}\label{eq:I36}
\begin{tikzcd}
&I^\sigma_{3,6}\coloneqq\setmid{[V_3\subset V_6]\in \Flag(3,6,V_{10})}{\sigma|_{V_6}=0}
    \ar{ld}[swap]{p}\ar{rd}{q}&\\
X^\sigma_3&&X^\sigma_6,
\end{tikzcd}
\end{equation}
where $q$ is a fibration with fibers isomorphic to $\Gr(3,6)$.  It is clear
that the incidence variety~$I^\sigma_{3,6}$ is smooth of expected
dimension~$13$ whenever~$X_6^\sigma$ is smooth of dimension~$4$. Note that the
projection~$p$ is not surjective since~$X^\sigma_3$ has dimension~$20$.
This correspondence induces a morphism
\[q_*p^*\colon H^{20}(X^\sigma_3,R)\to H^2(X^\sigma_6,R).\]
When $R=\bQ$, it was proven in~\cite{dv} that $q_*p^*$ gives an isomorphism
between the two $\bQ$-Hodge structures $H^{20}(X^\sigma_3,\bQ)_\van$ and
$H^2(X^\sigma_6,\bQ)_\prim$. We briefly recall the idea of the proof: the
authors first proved that the morphism is not identically~$0$. Then, since
the Hodge structure on~$H^2(X_6^\sigma,\bQ)_\prim$ is simple for~$\sigma$ very
general, the map~$q_*p^*$ is an isomorphism for such~$\sigma$. Finally, since
the topology does not change when we deform~$\sigma$, the isomorphism holds
whenever~$X_3^\sigma$ and~$X_6^\sigma$ are smooth.

We will show that $q_*p^*$ also gives a Hodge isometry with
$\bZ$-coefficients. This is analogous to the result of Beauville--Donagi for
cubic fourfolds in~\cite{bd}. Let us first state a lemma over $\bQ$-coefficients.
\begin{lemma}
\label{lemma:uniqueness_of_q}
Let $X$ be a hyperkähler manifold. Using parallel transport operators, a quadratic form on $H^2(X,\bQ)$ can be transported to $H^2(X',\bQ)$ for any small deformation $X'$ of $X$. Then the Beauville--Bogomolov--Fujiki form $\q$ is the unique up to scalar quadratic form on $H^2(X,\bQ)$ satisfying the following property: for any small deformation $X'$ of $X$, $H^{2,0}(X')$ is orthogonal to $H^{2,0}(X')\oplus H^{1,1}(X')$ with respect to $\q$.

We have a similar uniqueness result for quadratic forms on $H^2(X,\bQ)_\prim$ for a polarized hyperkähler manifold $(X,H)$, where we consider all small deformations of the pair $(X,H)$ instead.
\end{lemma}
\begin{proof}
We prove the first statement. Let $\q'$ be another quadratic form satisfying the desired property. We need to show that there exists~$\lambda$ such that~$\q'=\lambda\q$.
Let~$\omega\in H^{2,0}(X)$ be the class of a holomorphic symplectic form on~$X$. Then $\omega$ is orthogonal to $H^{2,0}(X)\oplus H^{1,1}(X)$ with respect to both $\q$ and $\q'$ by assumption. Since we also have $\q(\omega,\bar\omega)>0$, we can consider the following number
\[\lambda=\frac{\q'(\omega,\bar\omega)}{\q(\omega,\bar\omega)}.\]
We see that~$\omega$ is also orthogonal to~$\bar\omega$ with respect to the
form~$\q'-\lambda\q$, therefore~$\omega$ lies in the
kernel~$\ker(\q'-\lambda\q)$.
Equivalently, the linear subspace $K$ of $\bP(H^2(X,\bC))$ defined by the form
$\q'-\lambda \q$ contains the period point $[H^{2,0}(X)]$.
Using parallel transport operators, the value of $\lambda$ remains constant on all small deformations of $X$. Since the properties of $\q$ and $\q'$ also hold for all small deformations of $X$, the kernel~$K$ must contain all the period points $[H^{2,0}(X')]$. By the surjectivity of the local period map and the non-degeneracy of the form $\q$, $K$ coincides with the entire $\bP(H^2(X,\bC))$.
In other words, the form $\q'-\lambda \q$ is identically zero on $H^2(X,\bC)$, so
we may conclude that $\q'=\lambda \q$. The uniqueness in the polarized case follows from the same argument.
\end{proof}
\begin{corollary}
\label{cor:isometry_upto_constant}
The isomorphism of rational Hodge structures~$q_*p^*$ is a constant multiple of an
isometry.
\end{corollary}
\begin{proof}
Consider the intersection form on~$H^{20}(X_3^\sigma,\bQ)_\van$. The subspace $H^{11,9}(X_3^\sigma)$ is orthogonal to $H^{11,9}(X_3^\sigma)\oplus H^{10,10}(X_3^\sigma)_\van$ for degree reasons. Since $q_*p^*$ is a morphism of Hodge structures, the intersection form transports to a second quadratic form~$\q'$ on~$H^2(X_6^\sigma,\bQ)_\prim$ which satisfies the desired orthogonal condition. Now by varying $\sigma$ we get a locally complete family of polarized hyperkähler manifolds, so we may conclude that $\q'$ is a constant multiple of $\q$ by Lemma~\ref{lemma:uniqueness_of_q}.
\begin{comment}
only need to verify that $\q'$ satisfies $\q$ so we need to show that
there exists~$\lambda$ such that~$\q'=\lambda\q$. Let~$\omega\in
H^{2,0}(X_6^\sigma)$ be the class of a holomorphic symplectic form on~$X_6^\sigma$. We
have~$\q(\omega,\omega)=\q'(\omega,\omega)=0$, $\q(\omega,\bar\omega)>0$,
and~$\omega$ is orthogonal to~$H^{1,1}(X_6^\sigma)$ for both~$\q$ and~$\q'$.
Consider the following number
%
\[\lambda=\frac{\q'(\omega,\bar\omega)}{\q(\omega,\bar\omega)}.\]
%
We see that~$\omega$ is also orthogonal to~$\bar\omega$ for the
form~$\q'-\lambda\q$, therefore~$\omega$ lies in the
kernel~$\ker(\q'-\lambda\q)$. This form is thus degenerate so we have
$\det(\q'-\lambda\q)=0$. Since both~$\q$ and~$\q'$ have coefficients in~$\bQ$,
the value~$\lambda$ is an algebraic number and therefore does not change when
we deform~$\sigma$ in the moduli space. This means that the
kernel~$\ker(\q'-\lambda\q)$ must contain all the period points,
which, by the surjectivity of the local period map, span the
entire~$H^2(X_6^\sigma,\bQ)_\prim$. In other words, $\q'$ is identically equal
to~$\lambda\q$ on~$H^2(X^\sigma_6,\bQ)_\prim$.
\end{comment}
\end{proof}

% Let us sketch how the proof of the theorem will proceed.
% Just like for rational coefficients, it suffices to prove the isometry for
% one~$\sigma$ and conclude by continuity.
% For very general~$\sigma$ in the
% moduli space, the Hodge structure on~$H^2(X_6^\sigma,\bZ)_\prim$ is simple,
% while the {\em algebraic monodromy group}---the Zariski closure of the monodromy
% group in~$\GL(H^2(X_6^\sigma,\bR)_\prim)$---is the entire orthogonal group
% $\SO(H^2(X_6^\sigma,\bR)_\prim)$, which coincides with the Mumford--Tate group
% at~$\sigma$. Since~$q_*p^*$ is an isomorphism of $\bQ$-vector spaces, there is
% an induced second quadratic form~$\q'$ on~$H^2(X_6^\sigma,\bZ)_\prim$.
% We may choose some~$\lambda$ such that~$\q-\lambda\q'$ degenerates and admits
% some non-trivial kernel~$V$. This kernel is clearly monodromy-invariant,
% therefore it is also invariant by the Mumford--Tate group and induces a
% sub-Hodge structure.  By the simplicity of the Hodge structure,~$V$ must be the
% entire $H^2(X_6^\sigma,\bZ)_\prim$ so we get $\q=\lambda\q'$. In other words,
% the two forms~$\q$ and~$\q'$ can only differ by a non-zero scalar.

\subsection{Geometry of the Peskine variety \texorpdfstring{$X^\sigma_1$}{X1} and its incidence variety $I_{3,6}^\sigma$}
In this section, we give a criterion for the smoothness of the Peskine
variety~$X_1^\sigma$. Recall that this is the locus in~$\bP(V_{10})$ where the
rank of~$\sigma(V_1,-,-)$ drops to~$6$ or less. The orbits of skew-symmetric
forms are entirely determined by their ranks. Denote by $O^{ss}_r$ the orbit of
skew-symmetric (hence the superscript $ss$) $9\times 9$-matrices of rank $r$, where $r$ is even. Then
$X^\sigma_1$ is defined as the orbital degeneracy locus $D_{O^{ss}_6}(\sigma)$.
By the Bertini theorem for orbital degeneracy loci \cite[Proposition 2.3]{benedetti}, since $\bw2 \cQ(1)$ is globally generated, when $\sigma$ is
general the singular locus of~$X^\sigma_1$ is
\[
D_{\Sing(O^{ss}_{6})}(\sigma)=D_{O^{ss}_{4}}(\sigma)=\setmid{[V_1]\in\bP(V_{10})}{\rank\sigma(V_1,-,-)\le4},
\]
which is empty for dimensional reasons since
\[
\codim_{\bP(V_{10})}(D_{O^{ss}_{4}}(\sigma))=\codim_{\bw2 \bC^9}(O^{ss}_{4})= 10>
\dim\bP(V_{10}).
\]
So~$X_1^\sigma$ is smooth for a general~$\sigma$. Now we give a lemma that
describes the situation for all~$\sigma$.

\begin{lemma}
Let $\sigma\in \bw3 V_{10}^\vee$ be such that $X_1^\sigma$ is of dimension $6$.
A point $[V_1]\in X_1^\sigma$ is singular if and only if either
$\sigma(V_1,-,-)$ is of rank $\leq 4$ or there exists $V_3$ containing $V_1$
such that $\sigma(V_3,V_3,V_{10})=0$.
\end{lemma}

\begin{proof}
Since the orbit $O^{ss}_{6}$ is singular along $O^{ss}_{4}$,
the Bertini theorem for orbital degeneracy loci (\cite[Proposition 2.3]{benedetti}) implies that $D_{O^{ss}_{4}}(\sigma)
\subset \Sing D_{O^{ss}_{6}}(\sigma)=\Sing X_1^\sigma$,
which means that $X_1^\sigma$ is automatically
singular where $\sigma(V_1,-,-)$ is of rank $\le4$.  Therefore, we only need to
identify the singular locus of~$X_1^\sigma$ over the open
subset~$X_1^\sigma\setminus D_{O^{ss}_4}(\sigma)$, that is, where
$\sigma(V_1,-,-)$ is of rank $6$.

We will study the singularities by using a desingularization
$Z(\sigma)\subset \Flag(1,4,V_{10})$. The trivector $\sigma\in \bw3
V_{10}^\vee$ can be seen as a section of the vector bundle $(\cU_1\wedge
\cU_4\wedge V_{10})^\vee$ over the flag variety $\Flag(1,4,V_{10})$ (a
quotient of the trivial bundle $\bw3 V_{10}^\vee$).  The zero-locus
$Z(\sigma)$ of $\sigma$ gives a desingularization of
$X_1^\sigma$: since $Z(\sigma)$ parametrizes pairs
$[V_1\subset V_4]$ such that $V_4\subset\ker\sigma(V_1,-,-)$, the projection
$Z(\sigma)\to X_1^\sigma$ that maps $[V_1\subset V_4]$ to $[V_1]$
is an isomorphism over the open subset $X_1^\sigma\setminus
D_{O^{ss}_{4}}(\sigma)$ where $\sigma$ is of rank~$6$. Therefore, it suffices to
determine the singular locus of~$Z(\sigma)$.
% In other words, if $[V_1]\in X_1^\sigma$ is such that
% $\sigma(V_1,-,-)$ has rank $6$, then a neighborhood of $[V_1]$ in $X_1^\sigma$
% is isomorphic to a neighborhood of its preimage in $Z(\sigma)$.

The variety $Z(\sigma)$ being not smooth of dimension 6 at $[V_1\subset V_4]$
means that the differential
\[
d\sigma\colon  T_{\Flag(1,4,V_{10}),[V_1\subset
V_4]}\simeq\Hom(V_1,V_{10}/V_1)\oplus \Hom(V_4/V_1,V_{10}/V_4)\to(V_1\wedge
V_4\wedge V_{10})^\vee,
\]
which maps $f\in\Hom(V_1,V_{10}/V_1)\oplus \Hom(V_4/V_1,V_{10}/V_4)$ to the 3-form
\[d\sigma(f)\colon (v_1,v_2,v_3)\mapsto \sigma(f(v_1),v_2,v_3) +
\sigma(v_1,f(v_2),v_3) + \sigma(v_1,v_2,f(v_3)),\]
is not surjective (notice that $f$ can be seen as a class of maps from $V_4$ to $V_{10}$ modulo those maps that preserve $V_1$ and $V_4$). Equivalently, there exists a non-zero
$\omega\in\left((V_1\wedge V_4\wedge V_{10})^\vee\right)^\vee=V_1\wedge
V_4\wedge V_{10}$ such that $\omega |_{\Im(d\sigma)}=0$, that is, for any
$f\in\Hom(V_1,V_{10}/V_1)\oplus \Hom(V_4/V_1,V_{10}/V_4)$, we have
$d\sigma(f)(\omega)=0$.

Modulo a change of coordinates, one can always suppose that $V_1=\langle
e_1\rangle $, $V_4=\langle e_1,\dots,e_4\rangle $, and $V_{10}=\langle
e_1,\dots,e_{10}\rangle $, and that
\[
\omega=e_1\wedge\big(a (e_2\wedge e_3)+b (e_2\wedge e_5)+c (e_3\wedge e_6)+d (e_4\wedge e_7)\big)
\]
for certain coefficients $a,b,c,d \in \bC$. The proof is divided into three cases:
\begin{itemize}
\item If $d\neq 0$, consider a morphism $f$ sending $e_1,e_2,e_3$ to $0$. Then
$d\sigma(f)(\omega)=0$ shows that $\sigma(e_1,e_7,f(e_4))=0$. By varying $f$,
one gets $\sigma(V_1, V_4+\bC e_7,V_{10})=0$, which implies that
$\sigma(V_1,-,-)$ has rank at most $4$.
\item If $d=0$ and $b\neq 0$, consider a morphism $f$ sending $e_1,e_3,e_4$ to
$0$. Then $d\sigma(f)(\omega)=0$ shows that $\sigma(e_1,e_5,f(e_2))=0$. By
varying $f$, one gets $\sigma(V_1, V_4+\bC e_5,V_{10})=0$, again implying that
$\sigma(V_1,-,-)$ has rank at most $4$. Similarly one can treat the case when
$d=0$ and $c\ne 0$.
\item If $b=c=d=0$, consider a morphism $f$ sending $e_2,e_3,e_4$ to $0$. Then
$d\sigma(f)(\omega)=0$ shows that $\sigma(f(e_1),e_2,e_3)=0$. By varying $f$
and setting $V_3=\langle e_1,e_2,e_3 \rangle$, one gets $\sigma(V_3, V_3,V_{10})=0$.
\end{itemize}
Therefore, for a singular point $[V_1\subset V_4]$ in $Z(\sigma)$, either
$\sigma(V_1,-,-)$ is of rank $\le4$, or there exists $V_3\supset V_1$ with
$\sigma(V_3,V_3,V_{10})=0$. We may then conclude for the singular locus
of~$X_1^\sigma$.
\end{proof}

As in the previous case of~$\cD^{3,3,10}$ (see \eqref{eqD3310}), we may define the subvariety
\begin{equation}
\label{eq:1,6,10}
\cD^{1,6,10}\coloneqq\setmid{[\sigma]\in\cM}{\exists [V_1\subset V_6]\quad
\sigma(V_1,V_6,V_{10})=0}
\end{equation}
in the GIT quotient~$\cM$. Again by a parameter count, this subvariety is an
irreducible divisor of~$\cM$, and for $\sigma$ general in this divisor, there
exists a unique flag $[V_1\subset V_6]$ satisfying the degeneracy condition; indeed this variety is dominated by the $119$-dimensional total space of the vector bundle $\bw3 V_1^\perp + V_1 \wedge \bw2 V_6^\perp\subset \bw3 V_{10}^\vee$ over $\Flag(1,6,V_{10})$, which projects to a $\GL(V_{10})$-invariant hypersurface inside $\bw3 V_{10}^\vee$. We
can compute that the degree of the $\SL(V_{10})$-invariant hypersurface is
equal to $990$ (see Appendix \autoref{m2:degree_SLV10_divisors}), therefore this divisor
is indeed different from~$\cD^{3,3,10}$. These two divisors are related to the
non-smoothness of $X^\sigma_1$.

\begin{proposition}
\label{propsingD16}
The locus of trivectors~$[\sigma]$ for which $X^\sigma_1\subset \bP(V_{10})$
becomes singular is the union of two divisors $\cD^{1,6,10}\cup \cD^{3,3,10}$
in~$\cM$.
\begin{itemize}
\item If $[\sigma]\in \cD^{1,6,10}$ is general such that $\sigma(V_1,V_6,V_{10})=0$, then $\Sing(X_1^\sigma)=\set{[V_1]}$.
\item If $[\sigma]\in \cD^{3,3,10}$ is general such that $\sigma(V_3,V_3,V_{10})=0$, then $\Sing(X_1^\sigma)=\bP(V_3)$.
\end{itemize}
\end{proposition}

We will see that the divisor~$\cD^{1,6,10}\subset\cM$ maps birationally onto the Heegner
divisor $\cD_{24}\subset\cP$, which we will study in detail in \autoref{sec:D24}.

%\subsection{The incidence variety $I_{1,6}^\sigma$}
\label{sec:hodge_X1}

In this section, we will suppose that the trivector $\sigma$ does not
lie in~$\cD^{3,3,10}\cup\cD^{1,6,10}$, so all three
varieties~$X^\sigma_1$,~$X^\sigma_3$, and~$X^\sigma_6$ are smooth of expected
dimension. The Peskine variety $X^\sigma_1\subset \bP(V_{10})$ has many
interesting geometric aspects.  Firstly, it is a degree $15$ sixfold which is
Fano of index $3$ (this is obtained by studying a locally free resolution of the structure
sheaf, for example see \cite[Section 4.3]{bfm}; for the degree of $X_1^\sigma$ see the tables at the end of \cite{dfmr}). Denote by $h$ the natural
polarization on $X^\sigma_1$.  We have the following result from~\cite{han2}.
\begin{proposition}[Han]\label{prop:univ_palatini}
% Consider a trivector~$\sigma$ such that~$X_1^\sigma$ and~$X_6^\sigma$ are both
% smooth (i.e.\ when $\sigma\notin \cD^{3,3,10}\cup\cD^{1,6,10}$).
Let~$\sigma$ be a general trivector. For $[V_6]$ general in the Debarre--Voisin
variety $X^\sigma_6$, the intersection of $\bP(V_6)$ and $X^\sigma_1$ inside
$\bP(V_{10})$ is a {\em Palatini threefold}, that is, a smooth degree $7$
threefold in $\bP^5$ which is a scroll over a smooth cubic surface.
Otherwise stated, there is a~$7$-dimensional incidence variety~$I^\sigma_{1,6}$
called the {\em universal Palatini variety} and a diagram
\begin{equation}
\label{eq:I16}
\begin{tikzcd}
&I^\sigma_{1,6}\coloneqq\setmid{[V_1\subset V_6]\in \Flag(1,6,V_{10})}{[V_1]\in
X^\sigma_1,[V_6]\in X^\sigma_6} \ar[ld,"p"']\ar[rd,"q"]&\\
X^\sigma_1&&X^\sigma_6,
\end{tikzcd}
\end{equation}
where the general fiber of $q$ is a Palatini threefold.
\end{proposition}

We would like to use this correspondence to relate the Hodge structures
of~$X^\sigma_1$ and~$X^\sigma_6$.  However, \autoref{prop:univ_palatini} only holds for~$\sigma$ general, in which case the
fiber of~$q$ is a Palatini threefold only for general~$[V_6]\in X_6$. {\em A
priori} the incidence variety~$I^\sigma_{1,6}$ might not be smooth, or not have
the expected dimension at all. It turns out that it is not smooth but
has the expected dimension generically, and more precisely whenever~$X_1^\sigma$ and~$X_6^\sigma$ are
smooth. Therefore we can still use the correspondence~$q_*p^*$ in families. More precisely,
the incidence variety~$I_{1,6}^\sigma$ can also be seen as a subvariety
of~$X_1^\sigma\times X_6^\sigma$, so when it has the expected dimension, it
will give a well-defined cohomology class~$[I^\sigma_{1,6}]\in
H^6(X_1^\sigma\times X_6^\sigma)$. Thus by abuse of notation, we will
write~$q_*p^*$ for ${\pr_2}_*\big([I_{1,6}^\sigma]\cdot \pr_1^*(-)\big)$ and
similarly~$p_*q^*$ for ${\pr_1}_*\big([I_{1,6}^\sigma]\cdot \pr_2^*(-)\big)$.

First we show that~$I^\sigma_{1,6}$ is not smooth for a general~$\sigma$.  We
introduce a resolution of~$I^\sigma_{1,6}$ that can be realized as the
zero-locus of some section on a flag variety, whose smoothness can then be
deduced for general~$\sigma$. Consider the diagram
\begin{equation}
\label{eq:tildeI16}
\begin{tikzcd}
&\tilde I^\sigma_{1,6}\coloneqq\setmid{
\begin{array}{c}
[V_1\subset V_3 \subset V_6]\\
\in \Flag(1,3,6,V_{10})
\end{array}}
{\begin{array}{c}
\sigma(V_1,V_3,V_{10})=0,\\
\sigma|_{V_6}=0
\end{array}} \ar[ldd]\ar[rdd]\ar[d,"s"]&\\
& I^\sigma_{1,6}\ar[ld,"p"]\ar[rd,"q"']\\
X^\sigma_1&&X^\sigma_6.
\end{tikzcd}
\end{equation}
In other words, apart
from the pair~$[V_1\subset V_6]$, we introduce the extra information of a
subspace~$V_3$ contained in both $V_6$ and the kernel of the
form~$\sigma(V_1,-,-)$ (which immediately ensures that~$\rank\sigma(V_1,-,-)\le
6$).

\begin{proposition}
\label{propprojection16}
Let~$\sigma$ be a general trivector with both~$X_1^\sigma$ and~$X_6^\sigma$
smooth. The variety~$\tilde I^\sigma_{1,6}$ defined
in~\autoref{eq:tildeI16} is smooth of expected dimension~$7$ and the
projection~$s\colon\tilde I^\sigma_{1,6}\to I^\sigma_{1,6}$ obtained by
forgetting the subspace~$V_3$ is an isomorphism exactly on the complement of
a~$5$-dimensional subvariety of~$\tilde I_{1,6}^\sigma$. In
particular~$I^\sigma_{1,6}$ is not smooth.
\end{proposition}
\begin{proof}
The variety~$\tilde I^\sigma_{1,6}$ is defined inside~$\Flag(1,3,6,V_{10})$ as
the zero-locus of~$\sigma$ viewed as a section of the vector
bundle~$(\bw3 \cU_6 + \cU_1 \wedge \cU_3 \wedge V_{10})^\vee$, which is a quotient of $\bw3 V_{10}^\vee$ and thus it is globally generated; more explicitly, it is the dual of the subbundle $ \bw3 \cU_6 + \cU_1 \wedge \cU_3 \wedge V_{10}$ inside $\bw3 V_{10}$, and it can be constructed as an extension of $\bw3\cU_6^\vee $ by $ \big(\cU_1\wedge(\cU_3/\cU_1)\wedge(V_{10}/
\cU_6)\big)^\vee $. Therefore it is smooth of expected dimension~$7$ for a
general~$\sigma$.

The locus where the projection~$\tilde I^\sigma_{1,6}\to
I^\sigma_{1,6}$ is not an isomorphism is precisely above
those~$[V_1\subset V_6]$ where the kernel~$K_4$ of~$\sigma(V_1,-,-)$ is
contained in~$V_6$, in which case the fiber is a projective
plane~$\bP\big((K_4/V_1)^\vee\big)$ parametrizing~$V_3$ with~$V_1\subset V_3
\subset K_4$. We may look at the locus of such~$[V_1\subset K_4\subset V_6]$
inside the flag variety~$\Flag(1,4,6,V_{10})$: this is again the zero-locus
of~$\sigma$ viewed as a section of a certain homogeneous bundle, so for
general~$\sigma$ we get the~$3$-dimensional smooth subvariety
\[
Z \coloneqq\setmid{[V_1\subset K_4\subset
V_6]}{\sigma(V_1,K_4,V_{10})=0,[V_6]\in X_6^\sigma}\subset\Flag(1,4,6,V_{10}).
\]
Since the kernel~$K_4$ is uniquely determined by~$V_1$, the subvariety~$Z$
projects injectively onto its image in
$I^\sigma_{1,6}\subset\Flag(1,6,V_{10})$. The preimage in~$\tilde
I^\sigma_{1,6}$ therefore has dimension~$5$. This means
that the projection~$\tilde I^\sigma_{1,6}\to I^\sigma_{1,6}$ is a small
contraction for general~$\sigma$, so~$I^\sigma_{1,6}$ cannot be smooth.
\end{proof}
\begin{remark}
For general~$\sigma$, the~$3$-dimensional subvariety~$Z$ dominates a divisor
in~$X^\sigma_6$. This divisor has class~$10H$, which can be shown by computing
the degree of the pullback of the polarization~$H$. Note that this is a
canonically defined effective divisor in~$X_6^\sigma$, which could be useful in
constructing compactifications of the moduli space $\cM^{(2)}_{22}$.
\end{remark}

Now we show that~$I^\sigma_{1,6}$ always has the expected dimension~$7$. We
first state some lemmas.

\begin{lemma}
\label{lemma:unstable}
Let~$\sigma\in\bw3V_{10}^\vee$ be a trivector. If there is a subspace $V_7$
such that $\sigma|_{V_7}$ vanishes, then $\sigma$ is unstable with respect to
the $\SL(V_{10})$-action. 
\end{lemma}
\begin{proof}
This can easily be verified using the Hilbert--Mumford criterion. Choose a basis
$(e_1,\dots,e_{10})$ of~$V_{10}$ such that $V_7$ is the subspace
$\inner{e_1,\dots,e_7}$ and consider the $1$-parameter subgroup $\lambda\colon
\bC^*\to \SL(V_{10})$ given by
\[
t\mapsto \operatorname{diag}(t^3,t^3,t^3,t^3,t^3,t^3,t^3,t^{-7},t^{-7},t^{-7}).
\]
Then $\sigma$ has only negative weights with respect to this $1$-parameter
subgroup and is therefore unstable.
\end{proof}
\begin{lemma}\label{lemma:palatini}
Let~$[\sigma]\in\cM$. If the fiber of~$q$ above~$[V_6]$ is not of
dimension~$3$, there is a flag~$V_4\subset V_6\subset V_7$ such
that~$\sigma(V_4,V_7,V_7)=0$. This gives a
plane~$\bP\big((V_7/V_4)^\vee\big)$ contained in~$X_6^\sigma$, necessarily
Lagrangian.
\end{lemma}
\begin{proof}
This is a restatement of the results on Palatini threefolds obtained by
Fania--Mezzetti in~\cite{fm}. For each~$[V_6]$, since~$\sigma$ vanishes on~$V_6$,
we get an induced linear map
\[
\varphi_\sigma\colon V_{10}/V_6\to\bw2V_6^\vee.
\]
This map is injective: if some~$V_7/V_6$ is mapped to~$0$, the
trivector~$\sigma$ would vanish on~$V_7$ which is impossible by
\autoref{lemma:unstable}.

Therefore we have a~$4$-dimensional subspace of~$\bw2V_6^\vee$, or equivalently
a~$3$-dimensional projective subspace~$\Delta\subset\bP\big(\bw2V_6^\vee\big)$.
Inside~$\bP\big(\bw2V_6^\vee\big)$, there are two~$\SL(V_6)$-invariant orbits
given by the discriminant hypersurface~$\Gr(2,V_6)^*$---which is the Pfaffian
cubic---and the Grassmannian~$\Gr(2,V_6^\vee)\simeq\Gr(4,V_6)$, representing
skew-symmetric forms of ranks~$\le4$ and~$\le2$ respectively.
The fiber of~$q$ above~$[V_6]$ is precisely the union of degeneracy loci
in~$\bP(V_6)$ for the family of skew-symmetric forms parametrized by~$\Delta$, which is a Palatini scroll (see \cite{fm}). Therefore one obtains a rational morphism $\rho$ from $\Gr(3,\bw2 V_6^\vee)$ to the moduli space of Palatini scrolls, sending $\Delta$ to $q^{-1}([V_6])$. By \cite[Theorem~4.3 and Theorem~4.9]{fm} $\rho$ is well defined (and thus $q^{-1}([V_6])$ has dimension three) on the complementary of the locus inside $\Gr(3,\bw2 V_6^\vee)$ where:
%The classification result from \cite[Theorem~4.3 and Theorem~4.9]{fm} tells us that this degeneracy locus is of expected dimension~$3$ except in the following two cases
\begin{itemize}
\item $\Delta$ is entirely contained in the discriminant
hypersurface~$\Gr(2,V_6)^*$ (\cite[Theorem~4.3]{fm});
\item $\Delta$ is not contained in~$\Gr(2,V_6)^*$ but its intersection with the
Grassmannian~$\Gr(2,V_6^\vee)$ contains a line or a conic (\cite[Theorem~4.9]{fm}).
\end{itemize}
Moreover by a result of Manivel--Mezzetti~\cite[Corollary 11]{mm}, in the first
case the projective $3$-space~$\Delta$ will necessarily intersect the
Grassmannian~$\Gr(2,V_6^\vee)$.  Therefore in both cases,
the~$3$-space~$\Delta$ intersects the Grassmannian~$\Gr(2,V_6^\vee)$, which
means that there is a~$V_7\supset V_6$ whose image is decomposable: we have
$\varphi_\sigma(V_7/V_6)=f_1\wedge f_2$ where~$f_1,f_2\in V_6^\vee$ are linear
forms. The common kernel~$V_4\subset V_6$ of~$f_1$ and~$f_2$ therefore
satisfies the desired property~$\sigma(V_4,V_7,V_7)=0$.
\end{proof}

\begin{proposition}
Let~$[\sigma]\in\cM$ be a trivector such that~$X_1^\sigma$ and~$X_6^\sigma$ are
both smooth (that is, $[\sigma]\notin \cD^{3,3,10}\cup\cD^{1,6,10}$). The
variety~$I^\sigma_{1,6}$ defined in~\autoref{eq:I16} has only one irreducible
component of expected dimension~$7$, and this component is reduced.
\end{proposition}
% TODO
\begin{proof}
% We study the fibres of the fibration~$q$.

For a trivector~$\sigma\notin\cD^{3,3,10}\cup\cD^{1,6,10}$, since~$X^\sigma_6$
is hyperkähler of dimension~$4$, it contains only finitely many planes of the
form $\bP\big((V_7/V_4)^\vee\big)$, since any such plane is necessarily
Lagrangian hence rigid.  We saw that for any~$[V_6]$ away from these planes,
the fiber is of expected dimension~$3$ and is generically smooth, so the
irreducible component of~$I^\sigma_{1,6}$ that dominates~$X_6^\sigma$ is
reduced of expected
dimension~$7$. On the other hand, the preimage of each Lagrangian plane
$\bP\big((V_7/V_4)^\vee\big)$ has dimension~$\le 6$: otherwise
for any~$[V_6]\in \bP\big((V_7/V_4)^\vee\big)$ and any~$V_1\subset V_6$ we
have~$[V_1]\in X_1^\sigma$, then~$\bP(V_7)$ would be entirely contained
in~$X^\sigma_1$, which is impossible because~$X^\sigma_1$ is assumed to be
smooth of dimension~$6$.
\end{proof}

Therefore whenever~$X^\sigma_1$ and~$X_6^\sigma$ are both smooth,
the variety~$I^\sigma_{1,6}$ has one unique reduced component of expected
dimension~$7$.
It defines a class on the product~$X^\sigma_1\times X_6^\sigma$ with correct
codimension, and we can thus talk about the morphisms $q_*p^*$ and
$p_*q^*$ between Hodge structures given by this correspondence.

The degeneracy condition $\sigma(V_4,V_7,V_7)=0$ from \autoref{lemma:palatini}
defines a third divisor in the GIT moduli space~$\cM$ of trivectors: consider
the subvariety
\begin{equation}\label{eq:D477}
\cD^{4,7,7}\coloneqq\setmid{[\sigma]\in\cM}{\exists[V_4\subset V_7]\quad
\sigma(V_4,V_7,V_7)=0}.
\end{equation}
By a parameter count, it is an irreducible divisor in~$\cM$, and for a general
$\sigma$ in this divisor, there exists a unique flag $[V_4\subset V_7]$
satisfying the degeneracy condition. Again we can compute the degree of the
$\SL(V_{10})$-invariant hypersurface which is equal to $5500$ (see
Appendix \autoref{m2:degree_SLV10_divisors}). Hence this divisor is different from
both~$\cD^{3,3,10}$ and~$\cD^{1,6,10}$.
It will be studied in detail in the next section.
% We will see in the next section that
% this divisor corresponds to the Heegner divisor~$\cD_{28}$.

\section{The Heegner divisor of degree 28 and integral Hodge structures}
\label{sec:D28}

\subsection{The discriminant}\label{sec:disc28}

In \autoref{sec:hodge_X1}, we defined the divisor~$\cD^{4,7,7}$ in~$\cM$
given by the degeneracy condition~\autoref{eq:D477}
\[
\exists[V_4\subset V_7]\quad \sigma(V_4,V_7,V_7)=0.
\]
In \autoref{propsingD33}, we showed that the complement of~$\cM^\smooth$ is the
divisor~$\cD^{3,3,10}$. Recall that $\cM^\smooth$ is the set of $\GL(V_{10})$-classes of trivectors $[\sigma]$ such that $X_6^\sigma$ is smooth of dimension 4. Therefore, when~$[\sigma]$ lies in $\cD^{4,7,7}\setminus
\cD^{3,3,10}$,
the corresponding $X^\sigma_6$ is a smooth hyperkähler fourfold.
Moreover, given the flag $[V_4\subset V_7]$, we see that every $V_6$ in
$\Gr(2,V_7/V_4)=\bP\big((V_7/V_4)^\vee\big)$ is in~$X^\sigma_6$. Indeed, for any $V_6$ such that $V_4\subset V_6\subset V_7$, we have $\bw3 V_6=V_4\wedge \bw2 V_6\subset V_4\wedge \bw2 V_7$, and thus $\sigma|_{V_6}=0$.  So the
hyperkähler fourfold $X^\sigma_6$ contains a plane $P$, necessarily Lagrangian. Note that we have the following equivalent degeneracy condition.
\begin{lemma} 
\begin{equation}\label{eq:377}
\sigma \in \cD^{4,7,7} \mbox{ if and only if }\exists[V_3\subset V_7]\mid \sigma(V_3,V_7,V_7)=0.
\end{equation}
\end{lemma}

\begin{proof}
The equivalence can either be deduced directly from the descriptions of the $\GL(V_7)$-orbits in $\bw3V_7^\vee$ from the last section, or be verified by looking inside $\bP(V_7/V_3)$ for the vanishing condition $\sigma(U_{4/3}, V_7,V_7)=0$, which is the zero locus of a section of the rank-3 vector bundle $\cU_{4/3}^\vee\otimes\bw 2\cU_{7/4}^\vee$ with top Chern class~1, therefore there exists some $V_4$ such that $\sigma(V_4,V_7,V_7)=0$.
\end{proof}

We will first determine the discriminant of the corresponding
Noether--Lefschetz/Heegner divisor. In fact, we will show that a
Debarre--Voisin fourfold~$X_6^\sigma$ containing a linearly embedded Lagrangian plane (with respect to the Plücker polarization) is always in the
family~$\cC_{28}$.
This shows in particular that the divisor~$\cD^{4,7,7}$ is mapped
onto~$\cC_{28}$ via the modular map~$\m$, and that {\em any} Lagrangian plane
contained in a Debarre--Voisin fourfold is of the above form.

We begin by recalling the following general result of Hassett--Tschinkel on Lagrangian planes contained in hyperkähler fourfolds of $\KKK^{[2]}$-type (see~\cite[Section~5]{ht}, in particular the last part of the proof of Theorem~21).

\begin{proposition}[Hassett--Tschinkel]\label{prop:ht}
Let $X$ be a smooth hyperkähler fourfolds of $\KKK^{[2]}$-type and let~$P$ be a Lagrangian plane contained in~$X$. Let~$l\in H^6(X,\bZ)$ be the class of a line contained in the plane~$P$. Then there exists a unique class~$\lambda\in H^2(X,\bZ)$ satisfying the property
\begin{equation}\label{eq:lambda}
\forall x\in H^2(X,\bZ)\quad 2x\cdot l=\q(x,\lambda),\end{equation}
where $\q$ is the Beauville--Bogomolov--Fujiki form. Moreover, the class~$\lambda$ is of square~$\q(\lambda,\lambda)=-10$ and divisibility~$2$, and we have the following relation
\[
[P]=\tfrac1{20}q^\vee+\tfrac18\lambda^2,
\]
where~$q^\vee\in H^4(X,\bQ)$ is the distinguished algebraic class~$\frac56c_2(X)$ that satisfies~$q^\vee\cdot x_1\cdot x_2=25\cdot\q(x_1,x_2)$ for all $x_1,x_2\in H^2(X,\bZ)$, and $q^\vee\cdot q^\vee=575$.
\end{proposition}
We explain briefly the idea of the proof. For a very general $X$ containing a plane $P$, all Hodge classes of type $(2,2)$ are generated by the two classes $q^\vee$ and $\lambda^2$, so one can write $[P]=aq^\vee+b\lambda^2$. We then have three variables: the value $\q(\lambda,\lambda)$ and the two coefficients $a,b$. Using the following intersection numbers, we can obtain a system of Diophantine equations which turns out to admit a unique solution.
\begin{itemize}
\item Since $P$ is Lagrangian, we have $\cN_{P/X}\simeq \Omega_P$. So we may compute the self-intersection number
\[[P]^2=\int_Pc_2(\cN_{P/X})=\int_Pc_2(\Omega_P)=3.\]
\item Consider the normal sequence
\[
0\to \cT_P\to \cT_{X}|_P\to \cN_{P/X} \to 0.
\]
Using the splitting principle and the fact that $\cN_{P/X}\simeq \Omega_P$, we compute $c_2(X)\cdot[P]=\int_P c_2(\cT_p\oplus\Omega_P)=-3$.
\item The restriction $\lambda|_P$ is a multiple of the class of a line, and its degree is equal to the intersection number $\lambda\cdot l$ on $X$. In other words, we have $\lambda|_P=(\lambda\cdot l)l\in H^2(P,\bZ)$. So we may compute that
\[
\lambda^2\cdot[P]=(\lambda|_P)^2=(\lambda\cdot l)^2=\q(\lambda,\lambda)^2/4.
\]
\end{itemize}

We prove an extra lemma.
\begin{lemma}\label{lemma:P_unique}
Let $X$ be a smooth hyperkähler fourfolds of $\KKK^{[2]}$-type and $\lambda\in H^2(X,\bZ)$ be a class of square~$-10$ and divisibility~$2$. Moreover, let~$H$ be a polarization on~$X$. Then there is at most one linearly embedded plane~$P$ (with respect to~$H$) whose associated~$(-10)$-class is equal to~$\lambda$.
\end{lemma}
\begin{proof}
For two distinct linearly embedded planes~$P$ and~$P'$, their intersection can
be empty, a point, or a line~$L$. We show that the last case is not possible in
general: we have an exact sequence
\[
0\to \cT_P\to \cT_{X}|_P\to \cN_{P/X}\simeq \Omega_P \to 0,
\]
which, when restricted to $L$, gives
\[
0\to \cT_P|_L\simeq\cO_L(1)\oplus\cO_L(2)\to \cT_{X}|_L\to \cO_L(-1)\oplus\cO_L(-2) \to 0.
\]
If $P\cap P'=L$, the other normal bundle $\cN_{P'/L}\simeq\cO_L(1)$ should be a
subbundle of the quotient, which is not possible. In conclusion, the
intersection number $[P]\cdot[P]'$ is either 0 or 1.

If we now assume that the $(-10)$-classes associated with~$P$ and~$P'$ are
both~$\lambda$, then we may compute the intersection number using
\[
[P]\cdot[P']=(\tfrac1{20}q^\vee+\tfrac18\lambda^2)^2=[P]^2=3,
\]
which leads to a contradiction.
\end{proof}

\begin{proposition}
\label{propD28}\leavevmode
\begin{enumerate}
\item A smooth Debarre--Voisin fourfold~$X_6^\sigma$ containing a Lagrangian plane~$P$ is
always in the family~$\cC_{28}$.
\item Consequently, for $[\sigma]$ very general in the divisor $\cD^{4,7,7}$, the corresponding transcendental sublattice
$H^2(X^\sigma_6,\bZ)_\trans$ is of discriminant~$-28$. The divisor $\cD^{4,7,7}$ is mapped birationally onto the Noether--Lefschetz
divisor~$\cC_{28}$ by the modular map~$\m$, and then onto the Heegner
divisor~$\cD_{28}\subset\cP$ by the period map~$\p$.
\item A very general $X_6^\sigma$ in the family $\cC_{28}$ contains exactly one Lagrangian plane.
\item Finally, any Lagrangian plane~$P$ contained in a smooth Debarre--Voisin
fourfold is of the form~$\bP\big((V_7/V_4)^\vee\big)$, for a flag~$[V_4\subset
V_7]$ satisfying the degeneracy condition $\sigma(V_4,V_7,V_7)=0$.
% $\cD_{28}$ via $\p\circ\m\colon \cM\dashrightarrow \cP$.

\end{enumerate}
\end{proposition}

\begin{proof}
For statement~(1), let~$\sigma$ be such that~$X^\sigma_6$ is smooth of dimension~$4$ and contains a Lagrangian plane~$P$.
Let $H$ be the canonical polarization on $X^\sigma_6$ of square~22 and divisibility~2.
Let $l$ be the class of a line contained in the plane~$P$.
Consider the class~$\lambda\in H^2(X^\sigma_6,\bZ)$ given by \autoref{prop:ht}. Since $H\cdot l=1$, we have $\q(H,\lambda)=2$. Therefore the intersection matrix between $H$ and $\lambda$ is
\[ \begin{pmatrix} 22&2\\2&-10 \end{pmatrix},\]
with determinant~$-224=-7\cdot 2^5$.

We study the saturation of the sublattice~$\bZ H+\bZ\lambda$.
Since the discriminant of the lattice $H^2(X_6^\sigma,\bZ)$ is~2, and since
both~$H$ and~$\lambda$ are primitive of divisibility~2, the images of
$\frac12H$ and~$\frac12\lambda$ in the discriminant group are equal. The
class $\frac12(H+\lambda)$ is therefore integral. We may consider the
sublattice generated by $\frac12(H+\lambda)$ and~$\lambda$, which has
intersection matrix
\begin{equation}\label{eq:intersection_matrix_H_lambda}
\begin{pmatrix} 4&-4\\-4&-10 \end{pmatrix}.
\end{equation}
The discriminant of this lattice is equal to $-56=-7\cdot 2^3$. Suppose that it
is not saturated, then we can find a class
$a\cdot\frac12(H+\lambda)+b\lambda=2x$
with~$\gcd(a,b)=1$, where $x$ is still integral. We may compute that
$q(x)=a^2-2ab-\frac52b^2$ which is an integer,
so~$b$ is even,~$a$ is odd, and the square $q(x)$ is an odd number. This
contradicts the fact that $H^2(X_6^\sigma,\bZ)$ is an even lattice.
Therefore, $\bZ\frac12(H+\lambda)+\bZ\lambda$ is the saturation of $\bZ
H+\bZ\lambda$ and has discriminant~$-56$.
Since its discriminant is always twice that of its orthogonal by Lemma~\ref{lemma:disc_K_two_times}, we get a member of the family~$\cC_{28}$.

We have already checked that $\cD^{4,7,7}$ is distinct from $\cD^{3,3,10}$, by
computing the degree of the corresponding $\SL(V_{10})$-invariant hypersurface
in $\bP(\bw3 V_{10}^\vee)$. Therefore, a very general $[\sigma]$
in~$\cD^{4,7,7}\setminus\cD^{3,3,10}$ lies in $\cM^\smooth$ and gives rise to a
smooth $X^\sigma_6$. Since there exists a distinguished flag $[V_4\subset V_7]$
such that $\sigma(V_4,V_7,V_7)=0$, the corresponding Debarre--Voisin
variety~$X_6^\sigma$ contains a Lagrangian plane~$\bP\big((V_7/V_4)^\vee\big)$
and therefore is a member of the family $\cC_{28}$ by~(1).
Thus the divisor~$\cD^{4,7,7}$ is mapped to~$\cC_{28}$ in the moduli space. By
Lemma~\ref{lemma:normality-divisors-birational}, we conclude that $\m$ restricts to a birational
map from $\cD^{4,7,7}$ to $\cC_{28}$. This shows the statement~(2).

The divisor $\cD^{4,7,7}$ being mapped birationally onto $\cC_{28}$ shows that
a very general member~$X_6^\sigma$ of the family~$\cC_{28}$ indeed contains a plane.
Moreover, for such $X_6^\sigma$, the Picard group is of rank~$2$ and has intersection matrix as in
\eqref{eq:intersection_matrix_H_lambda}, so there is only one class~$\lambda$
satisfying $q(H,\lambda)=2$ and $q(\lambda,\lambda)=-10$.%
\footnote{Let $\lambda'=aH+b\lambda$ be a class satisfying
$2=\q(H,\lambda')=22a+2b$ and $-10=\q(\lambda',\lambda')=22a^2+4ab-10b^2$, then
we may solve that $(a,b)=(\frac2{11},-1)$ or $(0,1)$. For $\lambda'$ to be
integral, only the second case is possible, so $\lambda'$ is indeed equal to $\lambda$.}
By Lemma~\ref{lemma:P_unique}, this shows that a very general~$X_6^\sigma$ in
$\cC_{28}$ contains exactly one Lagrangian plane, which must be of the
form~$\bP\big((V_7/V_4)^\vee\big)$.

Finally, let us prove the last point. Planes inside $X_6^\sigma$ are also planes inside $\Gr(6,V_{10})$. These planes can be of two types. The first type of planes are given by $\Gr(2,V_7/V_4)\cong \bP^2\subset \Gr(6,V_{10})$, for fixed linear subspaces $V_4\subset V_7 \subset V_{10}$; for such a plane to be contained inside $X_6^\sigma$ a necessary and sufficient condition is that $\sigma(V_4,V_7,V_7)=0$, which is exactly the condition in the statement.
The second type of planes is given by $\bP(V_8/V_5)\cong \bP^2\subset \Gr(6,V_{10})$, for fixed linear subspaces $V_5\subset V_8 \subset V_{10}$; for such a plane to be contained inside $X_6^\sigma$ a necessary and sufficient condition is that $\sigma(V_5,V_5,V_8)=0$. However, the existence of such a flag $V_5\subset V_8$ implies that there exists $V_3\subset V_5$ such that $\sigma(V_3,V_{3},V_{10})$, i.e. $X_3^\sigma$ and hence $X_6^\sigma$ is singular. To prove that in this case such a $V_3$ exists, notice that $V_3$ can be seen as a point in the six dimensional variety $\Gr(3,V_5)$, and the condition $\sigma(V_3,V_{3},V_{10})=0$ (knowing that $\sigma(V_5,V_5,V_8)=0$) becomes $\sigma(V_3,V_3,V_{10}/V_8)=0$; thus the set of $[V_3]\in \Gr(3,V_5)$ satisfying $\sigma(V_3,V_3,V_{10})=0$ is the zero locus of a section of the bundle $(\wedge^2 \cU_3 \otimes \cO_{\Gr}^2)^\vee$, which is of rank six and never empty (top Chern class is different from zero).
%
\begin{comment}
Finally, for each Lagrangian plane~$P$ contained in a smooth Debarre--Voisin
fourfold, we may consider a generic deformation which preserves the Lagrangian
plane, using the results of Voisin~\cite{voisinlagrangian} on deformations of
Lagrangian subvarieties.
In this case, a very general members of the deformation has Picard rank 2, and
the Lagrangian plane it contains is indeed of the
form~$\bP\big((V_7/V_4)^\vee\big)$ for a certain flag~$[V_4\subset V_7]$.
As this is a deformation of the pair $(X_6^\sigma,P)$, the original plane~$P$
in the central fiber is necessarily also of this form: this follows
from the properness of the flag variety $\Flag(4,7,V_{10})$ and the fact that
$\sigma(V_4,V_7,V_7)=0$ is a closed condition. This concludes the
proof.
\end{comment}
\end{proof}

\begin{comment}
The last statement of Proposition~\ref{prop:D28} has an alternative proof:
a plane contained in the Grassmannian $\Gr(6,V_{10})$ is either of the form
$\bP\big((V_7/V_4)^\vee\big)$ for a flag $V_4\subset V_7$, or $\bP(V_8/V_5)$
for a flag $V_5\subset V_8$ (see for example \cite[Exercise~6.9]{harrisbook}).
So it suffices to show that a Debarre--Voisin
variety~$X^\sigma_6$ containing a plane of the second type is not smooth. In
this case, the trivector~$\sigma$ satisfies the degeneracy condition
$\sigma(V_5,V_5,V_8)=0$. To show that $X^\sigma_6$ is singular, we look for
a~$V_3$ contained in~$V_5$ satisfying the extra degeneracy condition
$\sigma(V_3,V_3,V_{10}/V_8)=0$. Equivalently, we take the
Grassmannian $\Gr(3,V_5)$ and consider the section of the rank-6 bundle
$\big(\bw2\cU_3^\vee\big)^{\oplus2}$ induced by $\sigma$.
Since $\Gr(3,V_5)$ is of dimension 6 and the vector bundle has top Chern
class~1, any section has a non-empty zero-locus, so such $V_3$ indeed exists.
\end{comment}

% particular this shows that the
% modular map~$\m$
% % extended period map~$\tilde\p$
% restricts to a
% dominant rational map~$\cD^{4,7,7}\dashrightarrow \cC_{28}$
% onto the
% Noether--Lefschetz divisor of discriminant~$28$, which is in turn
% mapped to the Heegner divisor~$\cD_{28}\subset\cP$ in the period domain.
% As before, the normality of the period domain~$\cP$ allows us to deduce the
% birationality.

\begin{remark}
\label{rmqunirationality28}
This result implies that $\cD_{28}$ is unirational. Indeed, $\cD^{4,7,7}$ can
be seen as the quotient of the vector bundle $(\cU_4\wedge \cU_7\wedge
\cU_7)^\perp \subset \bigwedge^3 V_{10}^\vee \otimes \cO_F$ over the flag
variety $F\coloneqq\Flag(4,7,V_{10})$ by the natural action of the group
$\SL(V_{10})$. Concerning the possibility of associated K3 surfaces, the lattice in \autoref{eq:intersection_matrix_H_lambda} does not represent~$28$,%
\footnote{Suppose that there exist $a,b\in\bZ$ satisfying $4a^2-8ab-10b^2=28$, by modulo 7 we may verify that $a\equiv b\pmod 7$. So we can write $a=7a'+r$ and $b=7b'+r$. By modulo 49 we get $r^2\equiv 5\pmod 7$, which is impossible.}
so there is no associated K3 surface of degree~$28$.
This last conclusion can also be obtained using \cite[Theorem 3.1]{dhov}.
\end{remark}

% \begin{remark}
% \label{rmkintclasshalf}
% Recall that the transcendental sublattice of~$X_6^\sigma$ is defined as the
% orthogonal complement of the algebraic part~$H^{1,1}(X^\sigma_6,\bZ)$.
% For~$[\sigma]$ very general in $\cD^{4,7,7}$, the algebraic part is the
% saturation of the sublattice generated by~$H$ and~$\lambda$.  We shall see that
% the
% transcendental sublattice~$H^2(X^\sigma_6,\bZ)_\trans$ is of discriminant~$28$
% (\autoref{propD28}). Since $H^2(X^\sigma_6,\bZ)$ is of discriminant 2, the
% discriminant of the algebraic part is either~$14$ or~$56$.
% The sublattice $\bZ H+\bZ\lambda$ is therefore not saturated and the class
% $\frac{1}{2}(H+\lambda)$ is integral.
% The sublattice generated by $\frac12(H+\lambda)$ and~$\lambda$ has
% intersection matrix
% \[\begin{pmatrix} 4&-4\\-4&-10 \end{pmatrix}.\]
% Since~$H^2(X_6^\sigma,\bZ)$ is an even lattice, no primitive class can be
% further divisible by~$2$, so the algebraic part is generated by
% $\frac12(H+\lambda)$ and~$\lambda$.
% By reduction modulo $7$, we see that this
% lattice does not represent~$28$, so there is no associated K3 surface of
% degree~$28$. This last conclusion can also be obtained using \cite[Theorem
% 3.1]{dhov}.
% \end{remark}

\subsection{Integral Hodge structures of \texorpdfstring{$X^\sigma_3$}{X3}}

In this section we proceed to the proof of the Hodge isometries between the integral Hodge structures of $X_3^\sigma$ and $X_6^\sigma$. By Corollary \ref{cor:isometry_upto_constant} we know that~$q_*p^*$ is a constant multiple of an isometry. If we
can show that this constant is $-1$, then since the discriminants of the two
lattices are the same, this isometry will also be onto, and we may prove the following theorem.
\begin{theorem}\label{thm:36}
When~$\sigma$ is such that~$X_3^\sigma$ and $X_6^\sigma$ are both smooth (that
is, when~$[\sigma]\notin\cD^{3,3,10}$), the morphism $q_*p^*$ gives an
isomorphism of polarized integral Hodge structures
\begin{equation}
\label{eq:qp36}
q_*p^*\colon H^{20}(X^\sigma_3,\bZ)_\van\simto H^2(X^\sigma_6,\bZ)_\prim(-1),
\end{equation}
where the $(-1)$ means that $H^2(X^\sigma_6,\bZ)_\prim$ is endowed
with the quadratic form $-\q$.
\end{theorem}

To determine the constant which defines $q_*p^*$, we will use the argument of continuity:
the constant is the same over the moduli space, so it suffices to
compute its value over the Heegner divisor~$\cD_{28}$, where we have some
explicit Hodge classes to work with. We will
also prove the integral Hodge conjecture on $H^{20}(X_3^\sigma,\bZ)$ as a
corollary (\autoref{cor:HodgeConjX3}).

In order to prove \autoref{thm:36}, we will use the
correspondence $X^\sigma_3\overset p\longleftarrow I^\sigma_{3,6}\overset
q\longrightarrow X^\sigma_6$ from \autoref{eq:I36}.  The key
point is to show that $q_*p^*$ sends the intersection product to
$-\q$. For this, it is enough to prove that
\[\exists x\in H^{20}(X^\sigma_3,\bZ)_\van\setminus\set0\quad
x^2=-\q(q_*p^*x,q_*p^*x).\]
The existence of such an $x$ only depends on the integral cohomology, which is constant in families; thus we may perform the computation when $\sigma$ is general in the divisor $\cD^{4,7,7}$, as long as we assume that $X_3^\sigma$ and $X_6^\sigma$ are smooth (which is the same as asking that $\sigma \notin \cD^{3,3,10}$).
%By a continuity argument, we may specialize to the case of a general~$[\sigma]$ in the divisor~$\cD^{4,7,7}$, for which~$X_3^\sigma$ and~$X_6^\sigma$ remain smooth.

Let us begin with some preliminary results. For~$[\sigma]\in\cD^{4,7,7}$
with~$\sigma(V_4,V_7,V_7)=0$, denote by~$l$ the class of a line
contained in the plane~$P=\bP\big((V_7/V_4)\big)^\vee$.
% Let $l$ be the class of a line contained in the plane~$P$.
Such a line can be expressed as
\[l=\setmid{[V_6]\in X^\sigma_6}{V_5\subset V_6\subset V_7},\]
where $V_5$ is a subspace such that $V_4\subset V_5\subset V_7$.
The class
$z\coloneqq p_*q^*l\in H^{20}(X^\sigma_3,\bZ)$ is represented by the
subvariety
\begin{equation}\label{eq:Z}
Z\coloneqq\setmid{[V_3]\in \Gr(3,V_7)}{\ \dim(V_3\cap
V_5)\ge2}.
\end{equation}
Indeed $q^{-1}(l)=\{ [V_3\subset V_6]\mid V_5\subset V_6\subset V_7, [V_3]\in X_3^\sigma\}$, thus $[V_3]\in p(q^{-1}(l))$ if and only if $V_3\subset V_7$, $[V_3]\in X_3^\sigma$ and $\dim(V_3\cap V_5)\geq 2$; however, this last inequality implies that $\dim(V_3\cap V_4)\geq 1$, which implies automatically that $[V_3]\in X_3^\sigma$. Notice moreover that, by construction, for a generic point of $Z$ the subspace $V_6$ is uniquely determined by $V_6=V_5+V_3$, thus $p|_{q^{-1}l}$ is of degree one.

We may decompose the class $z$ as the sum of its vanishing part $z_0\in
H^{20}(X^\sigma_3,\bQ)_\van$ and its Schubert part $z_1\in
j^*H^{20}(\Gr(3,V_{10}),\bQ)$ according to the decomposition \autoref{eq:H20}.

\begin{lemma}
\label{lem2:36}
In the notation above, the Schubert part~$z_1$ of the class $z\in
H^{20}(X^\sigma_3,\bZ)$ has square $z_1^2=\tfrac{5}{11}$.
\end{lemma}

\begin{proof}
The class $j_*z$ is the Schubert class $\sigma_{443}$ on $\Gr(3,V_{10})$
represented by $Z$. We can compute $z\cdot j^*\sigma_{433}=1$ while $z\cdot
j^*\sigma_{abc}=0$ for the rest of the Schubert classes. The intersection
numbers appearing in the intersection matrix \eqref{int_matrix} allow us to completely determine $z_1$ in terms of the basis
$j^*\sigma_{abc}$: we get
\begin{equation}
\label{eq:z1}
\begin{aligned}
z_1=\tfrac1{11}(
&j^*\sigma_{730}-3j^*\sigma_{721}-j^*\sigma_{640}+2j^*\sigma_{631}+3
j^*\sigma_{622}\\+&j^*\sigma_{550}-j^*\sigma_{541}-5j^*\sigma_{532}+6
j^*\sigma_{442}+5j^*\sigma_{433}).
\end{aligned}\end{equation}
We may then compute its self-intersection number and find~$\tfrac5{11}$.
\end{proof}

To compute $z_0^2$, we will specialize the trivector further
so that $X_6^\sigma$, while still smooth, contains two disjoint planes $P$ and
$P'$ (the existence of such a trivector is showed in Remark \ref{existence_trivect_two_planes}). Denote by $\lambda$ and $\lambda'$ their corresponding $(-10)$-classes as
defined in \autoref{prop:ht}.  We have the following result.

\begin{lemma}
\label{lem1:36}
If $X_6^\sigma$ is smooth and contains two disjoint planes $P,P'$, then either
$\q(\lambda, \lambda')=2$ or $\q(\lambda, \lambda')=-2$.
\end{lemma}

\begin{proof}
As the two planes $P,P'$ are disjoint, we have
\[0=[P]\cdot [P']=\left(\tfrac{1}{20}q^\vee+\tfrac{1}{8}\lambda^2\right)\cdot
\left(\tfrac{1}{20}q^\vee+ \tfrac{1}{8}\lambda'^2\right),\]
by \autoref{prop:ht}. Using $q^\vee\cdot q^\vee=575$,
$q^\vee\cdot\lambda^2=25\cdot \q(\lambda,\lambda)=-250=q^\vee\cdot \lambda'^2$,
and~$\lambda^2\cdot \lambda'^2=2\q(\lambda,\lambda')^2+\q(\lambda,\lambda)\cdot
\q(\lambda',\lambda')$,
we find $\q(\lambda,\lambda')^2=4$, therefore $\q(\lambda,\lambda')=\pm2$.
\end{proof}

Following~\cite[Section~2]{dv}, we will consider the following two situations
for two disjoint planes $P,P'$ contained in $X^\sigma_6$:
\begin{enumerate}[leftmargin=5em,label={\bf Case \arabic*.},ref=Case \arabic*]
\item \label{case1}
We have $V_4\subset V_7$ and $V_4'\subset V_7'$ with $\dim(V_7\cap
V_7')=4$ and $V_4\cap V_4'=\set 0$. For a suitable choice
of basis $(e_0,\dotsc,e_9)$, we may set
$V_7=\inner{e_0,\dotsc,e_6}$, $V_7'=\inner{e_3,\dotsc,e_9}$,
$V_4=\inner{e_1,e_2,e_3,e_4}$, and $V_4'=\inner{e_5,e_6,e_7,e_8}$. Note
that $\dim(V_4\cap V_7')=\dim(V_4'\cap V_7)=2$.
\item \label{case2}
We have $V_4\subset V_7$ and $V_4'\subset V_7'$ with $\dim(V_7\cap
V_7')=4$ but $V_4\cap V_4'$ one-dimensional. In this case,
we may set
$V_7=\inner{e_0,\dotsc,e_6}$, $V_7'=\inner{e_3,\dotsc,e_9}$,
$V_4=\inner{e_0,e_1,e_2,e_3}$, and $V_4'=\inner{e_3,e_7,e_8,e_9}$.
\end{enumerate}
In both cases, the planes $P\coloneqq\bP\big((V_7/V_4)^\vee\big)$ and
$P'\coloneqq\bP\big((V'_7/V'_4)^\vee\big)$ are disjoint.
\begin{remark}
\label{existence_trivect_two_planes}
Note that the existence of $\sigma$ such that $X_6^\sigma$ is smooth and contains two disjoint planes was not proved in \cite{dv}, although
it can be verified using a computer:
we choose random trivectors~$\sigma$ that satisfy the vanishing
conditions as above, and check the smoothness of the hyperplane
section~$X_3^\sigma$. For example, the following trivectors with coefficients
in~$\set{0,\pm 1}$ suffice in the two cases:
\begin{gather*}
\scriptstyle
[056]+[037]-[237]+[047]+[157]+[257]+[267]-[018]-[128]-[148]\\[-.5em]
\scriptstyle
-[058]+[258]+[168]-[078]-[129]+[249]+[349]+[059]+[269]-[289]
\end{gather*}
and
\begin{gather*}
\scriptstyle
[456]+[017]+[027]+[147]-[057]+[067]+[167]-[267]+[018]\\[-.5em]
\scriptstyle
+[138]+[238]-[148]-[258]+[039]+[149]+[169]+[189]
,
\end{gather*}
where~$[ijk]$ stands for the form~$e_i^\vee\wedge e_j^\vee\wedge e_k^\vee$.
\end{remark}

We are now ready to prove \autoref{thm:36}.

\begin{proof}[Proof of \autoref{thm:36}]
For $[\sigma]$ very general in the divisor $\cD^{4,7,7}$, the Debarre--Voisin
fourfold~$X_6^\sigma$ has Picard rank~$2$ since the rank of $H^2(X_6^\sigma,\bZ)\cap H^{1,1}(X_6^\sigma)$ is $2$ (recall that the image of $\cD^{4,7,7}$ in the space of periods is an Heegner divisor). Therefore the space
$H^2(X^\sigma_6,\bZ)_ \prim\cap H^{1,1}(X^\sigma_6)$ of primitive algebraic classes
has rank~$1$. Using the intersection matrix
\autoref{eq:intersection_matrix_H_lambda}, we see that it is generated by the
class $\frac{1}{2}(H-11\lambda)$. As proved in \cite[Theorem 1.1]{dv}, the map
\[q_*p^*\colon H^{20}(X_3^\sigma,\bQ)_\van\to H^2(X_6^\sigma,\bQ)_\prim\]
is an isomorphism of rational Hodge structures, so there is some rational number
$c\in\bQ$ such that $q_*p^*z_0=c(H-11\lambda)$; recall for this that $z_0$ is the vanishing part of $z=p_*q^*(l)$, $l$ being the class of a line in $P$.

We now specialize the trivector~$\sigma$ so that $X_6^\sigma$ contains two
disjoint planes $P$ and $P'$. We have two $(-10)$-classes $\lambda,\lambda'$ in
$H^2(X_6^\sigma,\bZ)$, and two classes $z,z'$ in $H^{20}(X_3^\sigma,\bZ)$, each defined by the class of a line contained respectively in $P$ and $P'$, and represented by the subvarieties $Z$ and $Z'$ defined in~\autoref{eq:Z}. Since
both~$Z$ and~$Z'$ are Schubert varieties of type~$\Sigma_{443}$, the two
classes~$z$ and~$z'$ share the same Schubert part $z_1=z_1'$, which can be
determined explicitly as in~\autoref{eq:z1} of~\autoref{lem2:36} and has
square~$\frac5{11}$.

Let us suppose that we are in {\bf \ref{case1}} above, so let us suppose that $V_4=\langle e_1,\cdots,e_4 \rangle \subset V_7 =\langle e_0,\cdots,e_6 \rangle$ and $V'_4=\langle e_5,\cdots,e_8 \rangle \subset V'_7 =\langle e_3,\cdots,e_9 \rangle$. The line defining $Z$ (respectively $Z'$) is defined by some five dimensional subspace of $V_7$ (resp. $V'_7$) containing $V_4$ (resp. $V'_4$); for instance choose $V_5=\langle e_0,\cdots,e_4 \rangle$ and $V_5=\langle e_5,\cdots,e_9 \rangle$. Since a point $[V_3]\in Z\cap Z'$ is such that $\dim(V_3\cap V_5)\geq 2\leq \dim(V_3\cap V'_5)$ one deduces that $\dim(V_3\cap V_5\cap V'_5)\neq 0$; but since $V_5\cap V'_5=0$ this implies that $Z\cap
Z'=\emptyset$. As a consequence, we have $0=z\cdot z'=z_0\cdot z_0'+z_1^2$ so
\[z_0\cdot z_0'=-z_1^2=-\tfrac{5}{11}.\]
Moreover, we have
\[z_0\cdot z_0'=z\cdot z_0'=p_*q^*l\cdot z_0'=l\cdot c(H-11\lambda')=c(1
-\tfrac{11}{2}\q(\lambda,\lambda'))\]
and
\[z_0\cdot z_0=z\cdot z_0=p_*q^*l\cdot z_0=l\cdot c(H-11\lambda)=56c.
\]
By \autoref{lem1:36}, $\q(\lambda,\lambda')$ has two possible values~$\pm2$.
If $\q(\lambda,\lambda')=2$, we get $c=\tfrac{1}{22}$ while if
$\q(\lambda,\lambda')=-2$, we get $c=-\tfrac{5}{132}$. We may compute
$z^2=z_0^2+\tfrac{5}{11}=56c+\tfrac{5}{11}$ which is equal to $3$ or
$-\tfrac53$ in the two cases. Since~$Z$ is integral, the latter is absurd, so
we may conclude that $\q(\lambda,\lambda')=2$, $c=\tfrac{1}{22}$, and $z^2=3$.
Finally, we get
\[z_0^2=\tfrac{28}{11},\quad \q(q_*p^*z_0,q_*p^*z_0)=\q\left(\tfrac{1}{22}(H-
11\lambda), \tfrac{1}{22}(H-11\lambda)\right)=-\tfrac{28}{11},\]
which proves what we need.
\end{proof}

\begin{remark}
\label{rem3:36}
By \autoref{lem1:36}, we know that $\q(\lambda, \lambda')=\pm 2$. In the proof
of the theorem, we saw that in {\bf \ref{case1}}, we have $\q(\lambda,
\lambda')=2$ (we could also have used {\bf \ref{case2}}, in which case one
obtains $\q(\lambda, \lambda')=-2$ instead). In the proof, we showed that $z^2=3$. This allows us to write out the full
intersection matrix of the sublattice $\bZ z+j^*H^{20}(\Gr(3,V_{10}), \bZ)$,
whose discriminant can then be computed to be~$28$. Since the middle
cohomology~$H^{20}(X_3^\sigma,\bZ)$ is a unimodular lattice, the orthogonal
complement $H^{20}(X^\sigma_3,\bZ)_\van^{\perp z}$ has the same discriminant.
This last lattice is mapped via $q_*p^*$ onto the sublattice $(\bZ H+\bZ \lambda)^\perp$ of $H^2(X^\sigma_6,\bZ)$, which is of signature $(2,19)$ and indeed has discriminant~$-28$.
\end{remark}
% particular this shows that the
% modular map~$\m$
% % extended period map~$\tilde\p$
% restricts to a
% dominant rational map~$\cD^{4,7,7}\dashrightarrow \cC_{28}$ onto the
% Noether--Lefschetz divisor of discriminant~$28$, which is in turn
% mapped to the Heegner divisor~$\cD_{28}\subset\cP$ in the period domain.
% As before, the normality of the period domain~$\cP$ allows us to deduce the
% birationality.

% \begin{proposition}
% \label{propD28}
% For $[\sigma]$ very general in the divisor $\cD^{4,7,7}$, the corresponding
% transcendental sublattice $H^2(X^\sigma_6,\bZ)_\trans$ is of discriminant~$28$.
% The divisor $\cD^{4,7,7}$ is mapped birationally onto the Heegner divisor
% $\cD_{28}$ via $\p\circ\m\colon \cM\dashrightarrow \cP$.
% \end{proposition}

Another consequence of the theorem is the integral Hodge conjecture for
$H^{20}(X^\sigma_3,\bZ)$, following ideas of Mongardi--Ottem~\cite{MO} for
cubic fourfolds. We first state a basic lemma on abelian groups.
\begin{lemma}
\label{lemma:dual-saturated-image}
Let $L$ and $M$ be two abelian groups and let $\varphi\colon L\to M$ be a
homomorphism. Let $L_1$ be a saturated subgroup of $L$ (that is, $L/L_1$
is torsion-free). Moreover, we suppose that $\varphi|_{L_1}\colon L_1\to M$ is
injective with saturated image (that is, $\coker \varphi|_{L_1}$ is
torsion-free). Then writing $\varphi^\vee$ for the dual homomorphism
$\varphi^\vee\colon
M^\vee\to L^\vee$, we have
\[
\varphi^\vee(M^\vee)+L_1^\perp=L^\vee,
\]
where $L_1^\perp\coloneqq\setmid{f\in L^\vee}{f|_{L_1}=0}$ is the orthogonal of
$L_1$ in $L^\vee$.
\end{lemma}
\begin{proof}
Since $L_1$ is saturated in $L$, we have a surjective map $p\colon L^\vee \onto
L_1^\vee$ where the kernel is precisely $L_1^\perp$. So we get an induced
isomorphism
\[
p\colon L^\vee/L_1^\perp\simto L_1^\vee.
\]
On the other hand, since $\varphi|_{L_1}\colon L_1\to M$ is injective with saturated
image, we get similarly a surjective map $\varphi^\vee\colon M^\vee\onto
L_1^\vee$, which factorizes as the composition of $\varphi^\vee\colon M^\vee\to
L^\vee$ and the natural map $p\colon L^\vee\onto L_1^\vee$. Comparing with the
isomorphism that we obtained above, we get a surjective map
\[
\varphi^\vee\colon M^\vee\onto L^\vee/L_1^\perp.
\]
We may thus conclude that $\varphi^\vee(M^\vee)$ and $L_1^\perp$ generate
$L^\vee$.
\end{proof}

\begin{corollary}
\label{cor:HodgeConjX3}
The integral Hodge conjecture holds for $H^{20}(X^\sigma_3,\bZ)$.
\end{corollary}
\begin{proof}
The maps in diagram \eqref{eq:I36} define an injective morphism
\[q_*p^*\colon H^{20}(X_3^\sigma,\bZ)_\van\simto H^2(X_6^\sigma,\bZ)_\prim\into
H^2(X_6^\sigma,\bZ)\]
of abelian groups.
By definition of the primitive cohomology, this has
saturated image, therefore we may apply Lemma~\ref{lemma:dual-saturated-image}
by letting $L=H^{20}(X_3^\sigma,\bZ)$ and $L_1=H^{20}(X_3^\sigma,\bZ)_\van$ and
obtain that
\[
(p_*q^*)^\vee(H^2(X_6^\sigma,\bZ)^\vee)+L_1^\perp=L^\vee.
\]
The group $L$ equipped with the intersection product is a unimodular lattice,
so we can identify $L^\vee$ with $L$ itself. By the construction of the
vanishing
cohomology, we have $L_1^\perp=j^*H^{20}(\Gr(3,V_{10}),\bZ)$.
Using the Poincaré duality, we may identify $H^2(X_6^\sigma,\bZ)^\vee$ with
$H^6(X_6^\sigma,\bZ)$ and $(q_*p^*)^\vee$ with $p_*q^*$.
So we obtain
\[
p_*q^* H^6(X^\sigma_6,\bZ)+ j^*H^{20}(\Gr(3,V_{10}),\bZ)=
H^{20}(X^\sigma_3,\bZ).
\]
Since the integral Hodge conjecture holds for $H^6(X^\sigma_6,\bZ)$
by~\cite[Theorem~0.1]{MO}, and since the map~$p_*q^*$ is given by an integral
correspondence, every integral $(10,10)$-class on $X^\sigma_3$ is therefore
algebraic.
\end{proof}

\begin{theorem}
\label{thm:HodgeConjX3}
When~$\sigma$ is such that ~$X^\sigma_3$ is smooth (that is,
when~$[\sigma]\notin\cD^{3,3,10}$), the integral Hodge conjecture holds for
$X^\sigma_3$ in all degrees.
\end{theorem}
\begin{proof}
Since $X^\sigma_3$ is a hyperplane section of the Grassmannian $\Gr(3,V_{10})$,
the cohomology classes in degrees $2$ to $18$ all come from Schubert classes
thanks to the Lefschetz hyperplane theorem.

The degree $20$ case is settled in \autoref{cor:HodgeConjX3}.

For degree $22$ to degree $38$, by using Poincaré duality, we can verify that
the Schubert classes also produce all the cohomology classes, except
in degree $22$, where they only generate a subgroup of
$H^{22}(X_3^\sigma,\bZ)$ of index $3$.

More precisely, in each degree $2k$ for $11\le k\le 19$, by the Lefschetz hyperplane theorem, we already know that the pullback of Schubert classes $\set{j^*\sigma_{abc}}_{a+b+c=20-k}$ generate $H^{40-2k}(X_3^\sigma,\bZ)\simeq H^{2k}(X_3^\sigma,\bZ)^\vee$. We then consider the pullback of Schubert classes $\set{j^*\sigma_{def}}_{d+e+f=k}$ and pair them with $\set{j^*\sigma_{abc}}$. Using
\[j^*\alpha\cdot j^*\beta = [X_3^\sigma]\cdot \alpha\cdot\beta=\sigma_{100}\cdot \alpha\cdot\beta,\]
we can explicitly compute the intersection matrix and then reduce it to Smith normal form. This gives us the desired description of the quotient group $H^{2k}(X_3^\sigma,\bZ)/j^*H^{2k}(\Gr(3,V_{10}),\bZ)$.

To conclude for degree $22$, we note that there is an extra algebraic
class $g$ represented by the Grassmannian $\Gr(3,U_6)$ for any $[U_6]\in
X_6^\sigma$. It is easy to see that the class $g$ only intersects the Schubert
class $j^*\sigma_{333}\in H^{18}(X_3^\sigma,\bZ)$ with intersection number $1$.
This allows us to extend the intersection matrix computed above and verify that $j^*H^{22}(\Gr(3,V_{10}),\bZ)+\bZ g$ generates
$H^{18}(X_3^\sigma,\bZ)^\vee\simeq H^{22}(X_3^\sigma,\bZ)$.
\end{proof}

\begin{remark}
The extra algebraic class $g$ can be seen as $p_*q^*[*]$, where
$[*]\in H^8(X_6^\sigma,\bZ)$ is the class of a point in $X_6^\sigma$. We see
that Schubert classes only produce the class $3g$.  This phenomenon reappears
below for the variety $X_1^\sigma$: if $\pi=p_*q^*[*]$ is the
class of a Palatini threefold, Schubert classes---in particular the Lefschetz
operator---only produce the class $3\pi$.
\end{remark}

\subsection{Integral Hodge structures of \texorpdfstring{$X^\sigma_1$}{X1}}

In this section we proceed to the proof of the Hodge isometries between the integral Hodge structures of $X_1^\sigma$ and $X_6^\sigma$. The Hodge
numbers of $X^\sigma_1$ were recently computed in~\cite[Table~4.1]{bfm} (beware that in \cite{bfm} $X_1^\sigma$ is denoted by $P$),
where the integral cohomology $H^*(X^\sigma_1,\bZ)$ is also shown to be
torsion-free. Since all the cohomologies in odd degree vanish, we list only
the even degree ones:
\begin{equation}
\label{eq:diamond}
\arraycolsep=4pt
\begin{array}{l|cccccccc}
h^0    &   &   &   & 1  &   &   &   \\
h^2    &   &   & 0 & 1  & 0 &   &   \\
h^4    &   & 0 & 1 & 22 & 1 & 0 &   \\
h^6    & 0 & 0 & 1 & 22 & 1 & 0 & 0 \\
h^8    &   & 0 & 1 & 22 & 1 & 0 &   \\
h^{10} &   &   & 0 & 1  & 0 &   &   \\
h^{12} &   &   &   & 1  &   &   &
\end{array}\end{equation}
We see that there are three Hodge structures of
K3-type on different levels. They are related by the Lefschetz operator (see
\autoref{lemma:Lefschetz}) and there is a polarization given by the cup
product on $H^6(X^\sigma_1,\bZ)$.  In~\cite{bfm}, the authors showed that the
Hodge structure of $H^{20}(X^\sigma_3,\bZ)_{\van}$ can be mapped into each of
the three Hodge structures of $X^\sigma_1$ by using certain geometric
constructions called {\em jumps} between the
Grassmannians~$\bP(V_{10})$,~$\Gr(2,V_{10})$, and~$\Gr(3,V_{10})$. Here we show
that this can also be done by using the incidence variety $I^\sigma_{1,6}$.

As in the case of $X^\sigma_3$, we first determine the suitable Hodge
structure to study: we define the vanishing cohomology $
H^6(X_1^\sigma,\bZ)_{\van} $ to be the orthogonal of the sublattice generated
by~$h^3$ and the class $\pi$ of a Palatini threefold in~$X_1^\sigma$ (see
\autoref{prop:univ_palatini}). To check that these two
classes generate a rank-$2$ sublattice, one can compute their intersection
matrix as follows:
\begin{itemize}
\item the self-intersection number $h^3\cdot h^3$ is the degree of
$X_1^\sigma$, which is 15 (see \cite{dfmr});

\item the intersection number $h^3\cdot \pi$ is the degree of the Palatini
threefold, which is 7 (see \cite{fm});

\item to compute the self-intersection number $\pi\cdot \pi$, we take two
general points $[V_6]$ and $[V_6']$ from $X_6^\sigma$. Their
intersection~$V_6\cap V_6'$ is a~$2$-dimensional subspace $V_2$, and the
sum~$V_6+V_6'$ is the whole $V_{10}$. So one obtains
$\sigma(V_2,V_2,V_{10})=0$. In particular, this shows that $[V_2]$ is in the
degeneracy locus $X_2^\sigma$. Since the two Palatini threefolds are represented by $\bP(V_6)\cap X_1^\sigma$ and $\bP(V_6')\cap X_1^\sigma$, their intersection is given by $\bP(V_2)\cap X_1^\sigma$. For any one dimensional subspace $l\subset V_2$ we have $\sigma(l,V_2,V_{10})=0$, thus $\sigma$ defines an element $\sigma(l,V_{10},V_{10})\in \wedge^2 (V_{10}/V_2)^\vee$. As a consequence one can verify that the intersection $\bP(V_2)\cap X_1^\sigma$ is given by the pfaffian of such a $2$-form, which is a quartic in $\bP(V_2)\cong \bP^1$, i.e. four points (with multiplicities!). Thus $\pi\cdot  \pi =4$. %It defines then a~$4$-secant line~$\bP(V_2)$ of the variety $X_1^\sigma$ (see~\cite[Section~3.1]{han1}). As the class $\pi$ of the Palatini threefolds can be represented by both $\bP(V_6)\cap X_1$ and $\bP(V_6')\cap X_1$, its self-intersection number is~$4$.
\end{itemize}
The intersection matrix for $\bZ h^3+\bZ\pi$ is therefore
\[\arraycolsep=5pt
\begin{pmatrix}15&7\\7&4\end{pmatrix}.\]
This is a saturated rank-$2$ lattice of discriminant~$11$. Its orthogonal
complement, the vanishing cohomology $H^6(X^\sigma_1,\bZ)_\van$---a polarized
integral Hodge structure of type $(1,20,1)$---therefore also has discriminant~$11$.

%For cohomologies in degree~$k=4,8$, we first use the Lefschetz operator $L_h$ over~$\bQ$ to identify~$H^k(X_1^\sigma,\bQ)_\van$, then define the corresponding intersection with the integral cohomology to be~$H^k(X_1^\sigma,\bZ)_\van$. A priori, the Lefschetz operators might not remain isomorphisms over integral coefficients. We will clarify this in \autoref{lemma:Lefschetz}.

The following is the analogue of \autoref{thm:36}.
\begin{theorem}\label{thm:16}
When~$\sigma$ is such that~$X_1^\sigma$ and~$X_6^\sigma$ are both smooth
(that is, when~$[\sigma]\notin\cD^{3,3,10}\cup\cD^{1,6,10}$), the morphism
\begin{equation}
\label{eq:qp16}
q_*p^*L_h\colon H^{6}(X^\sigma_1,\bZ)_\van\simto H^2(X^\sigma_6,\bZ)_\prim(-1)
\end{equation}
is an isomorphism of polarized integral Hodge structures. Here~$q_*p^*$ is
the correspondence defined by~$I_{1,6}^\sigma$ in the diagram~\autoref{eq:I16},
$L_h$ is the Lefschetz operator given by cup product with $h$, and the $(-1)$ means that $H^2(X^\sigma_6,\bZ)_\prim$ is endowed with the
quadratic form~$-\q$.
\end{theorem}
The proof is essentially the same as the proof of \autoref{thm:36} and involves
the study of~$\cD_{28}$.

\begin{remark}
One can also derive \autoref{thm:16} directly from \autoref{thm:36}
and~\cite[Theorem~19]{bfm} and, with the same line of ideas, one can obtain the
analogous statement concerning $X^\sigma_2$. However, we decided to include a
direct proof of \autoref{thm:16} to illustrate once more how one can
specialize~$[\sigma]$ to divisors to obtain more precise information on the
cohomologies.
\end{remark}

We will
use the correspondence $X^\sigma_1\overset p\longleftarrow I^\sigma_{1,6}\overset
q\longrightarrow X^\sigma_6$ from \autoref{eq:I16}. Recall that $h\in
H^2(X_1^\sigma,\bZ)$ is the polarization on $X^\sigma_1$, the class $\pi\in
H^{6}(X_1^\sigma,\bZ)$ is the class of a Palatini threefold, and
$H^6(X^\sigma_1,R)_\van$ is defined as~$\inner{h^3,\pi}^\bot$ for~$R=\bQ$ or $\bZ$. Note that the Lefschetz operators $L_h$ from $H^4(X^\sigma_1,\bQ)$ to $H^6(X^\sigma_1,\bQ)$ and from $H^6(X^\sigma_1,\bQ)$ to $H^8(X^\sigma_1,\bQ)$ both define an isomorphism of $\bQ$-vector spaces. 
\begin{proof}[Proof of Theorem~\ref{thm:16}]

We will first show that
\[
q_*p^*L_h\colon H^6(X_1^\sigma,\bQ)_\van\to H^2(X_6^\sigma,\bQ)_\prim
\]
is an isomorphism of $\bQ$-vector spaces.
For a very general~$\sigma$, the Hodge structure~$H^2(X^\sigma_6,\bQ)_\prim$ is
simple, so in this case it will suffice to show that the map is non-zero.
But this is a topological property, so we only need to show that the map is
non-zero for some particular $\sigma$, not necessarily very general.
Once we know that $q_*p^*L_h$ is an isomorphism of $\bQ$-vector spaces, we can
use Corollary~\ref{cor:isometry_upto_constant} to deduce that it is an isometry up
to a scalar, and it will suffice to show that this scalar is equal to $-1$.

As in the proof of Theorem~\ref{thm:36}, we consider a general~$[\sigma]$ in the
divisor~$\cD^{4,7,7}$, so that~$X_6^\sigma$ contains a unique
plane~$P=\bP\big((V_7/V_4)^\vee\big)$. Denote by~$\ell$ the class of a line
contained in $P$, and consider the class $z\coloneqq p_*q^*\ell\in
H^4(X^\sigma_1,\bZ)$.
We would like to study the intersection
\[\bP(V_7)\cap X^\sigma_1\subset X_1^\sigma\]
and find a subvariety that represents the class $z$.
It is easy to see that~$\bP(V_4)$ is always contained in~$X_1^\sigma$,
therefore this intersection is not irreducible, so we will use~$Z$ to denote
the other component of $\bP(V_7)\cap X_1^\sigma$ and show that
when it is of expected dimension~$4$, it has class~$p_*q^*\ell\in
H^4(X_1^\sigma,\bZ)$.

We first describe the geometry of the incidence variety~$I_{1,6}^\sigma$.
Fibers of the map~$q$ above~$[V_6]\in P$ are degenerate Palatini threefolds
having~$\bP(V_4)$ as one irreducible component. The preimage~$q^{-1}(P)$
therefore consists of two components~$Y$ and~$Y'$: the map~$p$ projects the
first component~$Y$ onto~$\bP(V_4)\subset X_1^\sigma$, and the second
component~$Y'$ onto~$Z$. The fibers of $p\colon Y\to \bP(V_4)$
are just copies of the plane~$P$, while the fibers of $p\colon Y'\to Z$ away
from $\bP(V_4)$ can be described as follows: each~$[V_1]$ not lying
in~$\bP(V_4)$ spans a~$5$-dimensional subspace $V_1\oplus V_4$, and since
$\sigma(V_4,V_7,V_7)=0$, each~$[V_6]$ in the line~$\bP(V_7/(V_1\oplus V_4))$
lies in~$X_6^\sigma$. Therefore the generic fibers of~$p\colon Y'\to Z$ are
lines contained in~$P$.
For a fixed line $\ell\subset P$, the generic fibers of $p\colon q^{-1}(\ell)\cap Y'\to Z$ are
therefore intersections of two lines in~$P$, so this is a birational map, and
% When we consider the preimage $q^{-1}(l)$ of a line~$l\subset P$, its component
% contained in $Y'$ has points as generic fibers for the map~$p$, and
we may conclude that the class $p_*q^*\ell\in H^4(X_1^\sigma,\bZ)$ is indeed
represented by the class of~$Z$ with multiplicity~$1$.
The geometry can be summarized in the following diagram
\[
\begin{tikzcd}
&Y\cup Y'\ar[ld]\ar[r,hook]&I_{1,6}^\sigma\ar[ld,near start,"p"]\ar[rd,"q"]\\
\bP(V_4)\cup Z\ar[r,hook]&X_1^\sigma& P\ar[crossing
over,lu,leftarrow]\ar[r,hook]&X_6^\sigma.
\end{tikzcd}
\]

By intersecting $z$ with $h$ we get a class $z\cdot h\in H^6(X^\sigma_1,\bZ)$
which we can write as a sum $z\cdot h=x_0+x_1$, where $x_0\in
H^6(X^\sigma_1,\bQ)_\van$ and $x_1\in \bQ h^3+\bQ\pi$. For~$\sigma$ very
general in the divisor~$\cD^{4,7,7}$, $X_6^\sigma$ has Picard rank 2
and~$\frac12(H-11\lambda)$ generates the space of primitive algebraic classes
$H^2(X_6^\sigma,\bZ)_\prim\cap H^{1,1}(X_6^\sigma)$,
so there is a rational number $c\in \bQ$ such that
$q_*p^*(x_0 \cdot h)=c(H-11\lambda)$. If $c$ is non-zero, then the map
$q_*p^*L_h$ will also be non-zero. We may compute
\[x_0^2=z\cdot h\cdot x_0=p_*q^*\ell\cdot h\cdot x_0
=\ell\cdot q_*p^*(h\cdot x_0) =\ell\cdot c(H-11\lambda)
=56c.\]

To deduce the value of $c$, we consider again the two special cases where
$X^\sigma_6$ contains two planes $P$ and $P'$, and get two subvarieties $Z$ and
$Z'$ and their classes~$z$ and~$z'$. The
two planes $P$ and $P'$ are in the same polarized monodromy orbit, hence so are
the classes $\ell$ and $\ell'$ for the two lines. The correspondence $p_*q^*$
is clearly monodromy equivariant, therefore the two classes $z$ and $z'$ are
also in the same monodromy orbit. In particular, this means that the components
$x_1$ and $x_1'$ are the same, since both classes $h^3,\pi\in
H^6(X_1^\sigma,\bZ)$ are monodromy invariant.
Therefore, $z\cdot z'\cdot h^2$ is equal to $x_0\cdot x_0'+x_1^2$.

On the other hand, we may similarly compute
\[
x_0\cdot x_0'=z\cdot h\cdot x_0'=p_*q^*\ell\cdot h\cdot x_0'=\ell\cdot q_*p^*(h\cdot
x_0')=\ell\cdot c(H-11\lambda')=c-c\cdot\tfrac{11}2\q(\lambda,\lambda').
\]
By Remark~\ref{rem3:36}, we know that $\q(\lambda,\lambda')$ equals to $2$ in
{\bf \ref{case1}} and $-2$ in {\bf \ref{case2}}. So $x_0\cdot x_0'$ is equal to
$-10c$ and $12c$ in the two cases respectively.

Now we compute the intersection number~$z\cdot z'\cdot h^2$ in an alternative
way, using the geometry. It suffices to determine the intersection
$\bP(V_7)\cap\bP(V_7')\cap X_1^\sigma$: we will show that in {\bf \ref{case1}},
it is a quadric surface and some lower-dimensional components, and in {\bf
\ref{case2}}, it is a cubic surface and some lower-dimensional components.

More precisely, in the basis $(e_0,\dots,e_9)$ of {\bf \ref{case1}}
described above, the intersection is defined inside $\bP(V_7\cap
V_7')=\bP(\inner{e_3,e_4,e_5,e_6})$ as the locus where the~$10\times10$
skew-symmetric matrix
\[\scalebox{.9}{$
\left(\begin{array}{ccccc|ccccc}
&&&&&a_{056}x_6&-a_{056}x_5&f_{07}&f_{08}&f_{09}\\
&&&&&&&f_{17}&f_{18}&f_{19}\\
&&&&&&&f_{27}&f_{28}&f_{29}\\
&&&&&&&&&-a_{349}x_4\\
&&&&&&&&&a_{349}x_3\\
\hline
    -a_{056}x_6&&&&\\
    a_{056}x_5&&&&\\
    -f_{07}&-f_{17}&-f_{27}&&\\
    -f_{08}&-f_{18}&-f_{28}&&\\
    -f_{09}&-f_{19}&-f_{29}&a_{349}x_4&-a_{349}x_3
\end{array}\right)
$}\]
has rank $\le 6$.  Here we only write down the non-zero entries: each
$a_{ijk}\coloneqq\sigma(e_i,e_j,e_k)$ is a constant, and each $f_{ij}$ is the
restriction of the linear form $\sigma(e_i,e_j,-)$ to
$\inner{e_3,e_4,e_5,e_6}$, a polynomial in $x_3,x_4,x_5,x_6$ of degree 1.
Note that the constants $a_{056}$ and $a_{349}$ are both non-zero: otherwise
$\sigma$ would vanish entirely on $V_7=\inner{e_0,\dots,e_6}$ or
$V_7'=\inner{{e_3,\dots,e_9}}$, which is not possible by
Lemma~\ref{lemma:unstable}.
Similarly, the quadratic form $f_{17}f_{28}-f_{27}f_{18}$ cannot be
zero: otherwise one easily verifies that there exists a suitable 7-dimensional
subspace of $\inner{e_1,\dots,e_8}$ on which $\sigma$ vanishes entirely.
So the locus where the rank drops is the union of the quadric surface defined by
$f_{17}f_{28}-f_{27}f_{18}$ and the two lines $x_3=x_4=0$ and $x_5=x_6=0$.

In {\bf \ref{case2}}, the matrix is instead the following
\[\scalebox{.9}{$
\left(\begin{array}{ccc|c|ccc|ccc}
        &&&&&&&f_{07}&f_{08}&f_{09}\\
        &&&&&&&f_{17}&f_{18}&f_{19}\\
        &&&&&&&f_{27}&f_{28}&f_{29}\\
        \hline
        &&&0&&&&\\
        \hline
        &&&&&ax_6&-ax_5&\\
        &&&&-ax_6&&ax_4&\\
        &&&&ax_5&-ax_4&&\\
        \hline
        -f_{07}&-f_{17}&-f_{27}&&&&\\
        -f_{08}&-f_{18}&-f_{28}&&&&\\
        -f_{09}&-f_{19}&-f_{29}&&&&\\
\end{array}\right).
$}\]
Again $a=a_{456}=\sigma(e_4,e_5,e_6)$ is a constant and each $f_{ij}$ is the
restriction of the linear form $\sigma(e_i,e_j,-)$ to
$\inner{e_3,e_4,e_5,e_6}$.
Similarly, the constant $a$ and the determinant $\det(f_{ij})$ are both
non-zero, because otherwise $\sigma$ would vanish entirely on a 7-dimensional
subspace.
Therefore the locus where the rank is $\le 6$ is the union of the
cubic surface defined by $\det(f_{ij})=0$ and the point $x_4=x_5=x_6=0$.

Consequently, we see that $z\cdot z'\cdot h^2$ equals to 2 or 3 in the two
cases respectively. This allows us to conclude that $c=\frac1{22}$ and
$x_1^2=\frac{27}{11}$. Since $c$ is non-zero, the map $q_*p^*L_h$ is indeed an
isomorphism of $\bQ$-vector spaces.

Once we know the isomorphism over $\bQ$, we automatically get a Hodge isometry
up to a scalar, so it remains to show that the scalar is~$-1$. We consider the
class $x_0$ which satisfies
\[x_0^2=56c=\tfrac{28}{11}\quad\text{while}\quad \q\left(q_*p^*(x_0\cdot
h)\right) = \q\left(\tfrac{1}{22}(H-11\lambda)\right)= -\tfrac{28}{11}.\]
This allows us to conclude the proof. Also note that $(z\cdot
h)^2=x_0^2+x_1^2=5$, which is an integer as one would expect.
\end{proof}
In the proof, we have seen that if we consider the class $z\coloneqq p_*q^*\ell\in
H^4(X_1^\sigma,\bZ)$, the class $z\cdot h$ admits components $x_0\in
H^6(X_1^\sigma,\bQ)_\van$ and $x_1\in \bQ h^3+\bQ\pi$, with
$x_0^2=\frac{28}{11}$ and $x_1^2=\frac{27}{11}$. This allows us to determine
the class $x_1$.
\begin{lemma}
In the above notation, we have $x_1=\frac3{11}(h^3+\pi)$.
\end{lemma}
\begin{proof}
The direct sum $H^6(X_1^\sigma,\bZ)_\van\oplus(\bZ h^3+\bZ \pi)$ is of index 11
in $H^6(X_1^\sigma,\bZ)$. Since the class $z\cdot h$ is integral, we see that
$11x_1\in \bZ h^3+\bZ\pi$, so we may write $11x_1=ah^3+b\pi$. The value
$x_1^2=\frac{27}{11}$ leads to the equation
\[
15a^2+14ab+4b^2=297,
\]
from which one may deduce that
\[
b=-\tfrac74a\pm\tfrac14\sqrt{11(108-a^2)}.
\]
The only integer solutions are $(a,b)=\pm(3,3)$ so
$x_1=\pm\frac3{11}(h^3+\pi)$. Finally, since the class $z$ is given by the
subvariety $Z$ and therefore is effective, we may conclude that
$x_1=\frac3{11}(h^3+\pi)$.
\end{proof}
We can then compute the intersection matrix for $h^3$, $\pi$, and $z\cdot h$,
and find
\begin{equation}
\label{eq:matrix-D28}
\begin{pmatrix}15&7&6\\7&4&3\\6&3&5\end{pmatrix}.
\end{equation}
In particular, the discriminant group of the lattice generated by these three
classes is $\bZ/28\bZ$. So this is also the discriminant group of the
orthogonal $H^6(X^\sigma_1,\bZ)_\van^{\perp z\cdot h}$ and that of the
transcendental lattice $H^2(X^\sigma_6,\bZ)_\trans$.

\begin{corollary}
\label{cor:image-h3-pi}
When~$\sigma$ is such that~$X^\sigma_1$ is smooth (that is,
when $[\sigma]\notin\cD^{3,3,10}\cup\cD^{1,6,10}$), we have
\[
q_*p^*(h^4)=6H,\quad
q_*p^*(\pi\cdot h)=3H,
\]
where $H\in H^2(X_6^\sigma,\bZ)$ is the Plücker polarization on $X_6^\sigma$. If moreover $[\sigma]$ lies in $\cD^{4,7,7}$, we have a Lagrangian plane
$P=\bP\big((V_7/V_4)^\vee\big)$ with $(-10)$-class $\lambda\in
H^2(X_6^\sigma,\bZ)$. Write $\ell\in H^6(X_6^\sigma,\bZ)$ for the class of a
line in $P$ and consider the class $z\coloneqq p_*q^*\ell \in
H^4(X_1^\sigma,\bZ)$. We have
\[
q_*p^*(z\cdot h^2)=\tfrac12(5H-\lambda).
\]
\end{corollary}
\begin{proof}
By the simplicity of the Hodge structure $H^2(X_6^\sigma,\bZ)_\prim$ for a very
general $\sigma$, we deduce that there exists constants $a,b$ such that
$q_*p^*(h^4)=aH$ and $q_*p^*(\pi\cdot h)=bH$. To determine their
values, we specialize to the case $[\sigma]\in \cD^{4,7,7}$ that we studied
above. For $\ell$ the class of a line, we use the intersection numbers in
\eqref{eq:matrix-D28} to compute
\[
a = q_*p^*(h^4)\cdot \ell=h^4\cdot p_*q^*\ell=h^4\cdot z=6,
\]
and
\[
b = q_*p^*(\pi\cdot h)\cdot \ell=\pi\cdot h\cdot
p_*q^*\ell=\pi\cdot h\cdot z=3.
\]

For the class $z=p_* q^*\ell$, we have the decomposition $z\cdot h =
x_0+x_1$, where we have shown that
$q_*p^*L_h(x_0)=\frac1{22}(H-11\lambda)$ and $x_1=\frac3{11}(h^3+\pi)$.
We may thus conclude that
\[
q_*p^*(z\cdot
h^2)=\tfrac1{22}(H-11\lambda)+\tfrac3{11}(6H+3H)=\tfrac12(5H-\lambda),
\]
where we used the images for $h^4$ and $\pi\cdot h$ that we just obtained.
\end{proof}

\begin{remark}
As a side note, since the projective space $\bP(V_4)$ is contained in
$X^\sigma_1$, its class should be a linear combination of the three classes
$h^3$, $\pi$, and $z\cdot h$. One may check that
$[\bP(V_4)]=\pi-z\cdot h$ using the intersection numbers.
\end{remark}

As in the case of $X^\sigma_3$, we can obtain the integral Hodge conjecture for
$X^\sigma_1$. From the Hodge diamond \autoref{eq:diamond} and the fact that
$X_1^\sigma$ contains lines (see Theorem \ref{thmfanoXone}), we see that the only non-trivial cases are in
degrees $4$, $6$, and $8$. First we treat the case of the middle cohomology.
\begin{corollary}
When~$\sigma$ is such that ~$X^\sigma_1$ is smooth (that is,
when~$[\sigma]\notin\cD^{3,3,10}\cup\cD^{1,6,10}$), the integral Hodge
conjecture holds for $H^6(X_1^\sigma,\bZ)$.
\end{corollary}
\begin{proof}
The proof is similar to that of Corollary~\ref{cor:HodgeConjX3}: namely, we
apply Lemma~\ref{lemma:dual-saturated-image}
with $L=H^{6}(X_1^\sigma,\bZ)$ and $L_1=H^{6}(X_1^\sigma,\bZ)_\van$ to get
\[
(q_*p^*L_h)^\vee(H^2(X_6^\sigma,\bZ)^\vee)+L_1^\perp=L^\vee,
\]
which, since $L$ has the structure of a unimodular lattice, translates into
\[
L_hp_*q^*H^6(X^\sigma_6,\bZ)+(\bZ h^3+\bZ \pi)=H^6(X^\sigma_1,\bZ).
\]
So we may again conclude from the integral Hodge
conjecture for $H^6(X^\sigma_6,\bZ)$.
\end{proof}
The proof for degrees $4$ and $8$ is a bit more involved, due to the
following.
\begin{lemma}
\label{lemma:Lefschetz}
Suppose that~$X_1^\sigma$ is smooth.
The image of the Lefschetz operator
\[L_h\colon H^6(X^\sigma_1,\bZ)\to H^8(X^\sigma_1,\bZ)\]
is a subgroup of index $3$. By duality, the same is true for
\[L_h\colon H^4(X^\sigma_1,\bZ)\to H^6(X^\sigma_1,\bZ).\]
When restricted to the vanishing parts, both Lefschetz operators become
isomorphisms.
\end{lemma}
\begin{proof}
We want to determine for which classes
in~$H^6(X_1^\sigma,\bZ)$ the image by~$L_h$ is divisible.
We will do this by studying the image of~$q_*p^*L_h$.
Since this is a topological property, it suffices to study the
case of a general~$\sigma$ in the divisor~$\cD^{4,7,7}$, so we retain the
notation of~$l\in H_2(X_6^\sigma,\bZ)$, $z\in H^4(X_1^\sigma,\bZ)$,
and~$z\cdot h = x_0+x_1$, where we have already shown that the algebraic
part~$x_1$ is equal to~$\frac3{11}(h^3+\pi)$.
Note that since~$z\cdot h$ does not lie in the direct
sum~$H^6(X_1^\sigma,\bZ)_\van\oplus(\bZ h^3+\bZ\pi)$ which is a sublattice
of~$H^6(X_1^\sigma,\bZ)$ of index~$11$, by adding the class~$z\cdot h$ we get
the entire lattice so
\[
H^6(X^\sigma_1,\bZ)=H^6(X^\sigma_1,\bZ)_\van+\bZ z\cdot h+(\bZ h^3+\bZ\pi).
\]
Using the intersection matrix \autoref{eq:matrix-D28}, we may compute
$q_*p^*(h^4)\cdot l=h^4\cdot z=6$ and $q_*p^*(\pi\cdot h)\cdot l=\pi\cdot
h\cdot z=3$. As $h^4$ and $\pi\cdot h$ are always algebraic, while
for very general $\sigma$, the only algebraic classes
in $H^2(X_6^\sigma,\bZ)$ are the multiples of $H$, we may conclude that
\[q_*p^*(h^4)=6H\quad\text{and}\quad q_*p^*(\pi\cdot h)=3H.\]
Consequently, the image of the class~$x_1\cdot h=\tfrac3{11}(h^3+\pi)\cdot h$
is
\[q_*p^*(x_1\cdot h)=q_*p^*\big(\tfrac3{11}(h^4+\pi\cdot
h)\big)=\tfrac{27}{11}H,\]
and for the class~$z\cdot h^2=(x_0+x_1)\cdot h$ we get
\[q_*p^*(z\cdot h^2)=q_*p^*(x_0\cdot h+x_1\cdot
h)=\tfrac1{22}(H-11\lambda)+\tfrac{27}{11}H=3H-\tfrac12(H+\lambda).\]
By \autoref{thm:16}, we see that $H^6(X^\sigma_1,\bZ)_\van$ is mapped isomorphically
onto $H^2(X_6^\sigma,\bZ)_\prim$ via $q_*p^*L_h$, so we now have a complete
description of the image of $q_*p^*L_h$. We deduce that any vanishing
class $x\in H^6(X_1^\sigma,\bZ)$ such that $L_h(x)$
becomes divisible in $H^8(X_1^\sigma,\bZ)$ must lie in $\bZ h^3+\bZ\pi$.
Take one such class $x=ah^3+b\pi$ with $\gcd(a,b)=1$. We have $x\cdot
h^3=15a+7b$ and $x\cdot z\cdot h=6a+3b$. One verifies easily that for
$L_h(x)=x\cdot h$ to be divisible, the only possibility is when $(a,b)=(1,0)$
and the divisibility is $3$.  In other words, only the class $h^4$ is
potentially divisible by $3$ in $H^8(X_1^\sigma,\bZ)$, so the index
of the image~$L_hH^6(X_1^\sigma,\bZ)$ in~$H^8(X_1^\sigma,\bZ)$ is either
$1$ or $3$.

To show that we have indeed index $3$, we embed $X^\sigma_1$ into the flag
variety $\Flag(1,4,V_{10})$ by mapping $[V_1]$ to the pair $[V_1\subset K_4]$,
where $K_4$ is the kernel of the skew-symmetric form $\sigma(V_1,-,-)$. This
allows us to exhibit explicit algebraic classes on $X^\sigma_1$ using
(relative) Schubert classes in \Macaulay (see Appendix \autoref{m2:Lefschetz} where we
provide the code for all the following computations). In particular we
construct algebraic classes $s\coloneqq \sigma_{200}\in H^4(X_1^\sigma,\bZ)$
and $t\coloneqq \sigma_{400}\in
H^8(X_1^\sigma,\bZ)$. One may verify using the intersection numbers that
\[\pi=\tfrac13(6h^2-s)\cdot h\quad\text{and}
\quad t=\tfrac{17}3h^4-7\pi\cdot h.\]
The second equality shows that $h^4$ is indeed divisible by $3$ in
$H^8(X_6^\sigma,\bZ)$, which allows us to conclude. For the degree $4$ case,
notice that we have the degree of a Palatini threefold $(6h^2-s)\cdot
\frac13h^4=7$, so the class $6h^2-s$ is primitive while its image
$L_h(6h^2-s)=3\pi$ is divisible by~$3$.

Finally, for both Lefschetz operators, the extra~$3$-divisible class lies in
the algebraic part. So when we restrict to the vanishing parts, we get
isomorphisms.
\end{proof}
\begin{theorem}
\label{thm:HodgeConjX1}
When~$\sigma$ is such that ~$X^\sigma_1$ is smooth (that is,
when~$[\sigma]\notin\cD^{3,3,10}\cup\cD^{1,6,10}$), the integral Hodge
conjecture holds for $X_1^\sigma$ in all degrees.
\end{theorem}
\begin{proof}
For $H^4(X_1^\sigma,\bZ)$, we see that the image of $6h^2-s$ (where $s$ is as in Lemma \ref{lemma:Lefschetz})
is divisible by $3$. The subgroup~$\bZ h^2+\bZ s$ is therefore saturated
in~$H^4(X_1^\sigma,\bZ)$ by Lemma \ref{lemma:Lefschetz}. Thus we have a surjective map
\[
p_*q^*\colon H^6(X_6^\sigma,\bZ)\simeq \big(H^2(X_6^\sigma,\bZ)\big)^\vee\onto
(H^8(X_1^\sigma,\bZ)_\van)^\vee,
\]
and by considering the duality between~$H^4(X_1^\sigma,\bZ)$
and~$H^8(X_1^\sigma,\bZ)$, we get a surjective map
\[
H^4(X_1^\sigma,\bZ)\simeq \big(H^8(X_1^\sigma,\bZ)\big)^\vee\onto
(H^8(X_1^\sigma,\bZ)_\van)^\vee,
\]
whose kernel is precisely~$\bZ h^2+\bZ s$ and is therefore generated by
algebraic classes. Comparing the two, we get
\[
p_*q^*H^6(X^\sigma_6,\bZ)+(\bZ h^2+\bZ s)=H^4(X^\sigma_1,\bZ),
\]
which allows us to conclude.

For $H^8(X_1^\sigma,\bZ)$, %we proceed similarly by showing that the orthogonal of~$H^4(X_1^\sigma,\bZ)_\van$ is generated by algebraic classes. 
since we have already obtained the integral Hodge conjecture for
$H^6(X_1^\sigma,\bZ)$, in view of Lemma \ref{lemma:Lefschetz} and its proof it suffices to show that
$\frac13h^4$ is algebraic. By looking at the integral algebraic class
$t=\frac{17}3h^4-7\pi\cdot h$, we have $\frac13h^4=6h^4-7\pi\cdot h-t$ so it is
indeed algebraic. This concludes the proof.
\end{proof}

\section{The Heegner divisor of degree 24: a new twisted K3 surface}
\label{sec:D24}

% \subsection{The degeneracy condition}
In the GIT moduli space~$\cM$ of trivectors, we have the
divisor~$\cD^{1,6,10}$ given by trivectors satisfying the degeneracy condition
$\sigma(V_1,V_6,V_{10})=0$ as in \autoref{eq:1,6,10}, which is also the locus
where the Peskine variety~$X_1^\sigma$ becomes singular and generically admits
an isolated singularity at~$[V_1]$. In this last section, we will study the
geometry along this divisor.  Notably we will give the geometric construction
of a K3 surface~$S$ of degree~$6$ and a divisor~$D$ in~$X^\sigma_6$ ruled
over~$S$.

First we show that the divisor $\cD^{1,6,10}$ is mapped to the
Noether--Lefschetz divisor~$\cC_{24}$ under the modular map~$\m$ and further to
the Heegner divisor $\cD_{24}$ by the period map $\p$.  We state a lemma
which gives an alternative description for $\cD^{1,6,10}$.
\begin{lemma}
\label{lemma:188}
For a trivector $\sigma$, there is a flag $V_1\subset V_6$ such that
$\sigma(V_1,V_6,V_{10})=0$ if and only if there is a flag $V_1\subset V_8$
such that $\sigma(V_1,V_8,V_8)=0$. The locus $\{ [V_8]\in \Gr(8,V_{10})\mid \sigma(V_1,V_8,V_8)=0 \}$ is a three dimensional quadric.
\end{lemma}
\begin{proof}
If we have a flag $V_1\subset V_8$ as above, the skew-symmetric~$2$-form
$\sigma(V_1,-,-)$ is of rank at most 4, so there is a 6-dimensional
$V_6$ in the kernel.

Conversely, if we have a flag $V_1\subset V_6$ as in the lemma, the set of $V_8$ in
$\Gr(2,V_{10}/V_6)$ such that $\sigma(V_1,V_8,V_8)=0$ is exactly the set of
subspaces which are isotropic with respect to $\sigma(V_1,-,-)$, which is a
linear section of the quadric $\Gr(2,V_{10}/V_6)$.
\end{proof}

\begin{proposition}
\label{propD24}
The divisor $\cD^{1,6,10}$ is mapped birationally onto the Noether--Lefschetz
divisor~$\cC_{24}$ via the moduli map~$\m$, and then to the Heegner
divisor~$\cD_{24}$ via the period map~$\p$.
\end{proposition}

\begin{proof}
\label{rmk:c10}
As with the other two divisors, it suffices to compute the discriminant. The degeneracy condition $\sigma(V_1,V_8,V_8)=0$ shows that there is a
Grassmannian $\Gr(2,7)=\Gr(2,V_8/V_1)$ contained in $X^\sigma_3$.
Notice that the choice of $V_8$ is not canonical, as generally these $V_8$ are
parametrized by a 3-dimensional quadric for a fixed flag $V_1\subset V_6$, and
so are the Grassmannians $\Gr(2,7)$ contained in $X^\sigma_3$.

From now on, let us assume that $[\sigma]\notin \cD^{3,3,10}$, which implies that $X_3^\sigma$ is smooth. If we fix one $Z=\Gr(2,7)$ contained in $X^\sigma_3$ and look at its class $z\in
H^{20}(X^\sigma_3,\bZ)$, we may compute the self-intersection number
\[
z^2=c_{10}(\cN_{Z/X^\sigma_3})=2.
\]
Indeed, using the two normal sequences
\begin{gather*}
0\to \cT_Z\to \cT_{X^\sigma_3}|_Z\to \cN_{Z/X^\sigma_3}\to 0,\\
0\to \cT_{X^\sigma_3}\to \cT_{\Gr(3,V_{10})}|_{X^\sigma_3}\to \cO_{X_3^\sigma}(1)\to 0,
\end{gather*}
the normal bundle $\cN_{Z/X^\sigma_3}$ can be expressed in terms of homogeneous
vector bundles $\cU_2$ and $\cQ_5$ on $Z=\Gr(2,7)$, so we may calculate
explicitly its Chern classes using Schubert calculus (see Appendix \autoref{m2:c10}). Moreover, notice that $j_*z=\sigma_{722}$ is a Schubert class. So we can
compute the full intersection matrix for the lattice $\bZ z+ j^*H^{20}
(\Gr(3,V_{10}), \bZ)$ (using again Schubert calculus and intersecting with the hyperplane class; this computation is analogous to the one that led to \autoref{int_matrix} and can be done using \Macaulay) and find that its determinant is 24. %It is possible to show that the discriminant group is $\bZ/24\bZ$ by explicitly finding a vector of divisibility 24.
\end{proof}

\begin{remark}
\label{rmqunirationality24}
This result implies that $\cD_{24}$ is unirational. Indeed,
$\cD^{1,6,10}$ can be seen as a quotient of the vector bundle $(\cU_1\wedge
\cU_6\wedge V_{10})^\perp \subset \bigwedge^3 V_{10}^\vee \otimes \cO_F$ over
the variety $F\coloneqq\Flag(1,6,V_{10})$ by the natural action of the group
$\SL(V_{10})$.
\end{remark}

\begin{proposition}
\label{unirat_div}
$\cD_{22}$, $\cD_{24}$, $\cD_{28}$ are unirational.
\end{proposition}

\begin{proof}
It directly follows from the fact that $\cD^{3,3,10}$, $\cD^{1,6,10}$, $\cD^{4,7,7}$ are unirational by Remarks \ref{rmqunirationality22}, \ref{rmqunirationality24}, \ref{rmqunirationality28}.
\end{proof}

% We also prove the following lemma concerning the smoothness of the incidence
% variety~$I^\sigma_{1,6}$.
% \begin{lemma}
% When~$X_1^\sigma$ is smooth of expected dimension~$6$, the incidence
% variety~$I^\sigma_{1,6}$ in~\autoref{eq:I16} is also smooth of expected
% dimension~$7$.
% \end{lemma}
% \begin{proof}
% We just saw that when~$X_1^\sigma$ is smooth, each~$[V_1]\in X_1^\sigma$
% satisfies~$\rank\sigma(V_1,-,-)=6$, and there is no~$V_3$
% with~$\sigma(V_3,V_3,V_{10})=0$. Moreover~$X_1^\sigma$ can be identified as a
% zero locus in~$\Flag(1,4,V_{10})$.

% For a pair~$[V_1\subset V_6]\in I^\sigma_{1,6}$, consider the kernel~$K_4$
% of~$\sigma(V_1,-,-)$: we have~$\sigma(V_1,K_4+V_6,K_4+V_6)=0$. By the above
% \autoref{lemma:188}, we know that~$\dim(K_4+V_6)\le 7$ so~$K_4\cap V_6$ has
% dimension~$\ge 3$. This allows us to identify~$I^\sigma_{1,6}$ as a subvariety of a
% certain flag variety as follows: let~$F\coloneqq\Flag(1,3,6,V_{10})$, and
% let~$P\coloneqq\bP_F(V_{10}/\cU_3)$ be the projectivization of the quotient
% bundle~$V_{10}/\cU_3$ on~$F$, we may construct a morphism from~$I^\sigma_{1,6}$ to~$P$
% by sending~$[V_1\subset V_6]$ to~$([V_1\subset K_4\cap V_6\subset V_6],[K_4])$
% \end{proof}

\subsection{Construction of a K3 surface and a uniruled divisor}
\label{secconstrK3surf}

From now on, we let $\sigma\in\bw3 V_{10}^\vee$ be a general trivector in the
divisor $\cD^{1,6,10}$, so there is a unique distinguished flag $[V_1\subset V_6]$
such that $\sigma(V_1,V_6,V_{10})=0$. We study the geometry of the Debarre--Voisin
variety $X^\sigma_6$, which resembles a lot that of a cubic fourfold
containing a plane (see Appendix \ref{app_cubic_fourfolds_containing}).  Notably, we will construct a K3 surface $S$ of degree $6$
and a uniruled divisor $D$ in $X^\sigma_6$ that admits a $\bP^1$-fibration
over~$S$. In later sections, we will compare the Hodge structures of
$X^\sigma_6$ and of the K3 surface $S$. The $\bP^1$-fibration
defines a non-trivial Brauer class $\beta\in\Br(S)$, and we will show that
$X^\sigma_6$ can be recovered as a moduli space of $\beta$-twisted sheaves on
$S$ (which is proved in a purely Hodge theoretical way). Let us start with the following result. We will use the orbit closure $Y\subset \wedge^3 W_7^\vee$ introduced in Appendix \ref{special_orbit_gl7}.
\begin{proposition}
\label{propK3S}
Suppose the trivector $\sigma\in\bw3 V_{10}^\vee$ is general in the divisor
$\cD^{1,6,10}$, that is, we have $\sigma(V_1,V_6,V_{10})=0$ for a unique flag $V_1\subset
V_6\subset V_{10}$. Then it defines a smooth K3 surface $S$ of degree $6$ inside
$\Gr(2,V_{10}/V_6)$, where the polarization is given by the Plücker line
bundle.
\end{proposition}

\begin{proof}
The Grassmannian $\Gr(2,V_{10}/V_6)$ is a $4$-dimensional quadric. The K3
surface~$S$ will be the intersection of a linear section and a cubic
section of this quadric, hence the Plücker line bundle will be of degree $6$.
For clarity, we denote by $\cU_{8/6}$ the tautological
subbundle and by $\cQ_{10/8}$ the quotient bundle on $\Gr(2,V_{10}/V_6)$
respectively.

The linear section is given by $\sigma(V_1,V_8,V_8)=0$: since
$\sigma(V_1,V_6,V_{10})=0$, this is equivalent to the condition
$\sigma(V_1,V_8/V_6,V_8/V_6)=0$, which can be seen as the vanishing of a general section of the
line bundle $\bw2\cU_{8/6}^\vee\simeq\cO(1)$. The zero-locus is therefore a
$3$-dimensional quadric $S'$. Now for each $[V_8/V_6]\in S'\subset \Gr(2,4)$, since we have $\sigma(V_1,V_8,V_8)=0$, the form
$\sigma$ induces an element of $\bw3(V_8/V_1)^\vee$. In the relative setting,
by letting $\cW_7\coloneqq V_6/V_1\oplus\cU_{8/6}$ where $V_6/V_1$ is the trivial
bundle $(V_6/V_1)\otimes \cO_{\Gr(2,V_{10}/V_6)}$, we get a global section
$\sigma'$ of the vector bundle $\bw3\cW_7^\vee$. So we may define the orbital
degeneracy locus
\[
S\coloneqq D_{Y}(\sigma')=\setmid{[V_8/V_6]\in S'}{\sigma'|_{V_8/V_6}\in Y\subset
\bw3(V_8/V_1)^\vee\simeq\big(\bw3\cW_7^\vee\big)_{[V_8/V_6]}}.
\]
As we have already seen, the hypersurface $Y$ is defined in $\bw3W_7^\vee$ by
the vanishing of the discriminant. Therefore $S$ is the hypersurface in $S'$
defined by the vanishing of $\disc\sigma'$, which is a
section of $\det(\cW_7^\vee)^{\otimes 3}\simeq\cO_{S'}(3)$. As $\sigma$ is
general, so is $\sigma'$ among sections of $\bw3\cW_7^\vee$: to be more precise, $S'$ is defined by a (non generic) element $\eta\in (V_1\otimes \bw2 V_{10})^\vee $; having fixed such an $\eta$ (and hence $S'$), the set of elements $\sigma\in \bw3 V_{10}^\vee$ whose image in $ (V_1\otimes \bw2 V_{10})^\vee$ is equal to $\eta$ is isomorphic to $\bw3 (V_{10}/V_1)^\vee$. This space clearly dominates the space of sections of $\bw3\cW_7^\vee$, which is a globally generated vector bundle already over $\Gr(2,V_{10}/V_6)$; thus $\sigma$ lives in a linear system globally generating $\bw3\cW_7^\vee$ over $S'$.  Thus we can apply a Bertini-type theorem for
orbital degeneracy loci (see~\cite{benedetti}), and since the
hypersurface $Y$ is smooth in codimension $2$, the zero-locus $S$ is a smooth
surface obtained as the intersection of a quadric and a cubic, that is, a
degree $6$ K3 surface.
\end{proof}

\begin{remark}
\label{remK3S}
As pointed out in \autoref{remorbclosureY1}, any general element $y\in Y$
defines a unique point $[W_3]\in \Gr(3,W_7)$. In the relative setting, this
implies that a general point $[V_8/V_6]$ of $S$ defines a
$3$-dimensional subspace of $V_8/V_1$, in other words a $4$-dimensional
subspace $V_4$ with $V_1\subset V_4\subset V_8$ such that
\[
\sigma'(V_4/V_1,V_4/V_1,V_8/V_1)=0\text{ or equivalently,
}\sigma(V_4,V_4,V_8)=0.
\]

In conclusion, having fixed the flag $V_1\subset V_6$ and the trivector
$\sigma\in \cD^{1,6,10}$, the K3 surface $S$ can also be defined as the set
\begin{equation}
\label{eqdefS1}
S=\setmid{[V_4\subset V_8]}{\begin{array}{c}V_1\subset V_4,\ \sigma(V_1,V_8,V_8)=0,\\
\text{and }\sigma(V_4,V_4,V_8)=0\end{array}}
\subset \Flag(4,8,V_{10}).
\end{equation}
The advantage of this description is that $S$ can now be characterized as the
zero-locus of a section of some vector bundle on a flag variety.
In particular, we get two more tautological bundles
$\cU_{4/1}\coloneqq\cU_4/V_1$ and $\cU_{8/4}\coloneqq \cU_8/\cU_4$
on $S$ induced by the inclusion $S\subset \Flag(4,8,V_{10})$ ($\cU_4$ and $\cU_8$ being the rank four and rank eight tautological bundles over the flag variety $\Flag(4,8,V_{10})$).
By construction, the line bundle $\det(\cU_{8/6}^\vee)$ gives the degree $6$
polarization $\cO_S(1)$ on $S$. One may verify numerically that
$\det(\cU_{4/1}^\vee)\simeq \cO_S(3)$ and $\det(\cU_{8/4})\simeq \cO_S(2)$ so
no new polarizations are produced this way (the program to verify this numerical computation is contained in Appendix \autoref{m2:K3degree}). In fact, we will see later that the family of polarized K3 surfaces of degree
$6$ parametrized by $\cD_{24}$ is a locally complete family, as a consequence
of the study of their Hodge structures. Hence we have $\Pic(S)\simeq\bZ$ for a
very general member of the family.
\end{remark}

Next, we construct a uniruled divisor $D$ in $X^\sigma_6$.
\begin{proposition}\label{prop:uniruled_divisor}
For a general~$\sigma$ in the divisor~$\cD^{1,6,10}$,
the set
\[D\coloneqq\setmid{[U_6]\in X^\sigma_6}{\exists [V_4\subset V_8]\in S\quad
V_4\subset U_6\subset V_8}\]
defines a divisor in $X^\sigma_6$ which has a smooth conic fibration $\pi\colon
D\to S$ over the K3 surface $S$.
\end{proposition}
\begin{proof}
First we construct the morphism $\pi\colon D\to S$ by showing that for
each~$[U_6]\in D$, the corresponding~$[V_4\subset V_8]\in S$ is unique.
We claim that each $[U_6]\in D$ satisfies $\dim (U_6\cap V_6)=4$. Otherwise,
suppose that there exists some $U_6$ with $U_5\subset (U_6\cap V_6)$, $\dim(U_5)=5$. Then we may
look for a $V_3$ with $V_1\subset V_3\subset U_5\subset U_6\cap V_6$ such that
$\sigma(V_3,V_3,V_{10})=0$. This is a codimension-$4$ condition on $\Gr(2,U_5/V_1)$ given by the vanishing of $\sigma$ restricted to $(\det(\cU_2)\otimes (V_{10}/U_6))^\vee$, which as a bundle over $\Gr(2,4)$ is isomorphic to $\cO(1)^4$; since the top Chern class of $\cO(1)^4$ does not vanish, such a $V_3$ must exist. For a general $\sigma$ in
$\cD^{1,6,10}\setminus \cD^{3,3,10}$ this will not happen, so we have
$\dim(U_6\cap V_6)=4$. We may then recover $V_8$ as the sum $U_6+V_6$ and get a
morphism $\pi\colon D\to S'$. 

Let us prove now that the image of $D$ via $\pi$ actually lands inside $S$. This image is contained inside $S$ since it is always possible to find $[V_4]\in \Gr(3,U_6/V_1)$ such that $\sigma(V_4,V_4,V_8)=0$; indeed, $\Gr(3,U_6/V_1)\cong \Gr(3,5)$, and since $\sigma(V_4,V_4,U_6)=\sigma(V_1,V_8,V_8)=0$, the section $\sigma$ restricted to $(\wedge^2 U_4\otimes V_8/U_6)^\vee$ defines a section of $ (\wedge^2 \cU_3\otimes \cO_{\Gr})^\vee$ over $\Gr(3,5)$, whose zero locus is always non-empty (the reader may check that the top Chern class of this bundle does not vanish). 

Now we show that this morphism~$\pi\colon D\to S$ is a smooth conic fibration.
We first study the fiber $\setmid{[U_6]\in X^\sigma_6}{V_4\subset U_6\subset
V_8}$ above each $[V_4\subset V_8]\in S$. This fiber can be seen as the locus
\[\setmid{[U_6]\in\Gr(2,V_8/V_4)}{\sigma|_{U_6}=0}.\]
The trivector $\sigma$, when restricted to $V_8$, becomes a section of
$(V_4/V_1)^\vee\otimes\bw2\cU_2^\vee=\cO(1)^{\oplus3}$. So we get three
hyperplane sections, whose intersection in $\Gr(2,V_8/V_4)$ is generically a conic.

In the relative setting, the fibers are defined inside the
projectivization~$\bP_S(\cE)$ of a rank-$3$ vector bundle $\cE$ over $S$, which
is realized as the kernel
\[ \begin{tikzcd}
\cE\ar[r,hook]&\bw2\cU_{8/4}\ar[r,"\sigma"]&\cU_{4/1}^\vee\ar[r]&0.
\end{tikzcd} \]
Here the last arrow is surjective for a general~$\sigma$: otherwise we would
get a subspace~$U_2\supset V_1$ such that~$\sigma(U_2,V_8,V_8)=0$; then all
the~$U_1$ contained in~$U_2$ will have~$\rank\sigma(U_1,-,-)\le4$ which does not
happen for~$\sigma$ general.
The quadratic form $q$ on~$\cE$ is given by
\[ \begin{tikzcd}
\Sym^2\cE\ar[r,hook]&\Sym^2\bw2\cU_{8/4}\ar[r,"q"]&\cL\coloneqq\det\cU_{8/4}
\end{tikzcd}
\]
where it takes value in the line bundle $\cL$. We have $\det\cE
\simeq\cO_S(3)$ while $\cL\simeq\cO_S(2)$. The discriminant locus is defined by
a general section of the line bundle $(\det\cE^\vee)^{\otimes2}\otimes\cL^{\otimes
\rank\cE}\simeq\cO_S$ and is therefore empty. Thus the conic fibration is
everywhere smooth.
\end{proof}

We also have an alternative description of $D$.
\begin{proposition}
For a general~$\sigma$ in the divisor~$\cD^{1,6,10}$, let $[V_1\subset V_6]$ be
the distinguished flag.  A point $[U_6]\in X^\sigma_6$
is contained in $D$ if and only if $U_6$ contains $V_1$. In other words, we
have
\[
D=\setmid{[U_6]\in X^\sigma_6}{U_6\supset V_1}.
\]
\end{proposition}
\begin{proof}
Since any $[U_6]$ in $D$ contains a subspace $V_4\supset V_1$, one direction is
evident. Suppose now that $U_6$ contains $V_1$ and $\sigma|_{U_6}=0$. Then $U_6$ is
isotropic with respect to $\sigma(V_1,-,-)$ and is contained inside a
maximal isotropic subspace $V_8$ of dimension eight. Let us consider a
point $[V_4]\in \Gr(3,U_6/V_1)$. Since $[U_6]\in X^\sigma_6$, any such $V_4$
satisfies $\sigma(V_4,V_4,U_6)=0$. The condition
$\sigma(V_4,V_4,V_8)=0$ is a codimension-$6$ condition since we already know that $\sigma(V_4,V_4,U_6)=\sigma(V_1,V_8,V_8)=0$ and thus the condition is equivalent to the vanishing of the restriction of $\sigma$ to $\wedge^2 V_4/V_1 \otimes V_8/U_6$; then, the existence of $V_4$ follows from the fact that there exists
at least one point $[V_4]\in \Gr(3,U_6/V_1)$ satisfying $\sigma|_{\wedge^2 V_4/V_1 \otimes V_8/U_6}=0$, which in turn follows from the fact that the bundle
$(\bw2 \cU_3^\vee)\otimes(\cO_{\Gr})^\vee$ over $\Gr(3,5)$ has non vanishing top Chern class (a computation easily done with \cite{m2}). This implies that $[V_4\subset V_8]$ is a point of $S$ as in \autoref{eqdefS1} and
therefore $[U_6]\in D$.
\end{proof}

\subsection{Hodge structures}
\label{sec:hodgestr}

Denote by $i\colon D\into X^\sigma_6$ the embedding of the divisor $D$
constructed above.  By abuse of notation, we denote the class $[D]\in
H^2(X^\sigma_6,\bZ)$ also by $D$. We first compute the intersection matrix for
the classes $H$ and $D$ under the Beauville--Bogomolov--Fujiki form $\q$.
Note that for $\sigma$ very general in $\cD^{1,6,10}$, the algebraic
sublattice $H^2(X^\sigma_6,\bZ)_\alg$ is generated by $H$ and $D$ over $\bQ$.
% Note that most of the results in this section only holds for very general~$\sigma$
% in~$\cD^{1,6,10}$, it is straightforward to check that

\begin{lemma}
\label{lemma:intersectionDH}
The intersection matrix between $H$ and $D$ is
\[\begin{pmatrix}22&2\\2&-2\end{pmatrix}.\]
So $H$ and $D$ span a lattice of discriminant $-48$ and generate the algebraic
sublattice~$H^2(X_6^\sigma,\bZ)_\alg$ over~$\bZ$.
\end{lemma}
\begin{proof}
By the adjunction formula and the fact that $X^\sigma_6$ has trivial canonical
bundle, the canonical class $K_D$ of the divisor $D$ is the restriction $i^*D$.
One can then compute explicitly the intersection numbers using Schubert
calculus in \Macaulay (see Appendix \autoref{m2:intersectionDH} for the code):
\[
D^4=12,\ D^3\cdot H=-12,\ D^2\cdot H^2=-36,\
D\cdot H^3=132,\ H^4=1452.
\]
Then we use the property \autoref{eq:bbf} of the Beauville--Bogomolov--Fujiki form
to obtain the desired numbers.

Since the divisors~$\cD^{1,6,10}$ is mapped to the Heegner divisor~$\cD_{24}$
by the period map, we may again use the fact that the discriminant of the
algebraic lattice is twice that of its orthogonal from Lemma~\ref{lemma:disc_K_two_times} to conclude that the algebraic sublattice is
generated by~$H$ and~$D$.
\end{proof}

We have the following useful result.
\begin{corollary}
\label{lemDnotdivisibility2}
The class $D$ has divisibility $1$, that is, there exists $C\in
H^2(X^\sigma_6,\bZ)$ such that $\q(C,D)=1$.
\end{corollary}
Note that the class~$C$ is not algebraic for very general $\sigma$ in the
family.
\begin{proof}
Recall that we denoted by $\Lambda$ the lattice $(H^2(X_6^\sigma,\bZ),\q)$. Suppose that~$D$ has divisibility~$2$, then the class~$[D/2]$ has
order~$2$ in the discriminant group~$D(\Lambda)\coloneqq\Lambda^\vee/\Lambda$,
which is isomorphic to~$\bZ/2\bZ$. In particular, the class~$[D/2]$ coincides with~$[H/2]$. This shows that the class~$[(H+D)/2]$ is trivial
in~$D(\Lambda)$, so $\frac12(H+D)$ is integral, which contradicts
the fact that~$\bZ H+\bZ D$ is saturated in~$\Lambda$ by Lemma \ref{lemma:intersectionDH}.
\end{proof}

Now we would like to compare the Hodge structures on $H^2(X^\sigma_6,\bZ)$ and
$H^2(S,\bZ)$.  Consider the diagram
\[ \begin{tikzcd}
H^2(X^\sigma_6,\bZ)\ar[r,"i^*"]&H^2(D,\bZ)\\
&H^2(S,\bZ)\ar[u,hook,"\pi^*"]
\end{tikzcd}\]
induced by the inclusion $i: D \hookrightarrow X_6^\sigma$ and the projection $\pi : D \to S$.
The idea is to make the comparison inside~$H^2(D,\bZ)$.

\begin{lemma}
\label{lemma:H2D}
For~$\sigma$ general in the divisor~$\cD^{1,6,10}$,
there exists a class $\zeta\in H^2(D,\bZ)$ such that
\[
H^2(D,\bZ)=\pi^*H^2(S,\bZ)\oplus \bZ\zeta.
\]
Let $H^2(X^\sigma_6,\bZ)^{\perp D}\subset H^2(X^\sigma_6,\bZ)$ denote the
orthogonal of $D$ with respect to $\q$. Then
\[
i^*(H^2(X^\sigma_6,\bZ)^{\perp D})\subset \pi^*H^2(S,\bZ).
\]
\end{lemma}
\begin{proof}
As we saw in \autoref{prop:uniruled_divisor}, the natural projection $\pi\colon
D\to S$ is a smooth conic fibration over the K3 surface $S$. Denote by $l\in
H_2(D,\bZ)$ the class of a fiber of $\pi$.  We have $i^*H\cdot l=2$ since $\pi$
is a conic fibration, and $i^*D\cdot l=-2$ since $i^*D$ is the canonical
divisor of~$D$.

Consider $i_*l\in H_2(X^\sigma_6,\bC)$ as the class of a rational curve on
$X^\sigma_6$ which is of type $(3,3)$. There exists a unique element $y\in
H^2(X^\sigma_6,\bQ)$ such that
\[\forall x\in H^2(X_6^\sigma,\bZ)\quad\q(x,y)=x\cdot i_*l.\]
Moreover $y$ must be of type $(1,1)$ so it is a $\bQ$-linear combination of $H$
and $D$. Since
\begin{gather*}
\q(D,D)=-2=i^*D\cdot l=D\cdot i_*l,\\
\q(H,D)=2=i^*H\cdot l=H\cdot i_*l,
\end{gather*}
we see that $y=D$. By \autoref{lemDnotdivisibility2}, there exists a
class $C\in H^2(X^\sigma_6,\bZ)$ such that $\q(C,D)=1$. We have $i^*C\cdot
l=C\cdot i_*l=\q(C,D)=1$, so the class $i^*C$ restricts to $\cO(1)$ on each
fiber $l$ of $\pi$. By the Leray--Hirsch theorem, the classes $1$ and $i^*C$
generate $H^*(D,\bZ)$ as a $\pi^*H^*(S,\bZ)$-module, hence we have
\[
H^*(D,\bZ)=\pi^*H^*(S,\bZ)\oplus \pi^* H^*(S,\bZ)(i^*C),
\]
and in particular
\[
H^2(D,\bZ)=\pi^*H^2(S,\bZ)\oplus \bZ(i^*C).
\]
We may choose~$i^*C$ as the class $\zeta$ that we want.
For each class in $H^2(D,\bZ)$, its coefficient before $i^*C$ is simply
its intersection number with the fiber $l$. Finally, any class $x\in H^2(X^\sigma_6,\bZ)$ with $\q(x,D)=0$ must satisfy
$i^*x\cdot l=x\cdot i_*l=0$. This shows that $i^*(
H^2(X^\sigma_6,\bZ)^{\perp D})$ is indeed contained in $\pi^* H^2(S,\bZ)$.
\end{proof}

Let $\q_S$ denote the intersection product on $S$ pulled back to
$\pi^*H^2(S,\bZ)$ via~$\pi^*$. By the previous lemma, $\q_S$ induces a form on
$H^2(X^\sigma_6,\bZ)^{\perp D}$ via $i^*$ and we can compare it
with $\q$.

\begin{proposition}
\label{propcompquadraticforms}
Let~$\sigma$ be general in the divisor~$\cD^{1,6,10}$.
For any $x\in  H^2(X^\sigma_6,\bZ)^{\perp D}$, we have
\[
\q(x,x)=\q_S(i^*x,i^*x).
\]
As a consequence, the morphism $i^*$ is injective.
\end{proposition}
\begin{proof}
Consider the class $i^*D\in H^2(D,\bZ)$. By \autoref{lemma:H2D},
since $i^*D\cdot l=-2$, we have $i^*D=-2\zeta+\pi^*y$ for some $y\in
H^2(S,\bZ)$. So the intersection number $i^*D\cdot i^*x\cdot i^*x$ is equal to
$-2\q_S(i^*x,i^*x)$.  On the other hand, we have $i^*D\cdot i^*x\cdot
i^*x=\int_{X^\sigma_6} D^2x^2$ which is equal to $\q(D,D)\q(x,x)+2\q(x,D)^2$.
Since $\q(x,D)=0$ and $\q(D,D)=-2$, we get the equality
$\q(x,x)=\q_S(i^*x,i^*x)$.

This shows that $i^*$ is injective when restricted to $
H^2(X^\sigma_6,\bZ)^{\perp D}$. But any $x$ such that $i^*x=0$ will satisfy
$i^*x\cdot l=0$ so we have $\q(x,D)=0$ and hence $x=0$. Thus $i^*$ itself is
injective.
\end{proof}

We see that the lattice $\big(H^2(X^\sigma_6,\bZ)^{\perp D},\q\big)$ embeds
isometrically inside $\big(H^2(S,\bZ),\,\cdot\,\big)$. It remains to determine the
index of the embedding.

\begin{theorem}
\label{thmindex2sublattice}
For~$\sigma$ very general in the divisor~$\cD^{1,6,10}$,
there is an embedding of integral Hodge structures
\[\iota\colon\big(H^2(X^\sigma_6,\bZ)^{\perp D},\q\big)\into\big(H^2(S,\bZ),\,\cdot\,\big)\]
as a sublattice of index $2$. We have
\[\iota(H+D)=2h,\]
where $h$ is the polarization on $S$ of degree $6$. Restricted to the
transcendental part, we get
\[\iota\colon\big(H^2(X^\sigma_6,\bZ)_\trans,\q\big)\into\big(H^2(S,\bZ)_\prim,\,\cdot\,\big)\]
again of index $2$.
\end{theorem}

\begin{proof}
By \autoref{lemDnotdivisibility2}, the class $D$ is of divisibility~1 in
$H^2(X^\sigma_6,\bZ)$. Since $H^2(X_6^\sigma,\bZ)$ has discriminant equal to two, the orthogonal $H^2(X_6^\sigma,\bZ)^{\perp D}$ is of discriminant twice the discriminant of $\bZ D$, i.e. of discriminant equal to $2\q(D,D)=-4$.
%As its orthogonal, the sublattice $H^2(X^\sigma_6,\bZ)^{\perp D}$ is of discriminant $4$. 
On the other hand,
$H^2(S,\bZ)$ is unimodular. Hence by comparing discriminants, the first
statement follows.

Since the $(1,1)$ part of $H^2(X^\sigma_6,\bZ)^{\perp D}$ is generated by the
class $H+D$ with square $\q(H+D,H+D)=24$ (see Lemma \ref{lemma:intersectionDH}), we have that $\iota(H+D)=\pm 2h $; since the generator $H+D$
is effective, it is clear that $i^*(H+D)$ must be equal to $2\pi^*h$, so
$\iota(H+D)=2h$.

Finally, the second embedding follows by looking at the respective orthogonals
of these two classes, while the index $2$ is again obtained by comparing
discriminants.
\end{proof}

It is possible to get a more precise description of the index-$2$ sublattice. We
first define a class $A$ in $H^2(S,\bZ)$ as follows. Consider the class $C$ as
in \autoref{lemDnotdivisibility2}. Since $\q(H-2C,D)=0$, we define $A$ to be
the image $\iota(H-2C)\in H^2(S,\bZ)$. Notice that, since $\q(H,H)=22$,
$\q(C,C)$ is even and $H$ is of divisibility two,
\[A\cdot A=\q(H-2C,H-2C)\equiv 6\pmod8,\]
so $A$ is not divisible by $2$.
\begin{proposition}
\label{prop:index2_embed}
For~$\sigma$ very general in the divisor~$\cD^{1,6,10}$,
the lattice $H^2(X^\sigma_6,\bZ)^{\perp D}$ can be identified via the
embedding $\iota$ as the sublattice
\[\Lambda_{\frac{1}{2}A}\coloneqq\setmid{u\in H^2(S,\bZ)}{u\cdot
A\in2\bZ},\]
while the sublattice $H^2(X^\sigma_6,\bZ)_\trans$ can be identified as the sublattice
\[\Lambda_{\frac{1}{2}A,\prim}\coloneqq\setmid{u\in
H^2(S,\bZ)_\prim}{u\cdot A\in2\bZ}.\]
\end{proposition}
\begin{proof}
For each class $x\in H^2(X^\sigma_6,\bZ)^{\perp D}$, the intersection number
\[\iota(x)\cdot A=\q(x,H-2C)=\q(x,H)-2\q(x,C)\]
is always even, because $\div(H)=2$. So we get the inclusion in one
direction. For the other direction: since the index is $2$, it suffices to show
that $\Lambda_{\frac{1}{2}A,\prim}$ is a proper sublattice of
$H^2(S,\bZ)_\prim$, and the sublattice $\Lambda_{\frac{1}{2}A}$ will then also
be proper in $H^2(S,\bZ)$.

Thus we search for a class $v\in H^2(S,\bZ)_\prim$ with $v\cdot A$ odd. First
we claim that $h\cdot A$ is odd: this is equivalent to
\[2h\cdot A=\q(H+D,H-2C)=22-2\q(C,H)\equiv2\pmod4,\]
which follows from the fact that $\div(H)=2$. Now we can let $h=e_1+3f_1$
for $(e_1,f_1)$ a standard basis for a copy of the hyperbolic plane $U$ in
$H^2(S,\bZ)$, and take $v\coloneqq e_1-3f_1\in H^2(S,\bZ)_\prim$.
Since $(h+v)\cdot A=2e_1\cdot A$ is even, the intersection number $v\cdot
A$ is odd.
\end{proof}

We explain in the next section the interpretation of $A$ in terms of a
B-field lifting of a Brauer class $\beta\in\Br(S)$.

\subsection{Twisted K3 surfaces and moduli space of twisted sheaves}

As already pointed out in the introduction, the divisor $\cD_{24}$ is precisely
the Heegner divisor $\cD^{(1)}_{24,24,\beta}$ in
\cite{KapuvanGeemen}.
This follows from the fact that $B\cdot B=B\cdot h =\frac{1}{2}$, as we will show in
\autoref{lemma:BB_Bh}. The intersection matrix appearing in
\autoref{lemma:intersectionDH} is diagonalized in the basis $\langle H+D,D\rangle$,
where it becomes $\begin{psmallmatrix}24&0\\0&-2\end{psmallmatrix}$; therefore
the class $H+D$ gives the contraction of the conic bundle $D\to S$
\cite[Proposition~3.5]{KapuvanGeemen}. Thus
\autoref{thmmodulitwistedsheaves}, that we are going to prove in a direct way and is the main result of this section, could also have been deduced by combining
\cite[Proposition~3.5]{KapuvanGeemen} with \cite[Theorem~4.2]{KapuvanGeemen}.
Notice finally that a crucial element for our result is
\autoref{thmindex2sublattice}, which is a particular case of
\cite[Proposition~4.6]{KapuvanGeemen}.

We first recall the notions of Brauer group and B-field lifting. We will only be
interested in the case of K3 surfaces. We follow~\cite[Chapter~18]{huyK3} (see
also~\cite{vangeemen}).

The {\em Brauer group} of a K3 surface $S$ can be characterized as the
cohomology groups
\[\Br(S)\simeq H^2_\et(S,\bG_m)\simeq H^2(S,\cO_S^*)_\tors\]
in the algebraic and analytic context respectively.
Since $H^3(S,\bZ)=0$, the exponential sequence allows us to have another
description
\[\Br(S)\simeq \big(H^2(S,\bZ)/\NS(S)\big)\otimes(\bQ/\bZ)
\simeq\Hom(\NS(S)^\perp,\bQ/\bZ).\]
For an element $\beta$ of $\Br(S)$, a representative $B\in H^2(S,\bQ)$ is
called a {\em B-field lifting} of $\beta$.

Each $\beta$ of order $n$ in the Brauer group gives a morphism from the
transcendental part $\NS(S)^\perp$ (when $S$ is of Picard rank 1, this
coincides with $H^2(S,\bZ)_\prim$) to $\bQ/\bZ$, and the kernel is a sublattice
of index $n$, which in particular does not depend on the choice of the B-field
lifting.

In a more geometric setting, each class $\beta$ gives a {\em Brauer--Severi
variety} $\pi\colon X\to S$, which is a projective $\bP^{n-1}$-fibration that is locally
trivial in the étale topology. Equivalently, it is the projectivization $\bP(\cE)$
of some $\beta$-twisted vector bundle $\cE$ on $S$. We refer to~\cite{hs} for
the definitions of $\beta$-twisted coherent sheaves as well as the {\em twisted
Chern classes} $c_i^B(E)$ and the {\em twisted Chern character} $\ch^B(E)$. We
only emphasize that the definition of twisted Chern classes depends not just
on $\beta$ but also on the choice of a B-field lifting $B$.

Notice that $A$ is not in $\NS(S)\otimes \bQ$. Indeed, if $A=xh$ then $\iota(H-2C)=A=\iota(H+D)x/2$, then $C$ is an element in $H^2(S,\bZ)$ generated by $H$ and $D$ s.t. $\q(C,D)=1$; but $\langle H,D\rangle$ is primitive and any integral element in $\langle H,D\rangle$ clearly has even intersection with $D$ (from their intersection matrix). Thus, back to the situation of \autoref{sec:hodgestr}, the class $\frac{1}{2}A$ gives
a Brauer class $\beta$ of order $2$, and the lattice
$\Lambda_{\frac{1}{2}A,\prim}$ is the index-$2$ sublattice defined by $\beta$, as
explained above. This is the reason why we adopted the notation
$\Lambda_{\frac{1}{2}A}$ instead of $\Lambda_A$. We first find another B-field lifting that is easier to work with.
\begin{lemma}
\label{lemma:BB_Bh}
There exists another B-field lifting $B$ of the same Brauer class $\beta$
such that $B\cdot B= B\cdot h=\frac{1}{2}$.
\end{lemma}
\begin{proof}
Recall that $\frac{1}{2}A\cdot\frac{1}{2}A \equiv\frac{1}{2}A\cdot
h\equiv\frac{1}{2}\in\bQ/\bZ$.  We try to find $B$ by adding integral
classes to $\frac{1}{2}A$.
Denote by $(e_1,f_1)$ and $(e_2,f_2)$ standard bases of two copies of $U$
inside $H^2(S,\bZ)$. Suppose that $h=e_1+3f_1$ and $A=a e_1+b f_1+ce_2+df_2$.
By adding $e_1$ and $f_1$, we can reduce the coefficients of $e_1$ and $f_1$ in
$\frac{1}{2}A$ to $0$ or $\frac{1}{2}$. The condition $\frac{1}{2}A\cdot
h\equiv\frac{1}{2}\in\bQ/\bZ$ shows that we have either $\frac{1}{2}e_1+0f_1$
or $0e_1+\frac{1}{2}f_1$. We may do the same for $e_2$ and $f_2$, and the
condition $\frac{1}{2}A\cdot\frac{1}{2}A\equiv\frac{1}{2}\in\bQ/\bZ$ shows
that we always get $\frac{1}{2}e_2+\frac{1}{2}f_2$. In the two
possible situations, we may choose $B$ to be equal to either
$\frac{1}{2}e_1-f_1+\frac{1}{2}e_2+\frac{3}{2}f_2$ or
$0e_1+\frac{1}{2}f_1+\frac{1}{2}e_2+\frac{1}{2}f_2$.
\end{proof}

So the identifications in \autoref{prop:index2_embed} become
\[\iota\colon H^2(X^\sigma_6,\bZ)^{\perp D}\simto\Lambda_B\coloneqq\setmid{u\in
H^2(S,\bZ)}{u\cdot B\in\bZ}\]
and
\[\iota\colon H^2(X^\sigma_6,\bZ)_\trans\simto\Lambda_{B,\prim}\coloneqq\setmid{u\in
H^2(S,\bZ)_\prim}{u\cdot B\in\bZ}.\]

Since $2B\cdot B=1\in\bZ$, or equivalently $2B\in \Lambda_B$, we may set $2B=\iota(x_0)$
for some $x_0\in H^2(X^\sigma_6,\bZ)^{\perp D}$. This gives the following lemma that we
will need shortly.

\begin{lemma}\label{D-x-divisible}
The class $D-x_0$ is divisible
by~$2$ in $H^2(X^\sigma_6,\bZ)$.
\end{lemma}
\begin{proof}
We remark that when we write $B=\frac{1}{2}A+u$, with $u\in H^2(S,\bZ)$,
we have
\[2=2B\cdot 2B=A\cdot A+4A\cdot u+4u\cdot u.\]
Since $A\cdot A\equiv6\pmod8$ while $u\cdot u$ is always even, we see that
$A\cdot u$ is odd. In particular $u\ne0$, so $B\ne\frac{1}{2}A$. Also, the
intersection number $A\cdot B$ is even. 

For the proof, it is equivalent to show that $D-x_0+2C$ is divisible by 2. We have
$$\iota(D-x_0+2C)=\iota((D+H)-x_0-(H-2C))=2(h-B-\tfrac{1}{2}A).$$
Now the class $h-B-\frac{1}{2}A$ is integral and the intersection
number $(h-B-\frac{1}{2}A)\cdot B=-\frac{1}{2}A\cdot B$ is also
integral since $A\cdot B$ is even. So $h-B-\frac12A$ lies in $\Lambda_B$ and
thus comes from a class in $H^2(X^\sigma_6,\bZ)^{\perp D}$, and $D-x_0+2C$ is
indeed divisible by 2.
\end{proof}

We now show that, for $\sigma$ very general in $\cD^{1,6,10}$, the projective
bundle $\pi\colon D\to S$ is precisely the Brauer--Severi variety for the
Brauer class $\beta$, which means that the Brauer class that we obtained
Hodge-theoretically actually comes from geometry. In particular, for $\sigma$
very general, the bundle $\pi\colon D\to S$ has non-trivial Brauer class.

\begin{proposition}
\label{brauerprop}
For~$\sigma$ very general in the divisor~$\cD^{1,6,10}$,
the projective bundle $\pi\colon D\to S$ is the Brauer--Severi variety for
the Brauer class $\beta$.
\end{proposition}
\begin{proof}
Recall that $K_{D/S}=K_D\otimes\pi^*K_S^{-1}=i^*D$, since $X^\sigma_6$ and $S$ both
have trivial canonical bundles. Denote by $\beta'$ the Brauer class defined by $\pi\colon D\to S$.  We may
suppose that $D=\bP(\cE)$ with $\cE$ a $\beta'$-twisted vector bundle on
$S$ of rank $2$. The relative $\cO(1)$ is a $(-\pi^*\beta')$-twisted line
bundle on $D$, and its square $\cO(2)$ is a non-twisted line bundle. Moreover,
$\cO(2)=\omega_{D/S}^\vee\otimes \pi^*\cL$ for some line bundle $\cL$ on $S$.
We may set $c_1(\cL)=kh$ for $k\in\bZ$, since for very general $\sigma$ in the
divisor, the K3 surface $S$ has Picard number 1. So the first Chern class
$c_1(\cO(2))$ is equal to $-i^*D+k(\pi^*h)$.

Consider a B-field lifting $B'$ of $\beta'$. We compute the
twisted Chern class
\begin{align*}2 c_1^{-\pi^*B'}(\cO(1))=c_1^{-2(\pi^*B')}(\cO(2))=
c_1(\cO(2))-2(\pi^*B')= -i^*D+k(\pi^*h)-2(\pi^*B').\end{align*}
The class $c_1^{-\pi^*B'}(\cO(1))$ is necessarily integral, so the last
term in the equality is divisible by~$2$. On the other hand,
$i^*(D-x_0)=i^*D-2\pi^*B$ is also divisible by~$2$ by
\autoref{D-x-divisible}. Thus the class $\frac{k}{2}h-B'-B$ is integral in
$H^2(S,\bZ)$, which shows that
$\beta'=\beta^{-1}=\beta$.
\end{proof}

Finally, we consider the moduli space of twisted sheaves on $S$,
following~\cite[Section~3]{ms}. We recall the definition of the twisted Mukai
lattice.  Consider the map
\begin{align*}\eta_B\colon H^2(S,\bC)&\to H^*(S,\bC)\\
    u&\mapsto(0,u,u\cdot B).\end{align*}
The {\em twisted Mukai lattice} $\tilde H(S,B,\bZ)$ is given by the usual Mukai
lattice $H^*(S,\bZ)\coloneqq H^0(S,\bZ)\oplus H^2(S,\bZ)\oplus H^4(S,\bZ)$,
equipped with the Hodge structure given by $\eta_B$, that is, its $(2,0)$-part
is the image of $H^{2,0}(S)$ under $\eta_B$. We recall that the {\em Mukai
pairing} is given by
\[
-\chi\big((r_1,c_1,s_1),(r_2,c_2,s_2)\big)\coloneqq c_1\cdot c_2-r_1s_2-r_2s_1.
\]

The algebraic part $\Pic(S,B)\subset\widetilde H(S,B,\bZ)$ is generated by $(2,2B,0)$, $(0,\Pic(S),0)$, and $(0,0,1)$. For a $\beta$-twisted sheaf
$E$, its {\em twisted Mukai vector} is defined as
\[v^B(E)\coloneqq \ch^B(E)\cdot\sqrt{\td(S)},\]
where~$\ch^B$ is the twisted Chern character as defined in \cite[Section 1]{hs}.
Let $M=M(S,v,B)$ be the moduli space of stable $\beta$-twisted sheaves $\cE$ on
$S$ with Mukai vector $v^B(\cE)=v\coloneqq(2,2B,0)$. Here $v^2=2$, so by the
general theory for moduli of twisted sheaves on K3 surfaces, $M$ is a
hyperkähler fourfold, with $H^2(M,\bZ)$ isometric to $v^\perp\subset \tilde
H(S,B,\bZ)$, the orthogonal of $v$ in the twisted Mukai lattice.
\begin{theorem}
For~$\sigma$ very general in the divisor~$\cD^{1,6,10}$,
there exists a Hodge isometry between $H^2(X^\sigma_6,\bZ)$ and $H^2(M,\bZ)$.
\end{theorem}
\begin{proof}
The lattice $H^2(M,\bZ)_\trans=\Pic(S,B)^\perp$ consists of elements $(a,u,b)$
satisfying
\begin{gather*}
-\chi\big((a,u,b),(2,2B,0)\big)=2u\cdot B-2b=0\\
-\chi\big((a,u,b),(0,h,0)\big)=u\cdot h=0\\
-\chi\big((a,u,b),(0,0,1)\big)=-a=0
\end{gather*}
which are precisely those in the image of $\Lambda_{B,\prim}$ by the map
$\eta_B\colon u\mapsto(0,u,u\cdot B)$. Since we have identified
$\Lambda_{B,\prim}$ with $H^2(X^\sigma_6,\bZ)_\trans$ via the isometry $\iota$,
we can thus define
\begin{align*}\phi\colon H^2(X^\sigma_6,\bZ)_\trans&\to
H^2(M,\bZ)_\trans\\x&\mapsto \eta_B(\iota(x))=
(0,\iota(x),\iota(x)\cdot B)\end{align*}
which is a Hodge isometry onto its image.

For the algebraic part, we may set $\phi(H)=(-2,2h-2B,0)$ and
$\phi(D)=(2,2B,1)$, and by using the intersection matrix of $\langle H,D\rangle$, $\phi$ is isometric when restricted to this space. It suffices now to extend~$\phi$ to the full lattice.
Consider the integral class $h-12B\in H^2(S,\bZ)_\prim$. Its
intersection number with $B$ is not integral, so it is not in
$\Lambda_{B,\prim}$. Thus $u_1\coloneqq 2h-24B$ is primitive in
$\Lambda_{B,\prim}$. We have $\eta_B(u_1)=(0,2h-24B,-11)$ so
$\eta_B(u_1)-\phi(H)+11\phi(D)=(24,0,0)$ is indeed divisible by $24$ in the
full lattice. This implies that the sum $H^2(M,\bZ)_\trans+(\bZ\phi(H)+\bZ\phi(D))$ is
of index $24$ in $H^2(M,\bZ)$.

Denote by $x_1\in H^2(X^\sigma_6,\bZ)_\trans$ the preimage $\iota^{-1}(u_1)$ of
$u_1$. Recall moreover that before Lemma \ref{D-x-divisible}, we have defined $ x_0\in H^2(X^\sigma_6,\bZ)^{\perp D} $ as the element satisfying $\iota(x_0)=2B$. Since $\iota(H+D)=2h$ and $\iota(x_0)=2B$, we have
\[x_1=\iota^{-1}(u_1)=\iota^{-1}(2h-24B)=(H+D)-12x_0.\]
By \autoref{D-x-divisible}, the class
\[x_1-H+11D=12(D-x_0)\]
is also divisible by $24$. So we may extend the map $\phi$ to the full
lattice by mapping $\frac{1}{2}(D-x_0)$ to $(1,0,0)$.
\end{proof}
Now that we have defined a Hodge isometry between the second cohomologies of
the hyperkähler fourfolds~$X_6^\sigma$ and~$M$, we may take advantage of the
powerful machinery of the Torelli theorem.
Notably, the Hodge structure on the second cohomology of a hyperkähler
manifold determines its birational model. Furthermore, the birational models
are parametrized by the chambers in the chamber decomposition of the movable
cone: each chamber corresponds to the ample cone of one birational model (see
for example~\cite[Theorem~3.21]{d}). Moreover, as we are in the 
type-$\KKK^{[2]}$ case, the chamber decomposition has an explicit numerical
description. Using these facts, we obtain the following result.
\begin{theorem}
\label{thmmodulitwistedsheaves}
A very general Debarre--Voisin fourfold $X^\sigma_6$ in the family $\cC_{24}$
is isomorphic to the moduli space~$M=M(S,v,B)$ of twisted sheaves with Mukai
vector $(2,2B,0)$ on the twisted K3 surface $(S,\beta)$.
\end{theorem}
\begin{proof}
%By the Torelli theorem, the existence of a Hodge isometry between second cohomologies shows that $X^\sigma_6$ and $M$ are birationally isomorphic. Moreover, the number of birational models is given by the number of chambers contained in the movable cone. 
The chambers in the movable cone of $X_6^\sigma$ and $M$ are cut out by hyperplanes of the type
$\kappa^\perp$, where $\kappa\in \Pic(X^\sigma_6)$ is a primitive class of
square $-10$ and divisibility $2$ (see~\cite[Theorem~3.16]{d}). We show that
for a very general $X^\sigma_6$ with Picard group generated by $H$ and $D$,
there is no such class $\kappa$: we may write $\kappa=aH+bD$ and get the
equation $22a^2+4ab-2b^2=-10$. By reduction modulo $5$, we verify that this
equation has no integral solutions, so no such $\kappa$ exists. Thus we may
conclude that there is only one birational model, and in particular
$X^\sigma_6\simeq M$.
\end{proof}

\appendix

\section{Macaulay2 code}

For completeness, we provide the various \Macaulay \cite{m2} code used throughout the paper, using
mainly the package {\tt Schubert2}.
\medskip

\noindent A1: {\bf Degree of \texorpdfstring{$\SL(V_{10})$}{SL(V10)}-invariant divisors.}
\label{m2:degree_SLV10_divisors} The following code computes the degree of the three $\SL(V_{10})$-invariant
divisors in $\bP\big(\bw3V_{10}^\vee\big)$. The degree $640$ of the
discriminant is a classical result by Lascoux~\cite{lascoux}.
\begin{lstlisting}[language=Macaulay2,firstline=45,lastline=55]
needsPackage "Schubert2";
-- divisor D_{3,3,10}: dual of Grassmannian Gr(3,10)
G = flagBundle{3,7}; (U,Q) = bundles G;
d1 = chern dual(exteriorPower_3 U+exteriorPower_2 U*Q);
-- divisor D^{1,6,10}
G = flagBundle{1,5,4}; (U1,U61,Q) = bundles G;
d2 = chern dual(U1*exteriorPower_2 U61+U1*U61*Q);
-- divisor D^{4,7,7}
G = flagBundle{4,3,3}; (U4,U74,Q) = bundles G;
d3 = chern dual(exteriorPower_3 U4+exteriorPower_2 U4*U74+U4*exteriorPower_2 U74);
<< (d1,d2,d3) / integral << endl; -- (640, 990, 5500) long time
\end{lstlisting}
\medskip

\noindent A2: {\bf
Remark~\ref{existence_trivect_two_planes}.} We verify that both of the two trivectors $\sigma$ define a smooth hyperplane section $X_3^\sigma$. We do this by checking the smoothness of $X_3^\sigma$ on each affine chart of the Grassmannian $\Gr(3,V_{10})$, and we base change to a small finite field to speed up the computation.
\begin{lstlisting}[language=Macaulay2]
-- the trivector is specified by a list of triples
checkSmooth = comps -> (
sigma = (u,v,w) -> sum(for idx in comps list det submatrix(u||v||w, idx));
-- check that the hyperplane section is smooth on each chart of Gr(3,10)
S = ZZ/2[g_0..g_20]; -- each chart has 21 coordinates
I3 = id_(S^3); M = genericMatrix(S,3,7);
-- generate the coordinates matrix for the chart ijk
chart = (i,j,k) ->
M_{0..i-1}|I3_{0}|M_{i..j-2}|I3_{1}|M_{j-1..k-3}|I3_{2}|M_{k-2..6};
isSmooth = true;
for ijk in subsets(10,3) do (Mijk = chart toSequence ijk;
if dim singularLocus ideal sigma(Mijk^{0},Mijk^{1},Mijk^{2}) >= 0
then isSmooth = false;);
isSmooth);
case1 = {(0,5,6),(0,3,7),(3,2,7),(0,4,7),(1,5,7),(2,5,7),(2,6,7),(1,0,8),
(2,1,8),(4,1,8),(5,0,8),(2,5,8),(1,6,8),(7,0,8),(2,1,9),(2,4,9),(3,4,9),
(0,5,9),(2,6,9),(8,2,9)};
case2 = {(4,5,6),(0,1,7),(0,2,7),(1,4,7),(5,0,7),(0,6,7),(1,6,7),(6,2,7),
(0,1,8),(1,3,8),(2,3,8),(4,1,8),(5,2,8),(0,3,9),(1,4,9),(1,6,9),(1,8,9)};
<< checkSmooth case2 << endl; -- true
\end{lstlisting}
\medskip

\noindent A3: {\bf Theorem~\ref{thm:HodgeConjX3}.} We verify the integral Hodge conjecture for $X^\sigma_3$ in degree 22 to 38, using Poincaré duality as explained in the proof.
\begin{lstlisting}[language=Macaulay2]
needsPackage "Schubert2";
G = flagBundle {3,7};
-- enumerate all the Schubert classes in codimension k
classes = k -> for p in partitions(k,7) list (
if #p <= 3 then (p = toList p; while #p < 3 do p = p|{0};
schubertCycle(p, G))
else continue);
-- compute the intersection matrices and reduce to Smith normal form
for k in 11..19 do print first smithNormalForm matrix (
for a in classes k list for b in classes (20-k) list
integral (a*b*schubertCycle({1,0,0}, G)));
\end{lstlisting}
\medskip

\noindent A4: {\bf \autoref{lemma:Lefschetz}.}
\label{m2:Lefschetz} The following code produces several algebraic classes on $X^\sigma_1$ realized
as relative Schubert classes. Note that the line where the
variety $X^\sigma_1$ is constructed takes quite a while to compute.
\begin{lstlisting}[language=Macaulay2,lastline=11]
needsPackage "Schubert2";
(U1,Q) = bundles projectiveBundle 9;
G = flagBundle({3,6},Q); (U41,Q) = bundles G;
time X = sectionZeroLocus dual(U1*exteriorPower_2 U41+U1*U41*Q); -- long time
h = chern_1 dual(U1*OO_X);
s = (X/G)^* schubertCycle_{2,0,0} G;
t = (X/G)^* schubertCycle_{4,0,0} G;
p = 1/3*(6*h^2-s)*h; -- the class pi of a Palatini 3-fold
assert((p*p, p*h^3) / integral == (4,7)); -- verify the intersection numbers
assert(t == 17/3*h^4-7*p*h);
assert(integral((6*h^2-s)*1/3*h^4) == 7);
\end{lstlisting}
\medskip

\noindent A5: {\bf\autoref{propD24}.}
\label{m2:c10} For a general $\sigma$ in $\cD^{1,6,10}$, the following code computes the
self-intersection number of a Grassmannian $\Gr(2,7)$ contained in
$X^\sigma_3$, by considering the top Chern class of the normal bundle using the
two normal sequences.
\begin{lstlisting}[language=Macaulay2,firstline=13,lastline=16]
needsPackage "Schubert2";
(U,Q) = bundles flagBundle{2,5};
N = dual(U+1)*(Q+2)-det Q-dual U*Q;
<< "c_10(N)=" << integral chern_10 N << endl; -- c_10(N)=2
\end{lstlisting}
\medskip

\noindent A6: {\bf \autoref{remK3S}.}
\label{m2:K3degree} The following code verifies that the two tautological bundles on the K3 surface
$S$ satisfy $\det(\cU_{4/1}^\vee)\simeq\cO_S(3)$ and
$\det(\cU_{8/4})\simeq\cO_S(2)$.
\begin{lstlisting}[language=Macaulay2,firstline=18,lastline=25]
needsPackage "Schubert2";
(U,Q) = bundles flagBundle{2,2}; -- first choose V8/V6 in V10/V6
(U3,U4) = bundles flagBundle({3,4},U+5); -- then choose V4/V1 in V8/V1
S = sectionZeroLocus dual(det U+det U3+exteriorPower_2 U3*U4);
h = chern_1(dual U*OO_S);
U41 = U3*OO_S;
U84 = U4*OO_S;
assert(chern_1 dual U41==3*h and chern_1 U84==2*h); -- verify the Chern classes
\end{lstlisting}
\medskip

\noindent A7: {\bf \autoref{lemma:intersectionDH}.}
\label{m2:intersectionDH} The following code computes the intersection numbers between the classes $H$
and $D$ on $X^\sigma_6$.
\begin{lstlisting}[language=Macaulay2,firstline=27,lastline=36]
needsPackage "Schubert2";
(U1,Q1) = bundles flagBundle{3,2}; -- first choose U4/V1 in V6/V1
(U2,Q2) = bundles flagBundle({2,4},Q1+4); -- then choose U6/U4 in V10/U4
D = sectionZeroLocus dual((1+U1)*det U2+det U1+exteriorPower_2 U1*U2);
h = chern_1(dual(1+U1+U2)*OO_D);
d = chern_1 cotangentBundle D;
(U,Q) = bundles flagBundle{6,4};
X = sectionZeroLocus dual exteriorPower_3 U;
h' = chern_1 OO_X(1);
<< (d^3,d^2*h,d*h^2,h^3,h'^4) / integral << endl; -- (12, -12, -36, 132, 1452)
\end{lstlisting}

\section{Fano variety of lines of \texorpdfstring{$X^\sigma_1$}{X1}
and the variety \texorpdfstring{$X^\sigma_7$}{X7}}
\label{sec:FanoX7}

The geometry of the universal Palatini variety $I^\sigma_{1,6}$ defined
in~\autoref{eq:I16} can be further studied by considering the Fano variety
$F(X^\sigma_1)$ of lines of $X^\sigma_1$ and the degeneracy locus $X^\sigma_7$.
In this section, we will assume that $\sigma$ is general in $\bw3V_{10}^\vee$.

We first recall the definition of $X^\sigma_7$. Consider the space $\bw3
V_7^\vee$ of trivectors, a complex vector space of
dimension 35. The orbit closures for the action of $\GL(V_7)$ have long been
classified (see~\cite[Section~35.3]{gurevich}). We list the smallest ones:
\[\begin{tikzcd}
O_{20}\coloneqq\setmid{\sigma\in\bw3V_7^\vee}{\rank\sigma\le5} \quad\text{and}\quad
O_{13}\coloneqq\setmid{\sigma\in\bw3V_7^\vee}{\rank\sigma\le3},
\end{tikzcd}\]
which are of respective dimensions 20 and 13 and thus of codimensions 15 and
22 in $\bw3V_7^\vee$; if a basis of $V_7^\vee$ is denoted by $e_1,\cdots,e_7$, a representative for $O_{20}$ is given by $ (e_1\wedge e_2+e_3\wedge e_4)\wedge e_5$ and a representative for $O_{13}$ is given by $ e_1\wedge e_2 \wedge e_3$. The two conditions can also be expressed equivalently as
the existence of a subspace~$V_k\subset V_7$ such
that~$\sigma(V_k,V_7,V_7)=0$, with~$k=2$ and~$4$ in the two cases respectively.
The variety $X^\sigma_7$ can be defined as the orbital degeneracy locus
\begin{align*}
X^\sigma_7 &\coloneqq\setmid{[V_7]\in\Gr(7,V_{10})}{\rank(\sigma|_{V_7})\le5} \\
&=\setmid{[V_7]\in\Gr(7,V_{10})}{
\sigma|_{V_7}\in O_{20}\subset \bw3 V_7^\vee}=D_{O_{20}}(\sigma).
\end{align*}

\begin{proposition}
\label{propsingD477}
For generic $\sigma$ the variety $X_7^\sigma$ is smooth. 

Let $[\sigma]$ be a general point in $ \cD^{4,7,7}$ such that
$\sigma(V_4,V_7,V_7)=0$; then $[V_7]\in \Sing(X^\sigma_7)$.
\end{proposition}

\begin{proof}
By the theory of orbital degeneracy loci~\cite{benedetti}, since $\sigma$ is
general, $X^\sigma_7$ has codimension 15 and its singular locus is equal to
\[D_{\Sing(O_{20})}(\sigma)=D_{O_{13}}(\sigma)=\setmid{[V_7]\in\Gr(7,V_{10})}{
\sigma|_{V_7}\in O_{13}\subset \bw3 V_7^\vee}.\]
However, this locus is empty for dimensional reasons, because
\[
\codim_{\Gr(7,V_{10})}(D_{O_{13}}(\sigma))=\codim_{\bw3 \bC^7}(O_{13})= 22>
\dim\Gr(7,V_{10}).
\]
So for $\sigma$ general, $X^\sigma_7$ is a smooth variety and each point
$[V_7]\in X_7^\sigma$ satisfies the condition
$\rank(\sigma|_{V_7})=5$. We now show that $X^\sigma_7$ is singular when
$[\sigma]$ is in the divisor $\cD^{4,7,7}$.

We already mentioned that the flag $V_4\subset V_7$ such that $\sigma(V_4,V_7,V_7)=0$ is unique for a general trivector
$\sigma$ in $\cD^{4,7,7}$. For this unique $V_7$, we see that $\rank(\sigma|_{V_7})=3$, so $[V_7]$
belongs to
\[
D_{O_{13}}(\sigma)=\setmid{[V_7]\in\Gr(7,V_{10})}{
\sigma|_{V_7}\in O_{13}\subset \bw3 V_7^\vee}.
\]
As $O_{13}$ is the singular locus of $O_{20}$, we know by the theory of
orbital degeneracy loci that $D_{O_{13}}(\sigma)$ is contained in $\Sing
(D_{O_{20}}(\sigma))$ ($D_{O_{20}}(\sigma)$ is of the expected dimension).
Therefore, $[V_7]$ is a singular point of the variety
$D_{O_{20}}(\sigma)=X^\sigma_7$. Note that, since $\sigma$ is not general in
$\cM$, one cannot a priori affirm that $\Sing
(D_{O_{20}}(\sigma))=D_{O_{13}}(\sigma)$.
\end{proof}

\begin{remark}
\label{remsingX7}
It is in principle possible to study the exact locus of trivectors $\sigma$
for which $X^\sigma_7$ becomes singular, as we did for $X^\sigma_2$, $X^\sigma_3$,
and $X^\sigma_6$ in \autoref{propsingD33}, and for $X^\sigma_1$
in \autoref{propsingD16}. However, due to the large number of cases to
study for $X^\sigma_7$, we decided not to include a more general statement.
\end{remark}

When $[\sigma]\notin \cD^{4,7,7}$, for each $[V_7]\in X^\sigma_7$, the
property $\rank(\sigma|_{V_7})=5$ means that there is a unique
$V_2\subset V_7$ such that $\sigma(V_2,V_7,V_7)=0$. Since $V_7$ is an isotropic subspace for $\sigma(v,V_{10},V_{10})$ for each $v\in V_2$, the rank of $\sigma(v,V_{10},V_{10})$ is at most $6$; hence each
$V_1\subset V_2$ is in $X^\sigma_1$, so we get a line $\bP(V_2)$ contained in
$X^\sigma_1$. This provides an injective morphism from $X^\sigma_7$ to the Fano variety of
lines $F(X^\sigma_1)$. For~$\sigma$ general, this morphism is in fact an
isomorphism.

\begin{theorem}
\label{thmfanoXone}
Let $\sigma$ be a general trivector in the moduli space.
% such that $[\sigma]\notin \cD^{3,3,10} \cup \cD^{1,6,10} \cup \cD^{4,7,7}$.
Then the Fano variety of lines $F(X^\sigma_1)$ is
isomorphic to $X^\sigma_7$. Moreover there exists a morphism $q_1:X_7^\sigma
\to X_6^\sigma $ which is a fibration in cubic surfaces.
%The projection map $\tilde q\colon \tilde I^\sigma_{1,6}\to X^\sigma_6$ from \autoref{eq:I16} factors as
%where $q_1$ is a $\bP^1$-bundle and $q_2$ is a fibration in cubic surfaces.
%identifying $I^\sigma_{1,6}$ with the incidence variety between $X^\sigma_1$ and $F(X^\sigma_1)$.
\end{theorem}
\begin{proof}
Since~$\sigma$ is supposed to be general, we may suppose that~$\sigma$ does not lie in
the union of the three special divisors $\cD^{3,3,10} \cup \cD^{1,6,10} \cup
\cD^{4,7,7}$. In this case, we already obtained an injective morphism from
$X^\sigma_7$ to $F(X^\sigma_1)$ (because for a point $[V_7]\in X_7^\sigma$, $[V_2]$ satisfying $\sigma(V_2,V_7,V_7)=0$ is uniquely defined by $[V_7]$). Moreover, by a Bertini-type argument, we may assume that~$X_7^\sigma$ is smooth.

Using \cite[Appendix A]{kmm} and the notation therein, note that $$X_1^\sigma=\bP(V_{10})\times_{\bP(\bw2 V_{10}^\vee)} \Pf_2(V_{10}^\vee) =\bP(V_{10})\times_{\bP(\bw2 V_{10}^\vee)}(\Pf_2(V_{10}^\vee)\Pf_3(V_{10}^\vee)),$$ where the map $\bP(V_{10})\to \bP(\bw2 V_{10}^\vee)$ is given by $\sigma$, and the second equality follows from $\sigma \notin \cD^{1,6,10}$. Therefore, $$F(X_1^\sigma) = \Gr(2, V_{10})\times_{\Gr(2,\bw2 V_{10}^\vee)}F(\Pf_2(V_{10}^\vee)).$$ Now applying the argument of \cite[Proposition A.6]{kmm} it is easy to see that $$F(X_1^\sigma) = \Gr(2, V_{10})\times_{\Gr(2,\bw2 V_{10}^\vee)}\tilde{F_2},$$ in particular one obtains a morphism $F(X_1^\sigma)\to \tilde{F_2}\to \Gr(7, V_{10})$. It remains to note that the fiber of this map over a point $[V_7]$ parameterizes all $[V_2]$ such that $V_2\subset\ker(\sigma|_{V_7})$, which is a single point if $\sigma\notin \cD^{4,7,7}$. Therefore the morphism $X^\sigma_7\to F(X_1^\sigma)$ that we described is an
isomorphism of smooth varieties.

Let us now construct the fibration $q_1\colon X^\sigma_7\to X^\sigma_6$. For
$[V_7]\in X^\sigma_7$, let $[\bP(V_2)]\in F(X^\sigma_1)$ be the line it
determines, so that $V_2\subset V_7$ with $\sigma(V_2,V_7,V_7)=0$. Over
$\Gr(4,V_7/V_2)$, the trivector $\sigma$ gives a section of $\bw3\cU_4^\vee$,
whose top Chern class is 1. We have therefore a $V_6$ such that
$V_2\subset V_6\subset V_7$ and $\sigma|_{V_6}=0$; it is easy to check that
$V_6$ is unique provided that there exists no $V_4$ such that
$\sigma(V_4,V_7,V_7)=0$.  This defines a morphism $q_1\colon X^\sigma_7\to
X^\sigma_6$. We can see that under this map, the
preimage of a point $[V_6]$ is the Pfaffian cubic surface defined in
$\bP(V_{10}/V_6)$, that is, the set of $[V_7]$ where
$\sigma(V_7/V_6,-,-)|_{V_6}$ degenerates. In fact, using the notation from the
proof of \autoref{lemma:palatini}, this is the intersection of
the~$3$-space~$\Delta=\bP(V_{10}/V_6)$ with the cubic
hypersurface~$\Gr(2,V_6)^*\subset\bP\big(\bw2V_6^\vee\big)$.
Since~$[\sigma]\notin\cD^{4,7,7}$, we see that~$\Delta$ cannot be entirely
contained in~$\Gr(2,V_6)^*$, so we always get a cubic surface.
\end{proof}

\begin{remark}
We still assume that the trivector~$\sigma$ is general and make some remarks on
the geometry of the incidence variety~$I^\sigma$
between~$X_1^\sigma$ and its Fano variety of lines~$F(X_1^\sigma)$.
Consider the diagram
\[\begin{tikzcd}
&I^\sigma\coloneqq\setmid{(x,l)}{x\in l\subset
X^\sigma_1}\ar[ld]\ar[rd,"\bP^1\text{-bundle}"]\\
X_1^\sigma&&F(X_1^\sigma)\simeq X_7^\sigma\ar[d,"q_1"]\\
&&X_6^\sigma
\end{tikzcd}\]
%
% It is a $\bP^1$-bundle over $F(X^\sigma_1)$, while the
% latter is isomorphic to $X^\sigma_7$ by the previous theorem.
% By composing the fibration~$I^\sigma\to F(X_1^\sigma)$ with $q_1\colon
% X_7^\sigma\to X_6^\sigma$, we get a
The fibers of the fibration $I^\sigma\to X_6^\sigma$ are
isomorphic to Palatini threefolds. Moreover, the
induced projection $r\colon I^\sigma\to I_{1,6}^\sigma$ is a birational
morphism. In fact, consider the following subvariety
\[
Z \coloneqq\setmid{[V_1\subset V_6]\in I^\sigma_{1,6}}
{\sigma(V_1,K_4,V_{10})=0,[V_6]\in X_6^\sigma,K_4\subset V_6}
\]
from the proof of
\autoref{propprojection16}.
% In other words, it is the locus where the kernel~$K_4$ of~$V_1$ is contained
% in~$V_6$.
We saw that it is precisely the singular locus of~$I^\sigma_{1,6}$, it
is itself smooth of dimension~$3$, and the projection $s\colon \tilde
I_{1,6}^\sigma \to I_{1,6}^\sigma$ is a small contraction with
$\bP^2$-fibers over~$Z$. We now show that~$r\colon I^\sigma\to
I^\sigma_{1,6}$ gives a second resolution which is another small contraction
with $\bP^1$-fibers over~$Z$.

First, note that the fiber of~$r$ over a pair~$[V_1\subset V_6]\in I^\sigma_{1,6}$
is the set of~$V_7\supset V_6$ such that~$\sigma(V_1,V_7,V_7)=0$: such~$V_7$
defines a degenerate skew-symmetric form on~$V_6$ so has a kernel~$V_2$
containing~$V_1$ such that~$\sigma(V_2,V_7,V_7)=0$.
If a pair~$[V_1\subset V_6]$ does not lie in~$Z$, that is, if the
kernel~$K_4$ of~$V_1$ is not contained in~$V_6$, then~$K_4+V_6$ provides the
only such~$7$-dimensional subspace, so $r^{-1}$ consists of the single
point~$[K_4+V_6]$. On the other hand, for a pair~$[V_1\subset V_6]$ lying
in~$Z$,
we get a unique~$8$-dimensional space $V_8\supset V_6$ such that
$\sigma(V_1,V_6,V_8)=0$. Each~$[V_7]\in\bP(V_8/V_6)$ satisfies the vanishing
condition~$\sigma(V_1,V_7,V_7)=0$, so the fiber of~$r$ over~$[V_1\subset V_6]$
is the line~$\bP(V_8/V_6)$.

% thus giving a second resolution of singularities of
% $I_{1,6}^\sigma$.
% However, the varieties $\tilde I_{1,6}^\sigma$ and $I^\sigma$ are not
% isomorphic. Indeed, let us denote by , and let us consider the variety $Z$ which has been defined in
% the proof of \autoref{propprojection16}; if $x$ is a general point
% inside $Z$, we have seen that $s^{-1}(x)\simeq \bP^2$, while $r^{-1}(x)\simeq
% \bP^1$ (thus giving that $s$ and $r$ are both small contractions).

% In order to obtain the isomorphism $r^{-1}(x)\simeq \bP^1$, suppose that
% $\sigma$ is general and $x\in Z$ is given by a flag $V_1\subset V_6 \subset
% V_{10}$ such that $\sigma(V_6,V_6,V_6)=0$ and $\sigma(V_1,K_4,V_{10})=0$, where
% $K_4$ is a four dimensional space contained inside $V_6$; then
% \[
% r^{-1}(x)=\{[V_2\subset V_7]\mid V_1\subset V_2\subset V_6\subset V_7 \mbox{ and } \sigma(V_2,V_7,V_7)=0\}.
% \]
% Since there exists a unique eight dimensional space $V_8$ such that
% $\sigma(V_1,V_6,V_8)=0$, we get that $V_7\subset V_8$. Therefore the possible
% flags $V_2\subset V_7$ in the fiber $r^{-1}(x)$ are given by the points in the
% subvariety of $\bP(V_6/V_1)\times \bP(V_8/V_6)$ parametrizing couples
% $[V_2/V_1, V_7/V_6]$ where $\sigma(V_2/V_1,V_6/V_2,V_7/V_6)=0$. Via a
% computation with \Macaulay it is easy to check that $r^{-1}(x)\simeq \bP^1$.

Moreover, one can also compute that $K_{\tilde
I_{1,6}^\sigma}\mid_{s^{-1}(x)}\cong \cO_{\bP^2}(-1)$ and that
$K_{I^\sigma}\mid_{r^{-1}(x)}\cong \cO_{\bP^1}(1)$; therefore we obtain the
following flip
\[\begin{tikzcd}
\tilde I^\sigma_{1,6} \ar[rdd, "s"'] \ar[rd,"\bP^2" description, no head]
 & \ar[l,dashed] \text{flip} \ar[r,dashed]
 & I^\sigma \ar[ldd, "r"]\ar[dd, "\bP^1\text{-bundle}"] \ar[ld,"\bP^1" description, no head]\\
 & Z\ar[d,hook] &\\
 \qquad\qquad& I^\sigma_{1,6}\ar[ld, "p"]\ar[rd, "q"']
 & F(X_1^\sigma)\simeq X_7^\sigma\ar[d, "q_1"]\\
 X_1^\sigma && X_6^\sigma.
\end{tikzcd}\]
\end{remark}

\section{A special $\GL(\bC^7)$-orbit}
\label{special_orbit_gl7}

Let $W_7$ be a complex vector space of dimension $7$.  We begin by recalling
some properties on $\GL(W_7)$-orbit closures inside $\bw3 W_7$ that we
need in Section \ref{sec:D24} (one can refer to \cite{WeymanE7}).
Let $Y\subset \bw3 W_7^\vee$ be the unique $\GL(W_7)$-invariant
hypersurface. It can also be characterized as the affine cone over the
projective dual variety $\Gr(3,W_7)^*$ embedded in $\bP(\bw3W_7^\vee)$, which
is a hypersurface of degree~$7$. In other words, the polynomial $f$ defining
$Y$ lives inside $\Sym^7(\bw3W_7^\vee)^\vee$, which is usually referred to as the
{\em discriminant} or the {\em hyperdeterminant}. Equivalently, $\bC f$ is
the unique one-dimensional $\GL(W_7)$-subrepresentation of $\Sym^7\bw3W_7$.
Since all one-dimensional representations of $\GL(W_7)$ are of the form
$\det(W_7)^{\otimes i}$ for $i\in \bZ$, weight invariance with respect to the torus
$\bC^*\Id \subset \GL(W_7)$ implies that we have $\bC f\simeq\det(W_7)^{\otimes
3}\subset \Sym^7\bw3W_7$. This also means that we can canonically define the
discriminant $\disc y$ of each $y\in \bw3 W_7^\vee$ as an element of
$\det (W_7^\vee)^{\otimes 3}$.
\begin{remark}
\label{remorbclosureY1}
The orbit closure $Y\subset \bw3 W_7^\vee$ admits a nice desingularization (a
{\em Kempf collapsing}, see~\cite{benedetti}). Consider the vector bundle
\[
\cW\coloneqq\bw2\cQ_4^\vee \wedge W_7^\vee
\]
over $\Gr(3,W_7)$,
where $\cQ_4$ is the tautological quotient bundle. The bundle $\cW$ is a
subbundle of the trivial bundle $(\bw3W_7^\vee)\otimes \cO_{\Gr(3,W_7)}$, whose
total space is just $\bw3W_7^\vee\times \Gr(3,W_7)$. By projecting onto the
first factor, we obtain a morphism $\Tot(\cW)\to \bw3 W_7^\vee$; the image of
this morphism is exactly the discriminant hypersurface $Y$, and it is
birational onto its image~\cite{WeymanE7}.
This should be interpreted in the following way: for any general element $y\in
Y$, there exists a unique subspace $W_3\subset W_7$ such that $y\in \bw2
(W_7/W_3)^\vee \wedge W_7^\vee$ or equivalently, $y(W_3,W_3,W_{7})=0$. One can check that this implies that $[W_3]$ is a singular point of the
hyperplane section defined by $y\in \bw3 W_7^\vee$ inside $\Gr(3,W_7)$, thus
providing the description of $\bP(Y)\subset \bP(\bw3 W_7^\vee)$ as the projective
dual of $\Gr(3,W_7)$.
\end{remark}

\section{Review: cubic fourfolds containing a plane} 
\label{app_cubic_fourfolds_containing}
Let us recall what happens for cubic fourfolds containing a plane, a case originally considered by
Voisin in her thesis~\cite{voisinthesis}, with later studies on their derived aspects by Kuznetsov~\cite{kuz}, and moduli aspects by Macrì--Stellari~\cite{ms} and Ouchi~\cite{ouchi}. 

Let $X$ be a very general cubic fourfold in $\bP(V_6)$ that contains a plane
$\bP(V_3)$ for some $V_3\subset V_6$.
As shown in~\cite{voisinthesis}, the blow up $\Bl_{\bP(V_3)}X$ projects onto
the plane $\bP^2=\bP(V_6/V_3)$ and the fibers are quadric surfaces which are
generically smooth. The discriminant locus, that is, the locus where the
quadrics are singular, is a sextic curve in~$\bP^2$. Let $S$ be the variety
parametrizing rulings of lines in these fibers: as the fibers are quadric
surfaces, the projection $S\to \bP^2$ is a generically $2$-to-$1$ morphism,
ramified along the discriminant curve, and $S$ is therefore a K3 surface of
degree 2.

The variety $F\subset \Gr(2,V_6)$ of lines contained in $X$ is a hyperkähler
fourfold of $\KKK^{[2]}$-type by~\cite{bd}.
%In particular, they proved
%that the incidence variety $X\overset p\longleftarrow I\overset q\longrightarrow F$
%provides a Hodge isometry
%%
%\[q_*p^*\colon H^4(X,\bZ)_\prim\simto H^2(F,\bZ)_\prim(-1).\]
%%
The lines in the fibers of $\Bl_{\bP(V_3)}X\to \bP^2$ form a uniruled divisor
$D$ in $F$. Alternatively, it can also be defined as the closure of the set of lines in $X$
that intersect $\bP(V_3)$ at one point. Clearly, $D$ admits a $\bP^1$-fibration
over~$S$. In~\cite{voisinthesis}, it was shown that the transcendental part
$(H^2(F,\bZ)_\trans,\cap)$ (of discriminant $-8$ and Hodge type $(1,20,1)$) embeds isometrically as a
sublattice of index two into the primitive cohomology $(H^2(S,\bZ)_\prim,\cap)$ (of
discriminant $-2$). This sublattice is closely related to the Brauer class $\beta$
induced by the $\bP^1$-fibration $D\to S$, so it should be considered as the
``primitive cohomology'' of the twisted K3 surface $(S,\beta)$
(see~\cite{vangeemen}). For a general $X$ containing a plane,
the class $\beta$ is non-trivial and is related to rationality questions
(see~\cite{hassett}).  Finally, it was proved in~\cite{ms} that
for a general~$X$ containing a plane, the hyperkähler variety $F$ can be
recovered (birationally) as a moduli space of $\beta$-twisted sheaves on $S$.

\end{document}